\documentclass[A4]{amsart}

\usepackage[square,sort,comma,numbers]{natbib}
\usepackage{amsmath, amssymb, amstext, amsfonts, textcomp, amsxtra, amsbsy, amsgen, amsopn, amscd, mathrsfs, amsthm, latexsym, array}
\usepackage[textwidth=16cm,textheight=22cm,centering]{geometry}

\usepackage{color}

\usepackage[all]{xy}

\usepackage{bbm, dsfont}

\newtheorem{theorem}             {Theorem}  [section]
\newtheorem{definition} [theorem] {Definition}
\newtheorem{lemma}      [theorem]{Lemma}
\newtheorem{corollary}  [theorem]{Corollary}
\newtheorem{proposition}[theorem]{Proposition}
\newtheorem{remark} [theorem] {Remark}

\numberwithin{equation}{section} \everymath{\displaystyle}

\newcommand{\Cont}{{\rm C}}
\newcommand{\Aut}{\mathcal{A}}

\newcommand{\Sch}{\mathcal{S}}

\newcommand{\intL}{{\rm L}}
\newcommand{\Ht}{{\rm Ht}}

\newcommand{\Nr}{{\rm Nr}}
\newcommand{\Tr}{{\rm Tr}}

\newcommand{\gp}[1]{\mathbf{#1}}
\newcommand{\GL}{{\rm GL}}
\newcommand{\PGL}{{\rm PGL}}
\newcommand{\SL}{{\rm SL}}
\newcommand{\SO}{{\rm SO}}
\newcommand{\SU}{{\rm SU}}
\newcommand{\diag}{{\rm diag}}

\newcommand{\ag}[1]{\mathbb{#1}}
\newcommand{\Mat}{{\rm M}}

\newcommand{\lcd}{{\rm lcd}}

\newcommand{\E}{\mathbf{E}}
\newcommand{\F}{\mathbf{F}}

\newcommand{\vo}{\mathfrak{o}}

\newcommand{\vp}{\mathfrak{p}}
\newcommand{\idlJ}{\mathfrak{J}}
\newcommand{\Dis}{{\rm D}}
\newcommand{\Dif}{\mathfrak{D}}
\newcommand{\gCl}{{\rm Cl}}

\newcommand{\ProjP}{{\rm P}}
\newcommand{\norm}[1][\cdot]{\lvert #1 \rvert}
\newcommand{\extnorm}[1]{\left\lvert #1 \right\rvert}
\newcommand{\Norm}[1][\cdot]{\lVert #1 \rVert}
\newcommand{\extNorm}[1]{\left\lVert #1 \right\rVert}
\newcommand{\Pairing}[2]{\langle #1, #2 \rangle}

\newcommand{\Four}[2][]{\mathfrak{F}_{#1}(#2)}

\newcommand{\Mellin}[2][]{\mathfrak{M}_{#1}(#2)}

\newcommand{\rpR}{{\rm R}}
\newcommand{\Res}{{\rm Res}}
\newcommand{\Ind}{{\rm Ind}}
\newcommand{\Bas}{\mathcal{B}}

\newcommand{\Intw}{\mathcal{M}}

\newcommand{\cusp}{{\rm cusp}}
\newcommand{\Eis}{{\rm Eis}}
\newcommand{\Sp}{{\rm sp}}

\newcommand{\Cond}{\mathbf{C}}
\newcommand{\cond}{\mathfrak{c}}
\newcommand{\fin}{{\rm fin}}
\newcommand{\eis}{{\rm E}}
\newcommand{\eisCst}{{\rm E}_{\gp{N}}}
\newcommand{\Reis}{\mathcal{E}}
\newcommand{\reg}{{\rm reg}}

\newcommand{\freg}{{\rm fr}}
\newcommand{\FundD}{\mathcal{D}}
\newcommand{\Ex}{\mathcal{E}{\rm x}}

\newcommand{\latL}{\mathcal{L}}
\newcommand{\latPlg}{\mathcal{P}}

\newcommand{\fsB}{{\rm B}}
\newcommand{\fsH}{{\rm H}}

\newcommand{\Vol}{{\rm Vol}}
\makeatletter

\newcommand{\Rmnum}[1]{\expandafter\@slowromancap\romannumeral #1@}
\makeatother

\title{Deducing Selberg Trace Formula via Rankin-Selberg Method for $\GL_2$}
\author{Han Wu}

\begin{document}

	\begin{abstract}
		In the 80s, Zagier and Jacquet-Zagier tried to derive the Selberg trace formula by applying the Rankin-Selberg method to the automorphic kernel function. Their derivation was incomplete due to a puzzle of the computation of a residue. We solve this puzzle and complete the derivation. The main input is an extension of the theory of regularized integrals invented by Zagier, which is of independent interest.
	\end{abstract}
	
	\maketitle
	
	\tableofcontents

\section{Introduction}

	\subsection{Jacquet-Zagier's Approach to Trace Formulae and Their Puzzle}
	
	Although the trace formulae were developed in great generality by Arthur, the truncation process on the automorphic kernel function is somewhat too complicated to lead to explicit form in many situations. Attempting to remedy this, Jacquet-Zagier \cite{JZ87} initiated a new approach by introducing the Rankin-Selberg method into the treatment of the kernel function for $\GL_2$. One of their two main goals is to derive the Selberg trace formula avoiding the recourse to Arthur's truncation.
	
	Precisely, let $\F$ be a number field, $\ag{A}$ be its ring of adeles. Let $\omega$ be a Hecke character. Let $\Psi \in \Cont_c^{\infty}(\GL_2(\ag{A}), \omega^{-1})$ be a smooth function with compact support modulo the center and transforms as $\omega^{-1}$ under the action of the center. Assume $\Psi$ is $\gp{K}$-finite on both sides for simplicity of discussion, where $\gp{K}$ is the standard maximal compact subgroup of $\GL_2(\ag{A})$. If $\rpR_0$ denotes the right translation of $\GL_2(\ag{A})$ on the space $\intL_0^2(\GL_2(\F) \backslash \GL_2(\ag{A}), \omega)$ of cusp forms with central character $\omega$, the operator $\rpR_0(\Psi)$ can be represented by a kernel function $K_0(x,y)$ which has an explicit formula
	$$ K_0(x,y) = K(x,y) - K_{\Eis}(x,y) - K_{\Sp}(x,y), $$
	$$ K(x,y) := \sideset{}{_{\gamma \in \PGL_2(\F)}} \sum \Psi(x^{-1} \gamma y). $$
Under our assumption, $K_0(x,y)$ is of rapid decay in $(x,y) \in [\PGL_2]^2$. Hence the following integral with parameter $s \in \ag{C}$
	$$ I(s) := \int_{[\PGL_2]} K_0(x,x) \eis(s,x) dx $$
is uniformly convergent outside the poles of $\eis(s,x)$, an Eisenstein series with trivial central character\footnote{The setting in \cite{JZ87} deals with arbitrary central character, but the \emph{Jacquet-Zagier puzzle} we are discussing concerns only the case of trivial central character.}, and defines an meromorphic function with possible poles given by those of $\eis(s,x)$. For simplicity of notations, we have written\footnote{For our convention, $\PGL_2$ does not mean the quotient of $\GL_2$ by its center in the category of algebraic groups, but $\PGL_2(k)=k^{\times} \backslash \GL_2(k)$ in the group theoretic sense for any field $k$. The former is in general a larger group.}
	$$ [\PGL_2] = \PGL_2(\F) \backslash \PGL_2(\ag{A}) = \gp{Z}(\ag{A}) \GL_2(\F) \backslash \GL_2(\ag{A}). $$
In particular, if $\eis(s,x)$ is the spherical Eisenstein series, normalized to admit a constant residue equal to $1$ at\footnote{Our convention is to take $s$ as the spectral parameter for the induced representation $\pi_s = \Ind_{\gp{B}(\ag{A})}^{\GL_2(\ag{A})}(\norm_{\ag{A}}^s, \norm_{\ag{A}}^{-s})$. Consequently, our Eisenstein series $\eis(s,x)$ should be regarded as $E(x,s+1/2)$ in \cite{JZ87}, and has the center of symmetry $s=0$.} $s=1/2$, then the computation of the residue
	$$ \Res_{s=1/2} I(s) = \int_{[\PGL_2]} K_0(x,x) dx = \Tr(\rpR_0(\Psi)) $$
should recover the Selberg trace formula. This was achieved in some special cases in the non-adelic setting, for example in the paper by Zagier \cite{Za81} for $\F=\ag{Q}$ and $\gp{K}$-invariant case, by Szmidt \cite{Sz83} for $\F=\ag{Q}[i]$ and $\gp{K}$-invariant case. But in the adelic setting, the derivation of the Selberg trace formula via this method in Jacquet-Zagier's work \cite{JZ87} was incomplete, due to the complication of the computation of a certain residue \cite[p.3]{JZ87}, to which we shall refer as the \emph{Jacquet-Zagier puzzle}. We will get back to a detailed recall on the puzzle in a later paragraph. Our first main result is the completion of this approach to the Selberg trace formula. The crucial input is an extension of the theory of regularized integrals invented by Zagier \cite{Za82}.
\begin{theorem}
	Decomposing $I(s)$ as in the main theorem of \cite{JZ87}, the residue at $s=1/2$ of $I(s)$ corresponds term-by-term to the (adelic) Selberg trace formula \cite[Theorem (6.33)]{GJ79}.
 \label{Main1}
\end{theorem}

	We will form the relevant spherical Eisenstein series $\eis(s,x)$ from the standard flat section, while Jacquet-Zagier used another one $\eis(s,\Phi)$ formed from the standard \emph{Godement sections}. The relation of the two Eisenstein series will be recalled in (\ref{SecRel}). We write the Jacquet-Zagier's version of $I(s)$ as
	$$ \tilde{I}(s) := \int_{[\PGL_2]} K_0(x,x) \eis(s,\Phi)(x) dx. $$
	More generally, we give every notation from \cite{JZ87} a tilde. Regrouping properly the summands in $K_0(x,x)$, especially the conjugacy classes in the summation of $K(x,x)$, leads to a decomposition
	$$ \tilde{I}(s) = \sideset{}{_{\E}} \sum \tilde{I}_{\E}(s) +  \tilde{I}_1(s) + \tilde{I}_2(s) + \tilde{I}_{\infty}'(s) + \tilde{I}_{\infty}''(s) + \tilde{I}_{\infty}'''(s). $$
	Here $\tilde{I}_1(s)$, $\tilde{I}_2(s)$ and $ \tilde{I}_{\infty}'''(s)$ have a pole at $s=1/2$ of order $2$, and the order $2$ parts of them are expected to cancel, while the order $1$ part of all the terms are expected to give the geometric side of the Selberg trace formula. With certain functionals $B(\Phi), C(\Phi)$ on $\Sch(\ag{A}^2)$ and certain functionals $A(\Psi), T_1(\Psi), T_2(\Psi)$ on $\Cont_c^{\infty}(\GL_2(\ag{A}), \omega^{-1})$, that we will recall and compute explicitly later, the principal part of $\tilde{I}_1(s) + \tilde{I}_2(s)$ is determined in \cite[p.43 \& 44]{JZ87} as\footnote{Various computation in \cite{JZ87} concerning the Fourier analysis on $\F^{\times} \backslash \ag{A}^{(1)}$ does not take into account the volume. We remedy this as much as possible here.}
	$$ \tilde{I}_1(s) + \tilde{I}_2(s) = \left( \frac{(\zeta_{\F}^*)^2 C(\Phi)}{2(s-1/2)^2} + \frac{\zeta_{\F}^* B(\Phi)}{s-1/2} \right) A(\Psi) + \frac{(\zeta_{\F}^*)^2 C(\Phi)}{2(s-1/2)} \left( T_1(\Psi) + T_2(\Psi) \right) + O(1). $$
	$\tilde{I}_{\infty}'''(s)$ has a natural decomposition determined in \cite[p.31]{JZ87} as
	$$ \tilde{I}_{\infty}'''(s) = \sideset{}{_{\chi}} \sum \frac{1}{4 \pi} \int_{-\infty}^{\infty} \tilde{I}_{\chi}(s, i\tau) d\tau . $$
	It is expected that the inner integral of the RHS of the above equation has as principal part the form
\begin{equation}
	\frac{1}{4 \pi} \int_{-\infty}^{\infty} \tilde{I}_{\chi}(s, i\tau) d\tau = \left( \frac{(\zeta_{\F}^*)^2 C(\Phi)}{2(s-1/2)^2} + \frac{\zeta_{\F}^* B(\Phi)}{s-1/2} \right) A_{\chi}(\Psi) + \frac{(\zeta_{\F}^*)^2 C(\Phi)}{2(s-1/2)} T_{\chi}(\Psi) + O(1),
\label{JZPuzzle}
\end{equation}
	where $A_{\chi}(\Psi)$ and $T_{\chi}(\Psi)$ are certain functionals on $\Cont_c^{\infty}(\GL_2(\ag{A}), \omega^{-1})$ of different natures such that
\begin{equation}
	\sideset{}{_{\chi}} \sum A_{\chi}(\Psi) = -A(\Psi),
\label{Order2Match}
\end{equation}
	and that the summation over $T_{\chi}(\Psi)$ is equal to the term \cite[(6.36)]{GJ79}. But Jacquet-Zagier did not succeed to identify the order $1$ part of (\ref{JZPuzzle}), which is the precise meaning of what we call the \emph{Jacquet-Zagier puzzle} in this paper. We will get the precise form of (\ref{JZPuzzle}) via the combination of an extension of the theory of \emph{regularized integrals} due to Zagier \cite{Za82} and a \emph{deformation technic} inspired by a similar idea of deformation due to Michel \& Venkatesh \cite{MV10}.
\begin{remark}
	\emph{The whole theory of regularized integrals is based on Arthur's truncation applied to Eisenstein series. Hence our treatment does not completely avoid the technic of Arthur's truncation. But the goal of avoiding Arthur's truncation is to avoid its complication not the technic itself. Our completion of Jacquet-Zagier's approach applies Arthur's truncation with a minimum amount, at a level corresponding to the obtention of the Maass-Selberg relations in the classical treatment of the trace formula. It would be interesting to see how far our treatment can extend to general reductive groups based on the potential generalization of the regularized integrals to those settings.}
\end{remark}
\begin{remark}
	\emph{The development of the regularized integrals in this paper is made a little more general than it is needed to solve the Jacquet-Zagier puzzle. Namely, we make an effort to treat the theory for \emph{smooth} Eisenstein series instead of $\gp{K}$-finite ones. This is due to its application to the (explicit) subconvexity problem for $\GL_2$, especially our preprint \cite{Wu8}. In fact, the original choice of test vectors at infinite places in Michel \& Venkatesh's work \cite{MV10}, which are smooth vectors, seems to be irreplaceable, unlike the case for $\GL_1$ treated in \cite{Wu2}.}
\end{remark}

	\subsection{Notations, Conventions and Preliminaries}
	
		\subsubsection{Complex Analytic Notations}
	
	If $f$ is a meromorphic function around $s=s_0$, we introduce the coefficients into its Laurent expansion
	$$ f(s) = \sideset{}{_{-\infty < k < 0}} \sum \frac{f^{(k)}(s_0)}{(-k)!} (s-s_0)^k + \sideset{}{_{k \geq 0}} \sum \frac{f^{(k)}(s_0)}{k!} (s-s_0)^k. $$
	The terms for $k<0$ form the \emph{principal part} of $f$ at $s_0$. When $s_0$ is implicit and makes no ambiguity, $f^{(0)}(s_0)$ is intimately related with the \emph{finite part} functional, denoted by $\mathrm{f.p.}$ in \cite[Theorem (6.33)]{GJ79}.
	
		\subsubsection{Number Theoretic Notations}
	
	Throughout the paper, $\F$ is a (fixed) number field with ring of integers $\vo$ and of degree $r=[\F : \ag{Q}] = r_1 + 2 r_2$, where $r_1$ resp. $r_2$ is the number of real resp. complex places. $V_{\F}$ denotes the set of places of $\F$ and for any $v \in V_{\F}$, $\F_v$ is the completion of $\F$ with respect to the absolute value $\norm_v$ corresponding to $v$. $\ag{A} = \ag{A}_{\F}$ is the ring of adeles of $\F$, while $\ag{A}^{\times}$ denotes the group of ideles. We fix a section $s_{\F}$ of the adelic norm map $\norm_{\ag{A}}: \ag{A}^{\times} \to \ag{R}_+$, hence identify $\ag{A}^{\times}$ with $\ag{R}_+ \times \ag{A}^{(1)}$, where $\ag{A}^{(1)}$ is the kernel of the adelic norm map, i.e., the subgroup of ideles with norm $1$. For example, we can take
	$$ s_{\F}: \ag{R}_+ \to \ag{A}^{\times}, \quad t \mapsto (\underbrace{t^{1/r},\dots,t^{1/r}}_{r_1 \text{ real places}}, \underbrace{t^{1/r}, \dots, t^{1/r}}_{r_2 \text{ complex places}}, 1, \dots). $$
	
	We put the standard Tamagawa measure $dx = \sideset{}{_v} \prod dx_v$ on $\ag{A}$ resp. $d^{\times}x = \sideset{}{_v} \prod d^{\times}x_v$ on $\ag{A}^{\times}$. We recall their constructions. Let $\Tr = \Tr_{\ag{Q}}^{\F}$ be the trace map, extended to $\ag{A} \to \ag{A}_{\ag{Q}}$. Let $\psi_{\ag{Q}}$ be the additive character of $\ag{A}_{\ag{Q}}$ trivial on $\ag{Q}$, restricting to the infinite place as
	$$ \ag{Q}_{\infty} = \ag{R} \to \ag{C}^{(1)}, x \mapsto e^{2\pi i x}. $$
	We put $\psi = \psi_{\ag{Q}} \circ \Tr$, which decomposes as $\psi(x) = \sideset{}{_v} \prod \psi_v(x_v)$ for $x=(x_v)_v \in \ag{A}$. $dx_v$ is the additive Haar measure on $\F_v$, self-dual with respect to $\psi_v$. Precisely, if $\F_v = \ag{R}$, then $dx_v$ is the usual Lebesgue measure on $\ag{R}$; if $\F_v = \ag{C}$, then $dx_v$ is twice the usual Lebesgue measure on $\ag{C} \simeq \ag{R}^2$; if $v = \vp < \infty$ such that $\vo_{\vp}$ is the valuation ring of $\F_{\vp}$, then $dx_{\vp}$ gives $\vo_{\vp}$ the mass $\Dis_{\vp}^{-1/2}$, where $\Dis_{\vp} = \Dis(\F_{\vp})$ is the local component at $\vp$ of the discriminant $\Dis(\F)$ of $\F/\ag{Q}$ such that $\Dis(\F) = \sideset{}{_{\vp < \infty}} \prod \Dis_{\vp}$. Consequently, the quotient space $\F \backslash \ag{A}$ with the above measure quotient by the discrete measure on $\F$ admits the total mass $1$ \cite[Ch.\Rmnum{14} Prop.7]{Lan03}. Recall the local zeta-functions: if $\F_v = \ag{R}$, then $\zeta_v(s) = \Gamma_{\ag{R}}(s) = \pi^{-s/2} \Gamma(s/2)$; if $\F_v = \ag{C}$, then $\zeta_v(s) = \Gamma_{\ag{C}}(s) = (2\pi)^{-s} \Gamma(s)$; if $v=\vp < \infty$ then $\zeta_{\vp}(s) = (1-q_{\vp}^{-s})^{-1}$, where $q_{\vp} := \Nr(\vp)$ is the cardinality of $\vo/\vp$. We then define
	$$ d^{\times} x_v := \zeta_v(1) \frac{dx_v}{\norm[x]_v}. $$
	In particular, $\Vol(\vo_{\vp}^{\times}, d^{\times}x_{\vp}) = \Vol(\vo_{\vp}, dx_{\vp})$ for $\vp < \infty$. We equip $\ag{A}^{(1)}$ with the above measure on $\ag{A}^{\times}$ quotient by the measure $d^{\times} t = dt/\norm[t]$ on $\ag{R}_+$, where $dt$ is the usual Lebesgue measure on $\ag{R}$ restricted to $\ag{R}_+$. Consequently, $\F^{\times} \backslash \ag{A}^{(1)}$ admits the total mass \cite[Ch.\Rmnum{14} Prop.13]{Lan03}
	$$ \Vol(\F^{\times} \backslash \ag{A}^{(1)}) = \zeta_{\F}^* = \zeta_{\F}^{(-1)}(1) = \Res_{s=1} \zeta_{\F}(s), $$
	where $\zeta_{\F}(s) := \sideset{}{_{\vp < \infty}} \prod \zeta_{\vp}(s)$ is the Dedekind zeta-function of $\F$.
\begin{remark}
	\emph{Although the star $*$ is ambiguous (for example the complete Eisenstein series also uses it), it is conventional in the literature. Hence we keep it for $\zeta_{\F}^*$.}
\end{remark}
	
	For any automorphic representation $\pi$, $L(s,\pi)$ denotes the usual $L$-function of $\pi$ without components at infinity, $\Lambda(s,\pi)$ denotes its completion with components at infinity. The local component $L_{\vp}(s,\pi_{\vp})$ at a finite place $\vp$ takes $\Dis_{\vp}$ into account, so that $L_{\vp}(s,\mathbbm{1}_{\vp}) = \Dis_{\vp}^{s/2} \zeta_{\vp}(s)$, where $\mathbbm{1}$ is the trivial representation. We define the complete Dedekind zeta-function to be
	$$ \Lambda_{\F}(s) := \Lambda(s, \mathbbm{1}) = \Dis_{\F}^{s/2} \cdot \left( \sideset{}{_{v \mid \infty}} \prod \zeta_v(s) \right) \cdot \zeta_{\F}(s), $$
	so that it satisfies the functional equation $\Lambda_{\F}(s) = \Lambda_{\F}(1-s)$.
	
		\subsubsection{Automorphic Representation Theoretic Notations}
		
	We will work on algebraic groups $\GL_2$ and $\PGL_2$ over $\F$, the latter being the quotient of $\GL_2$ by its center over $\ag{A}$ or $\F_v$ in the category of abstract groups. We put the \emph{hyperbolic measure} instead of the Tamagawa measure on $\GL_2$. We recall its definition. We pick the standard maximal compact subgroup $\gp{K} = \sideset{}{_v} \prod \gp{K}_v$ of $\GL_2(\ag{A})$ by defining
	$$ \gp{K}_v = \left\{ \begin{matrix} \SO_2(\ag{R}) & \text{if } \F_v = \ag{R} \\ \SU_2(\ag{C}) & \text{if } \F_v = \ag{C} \\ \GL_2(\vo_{\vp}) & \text{if } v = \vp < \infty \end{matrix} \right. , $$
and equip it with the Haar probability measure $d\kappa_v$. We define the following one-parameter algebraic subgroups of $\GL_2(\F_v)$
	$$ \gp{Z}_v = \gp{Z}(\F_v) = \left\{ z(u) := \begin{pmatrix} u & 0 \\ 0 & u \end{pmatrix} \ \middle| \ u \in \F_v^{\times} \right\}, $$
	$$ \gp{N}_v = \gp{N}(\F_v) = \left\{ n(x) := \begin{pmatrix} 1 & x \\ 0 & 1 \end{pmatrix} \ \middle| \ x \in \F_v \right\}, $$
	$$ \gp{A}_v = \gp{A}(\F_v) = \left\{ a(y) := \begin{pmatrix} y & 0 \\ 0 & 1 \end{pmatrix} \ \middle| \ y \in \F_v^{\times} \right\}, $$
and equip them with the Haar measures on $\F_v^{\times}, \F_v, \F_v^{\times}$ respectively. The hyperbolic Haar measure $dg_v$ on $\GL_2(\F_v)$ is the push-forward of the product measure $d^{\times}u \cdot dx \cdot d^{\times}y / \norm[y]_v \cdot d\kappa_v$ under the Iwasawa decomposition map
	$$ \gp{Z}_v \times \gp{N}_v \times \gp{A}_v \times \gp{K}_v \to \GL_2(\F_v), \quad (z(u), n(x), a(y), \kappa) \mapsto z(u) n(x) a(y) \kappa. $$
	Similarly, the hyperbolic Haar measure $d\bar{g}_v$ on $\PGL_2(\F_v)$ is the push-forward of the product measure $dx \cdot d^{\times}y / \norm[y]_v \cdot d\kappa_v$ under the composition map
	$$ \gp{N}_v \times \gp{A}_v \times \gp{K}_v \to \GL_2(\F_v) \to \PGL_2(\F_v), \quad (n(x), a(y), \kappa) \mapsto [n(x) a(y) \kappa]. $$
	We then define and equip the quotient space
	$$ [\PGL_2] := \gp{Z}(\ag{A}) \GL_2(\F) \backslash \GL_2(\ag{A}) = \PGL_2(\F) \backslash \PGL_2(\ag{A}) $$
with the product measure $d\bar{g} := \sideset{}{_v} \prod d\bar{g}_v$ on $\PGL_2(\ag{A})$ quotient by the discrete measure on $\PGL_2(\F)$. The total mass is finite and equal to (c.f. Theorem \ref{AdelicRegThm} (4))
	$$ \Vol([\PGL_2]) = \frac{2\Lambda_{\F}(2)}{\Lambda_{\F}^{(-1)}(1)} \zeta_{\F}^* = 2 \Dis_{\F}^{1/2} \pi^{-r_1} (2\pi)^{-r_2} \zeta_{\F}(2). $$
	
	The product $\gp{B} := \gp{Z} \gp{N} \gp{A}$ is a Borel subgroup of $\GL_2$. We have the \emph{height function} $\Ht$ resp. $\Ht_v$ on $\GL_2(\ag{A})$ resp. $\GL_2(\F_v)$ associated with $\gp{B}$ defined by
	$$ \Ht_v \left( \begin{pmatrix} t_1 & x \\ 0 & t_2 \end{pmatrix} \kappa \right) := \extnorm{\frac{t_1}{t_2}}_v, \quad \forall t_1, t_2 \in \F_v^{\times}, x \in \F_v; $$
	$$ \Ht(g) = \sideset{}{_v} \prod \Ht_v(g_v), \quad g = (g_v)_v \in \GL_2(\ag{A}). $$
	
	We fix a Hecke character $\omega$ of $\F^{\times} \backslash \ag{A}^{\times}$, identify it with a unitary character of $\gp{Z}(\ag{A})$ in the obvious way. Let $\intL^2(\GL_2, \omega)$ denote the (Hilbert) space of Borel measurable functions $\varphi$ satisfying
	$$ \left\{ \begin{matrix} \varphi(z \gamma g) = \omega(z)\varphi(g), \quad \forall \gamma \in \GL_2(\F), z \in \gp{Z}(\ag{A}), g \in \GL_2(\ag{A}), \\ \int_{[\PGL_2]} \norm[\varphi(g)]^2 d\bar{g} < \infty. \end{matrix} \right. $$
	Let $\intL_0^2(\GL_2, \omega)$ denote the subspace of $\varphi \in \intL^2(\GL_2, \omega)$ such that its \emph{constant term}
	$$ \varphi_{\gp{N}}(g) := \int_{\F \backslash \ag{A}} \varphi(n(x)g) dx = 0, \quad \text{a.e. } \bar{g} \in [\PGL_2]. $$
	$\intL_0^2(\GL_2, \omega)$ is a closed subspace of $\intL^2(\GL_2, \omega)$. $\GL_2(\ag{A})$ acts on $\intL_0^2(\GL_2, \omega)$ resp. $\intL^2(\GL_2, \ag{A})$, giving rise to a unitary representation $\rpR_0$ resp. $\rpR$. The ortho-complement of $\rpR_0$ in $\rpR$ is the orthogonal sum of the one-dimensional spaces
	$$ \ag{C} \left( \xi \circ \det \right) : \quad \xi \text{ Hecke character such that } \xi^2 = \omega $$
and $\rpR_c$, which can be identified as a direct integral representation over the unitary dual of $\F^{\times} \backslash \ag{A}^{\times} \simeq \ag{R}_+ \times (\F^{\times} \backslash \ag{A}^{(1)} )$. Precisely, let $\tau \in \ag{R}$ and $\chi$ be a unitary character of $\F^{\times} \backslash \ag{A}^{(1)}$ regarded as a unitary character of $\F^{\times} \backslash \ag{A}^{\times}$ via trivial extension, we associate a unitary representation $\pi_{\chi}(i\tau)$ of $\GL_2(\ag{A})$ on the following Hilbert space $V_{\chi}(i\tau)$ of functions via right regular translation
	$$ \left\{ \begin{matrix} f\left( \begin{pmatrix} t_1 & x \\ 0 & t_2 \end{pmatrix} g \right) = \chi(t_1) \omega\chi^{-1}(t_2) \extnorm{\frac{t_1}{t_2}}_{\ag{A}}^{\frac{1}{2}+i\tau} f(g), \quad \forall t_1,t_2 \in \ag{A}^{\times}, x \in \ag{A}, g \in \GL_2(\ag{A}) ; \\ \int_{\gp{K}} \norm[f(\kappa)]^2 d\kappa < \infty . \end{matrix} \right. $$
	Let $\Psi: \GL_2(\ag{A}) \to \ag{C}$ be any smooth function of compact support modulo the center $\gp{Z}(\ag{A})$ such that
	$$ \Psi(zg) = \omega(z)^{-1} \Psi(g), \quad \forall z \in \gp{Z}(\ag{A}), g \in \GL_2(\ag{A}). $$
	Then $\Psi$ defines an operator $\rpR_c(\Psi)$ resp. $\pi_{\chi}(i\tau)(\Psi)$, which is of trace class. Part of $\Tr(\rpR_c(\Psi))$ is
	$$ \sideset{}{_{\chi}} \sum \frac{1}{4\pi} \int_{-\infty}^{\infty} \Tr(\pi_{\chi}(i\tau)(\Psi)) d\tau. $$
	The number of $\chi$ appearing in the above sum is in general infinite if the rank of the unit group $\vo^{\times}$ is non zero, but the sum and integral are absolutely convergent and the order is interchangeable.
	
		\subsubsection{Preliminaries on Smooth Eisenstein Series}
		
	There is an obvious extension of definition of $\pi_{\chi}(s)$ and $V_{\chi}(s)$ from $s \in i\ag{R}$ to $s \in \ag{C}$. Fixing $\chi$ and varying $s$, the \emph{flat section} map
	$$ V_{\chi} := V_{\chi}(0) \to V_{\chi}(s), \quad f \mapsto f_s $$
defined by requiring $f_s \mid_{\gp{K}} = f \mid_{\gp{K}}$ is an isometry of Hilbert spaces and $\gp{K}$-covariant. For $f \in V_{\chi}^{\infty}$ a smooth vector, we can define an Eisenstein series
	$$ \eis(s,f)(g) := \sideset{}{_{\gamma \in \gp{B}(\F) \backslash \GL_2(\F)}} \sum f_s(\gamma g) $$
convergent for $\Re s > 1/2$ and admitting a meromorphic continuation to $s \in \ag{C}$. The analytic properties of an Eisenstein series are best understood via the \emph{Godement sections}. Namely for any $\Phi \in \Sch(\ag{A}^2)$ a Schwartz function, the function
	$$ g \mapsto f_{\Phi}(s,g) = f_{\Phi}(s; \chi, \omega\chi^{-1}; g) := \chi(\det g) \norm[\det g]_{\ag{A}}^{\frac{1}{2}+s} \int_{\ag{A}^{\times}} \Phi((0,t)g) \omega^{-1}\chi^2(t) \norm[t]_{\ag{A}}^{1+2s} d^{\times}t $$
lies in $V_{\chi}(s)$. We then define
	$$ \eis(s,\Phi)(g) = \eis(s; \chi, \omega\chi^{-1}; \Phi)(g) := \sideset{}{_{\gamma \in \gp{B}(\F) \backslash \GL_2(\F)}} \sum f_{\Phi}(s, \gamma g). $$
	The constant term of an Eisenstein series is intimately related with the intertwining operator
	$$ \Intw_{\chi}(s): V_{\chi}(s) \to V_{\omega\chi^{-1}}(-s), \quad f \mapsto \left( g \mapsto \int_{\ag{A}} f(wn(x)g) dx \right). $$
	We have in fact
	$$ \eisCst(s,f)(g) = f_s(g) + \left( \Intw_{\chi}(s) f_s \right)(g), \quad \eisCst(s,\Phi)(g) = f_{\Phi}(s,g) + f_{\widehat{\Phi}}(-s,g), $$
where $\widehat{\Phi}$ is the twisted Fourier transform defined by
	$$ \widehat{\Phi}(x,y) := \int_{\ag{A}^2} \Phi(u,v) \psi\left( - (u,v) w \begin{pmatrix} x \\ y \end{pmatrix} \right) dudv, \quad w := \begin{pmatrix} 0 & -1 \\ 1 & 0 \end{pmatrix}. $$
	Note that $\Intw_{\chi}(s)$ is unitary if $s \in i\ag{R}$.
	
	If $e \in V_{\mathbbm{1}}(0)$ is the spherical vector taking value $1$ on $\gp{K}$, we write
	$$ \eis(s,g) := \eis(s,e)(g). $$
	There is a standard choice of $\Phi$ given by
\begin{equation}
	\Phi(x,y) = \sideset{}{_v} \prod \Phi_v(x_v, y_v), \quad \Phi_v(x,y) := \left\{ \begin{matrix} e^{-\pi (x^2 + y^2)} & \F_v = \ag{R} \\ e^{-2\pi (\norm[x]^2 + \norm[y]^2)} & \F_v = \ag{C} \\ 1_{\vo_{\vp}}(x) \cdot 1_{\vo_{\vp}}(y) & v=\vp < \infty \end{matrix} \right. .
\label{SphPhi}
\end{equation}
	Then we have the relations
\begin{equation}
	\left. \begin{matrix} f_{\Phi}(s,g) \\ \eis(s,\Phi)(g) \end{matrix} \right\} = \Dis_{\F}^{-1-s} \cdot (2\pi)^{r_2} \cdot \Lambda_{\F}(1+2s) \cdot \left\{ \begin{matrix} e_s(g) \\ \eis(s,g) \end{matrix} \right. ,
\label{SecRel}
\end{equation}
\begin{equation}
	f_{\widehat{\Phi}}(s,g) = \Dis_{\F}^{2s} f_{\Phi}(s,g).
\end{equation}
	We thus obtain
	$$ \Intw_{\mathbbm{1}}(s) e_s = \frac{\Lambda_{\F}(1-2s)}{\Lambda_{\F}(1+2s)} e_{-s}, \quad \Res_{s=1/2} \eis(s,g) = \frac{\Lambda_{\F}^{(-1)}(1)}{2\Lambda_{\F}(2)}. $$
	For a general smooth vector $f \in V_{\chi}^{\infty}$, we can always find $\Phi \in \Sch(\ag{A}^2)$ such that $\eis(s,\Phi)$ is the product of $\eis(s,f)$ and a meromorphic function in $s$ independent of $g$. 


	\subsection{Plan of The Paper}
	
	As indicated in the introduction, our solution of the Jacquet-Zagier puzzle has two main ingredients.
	
	The first main ingredient is an extension of the theory of regularized integrals due to Zagier, aiming at removing the restrictions in the original theory. This will be treated in Section 2. 	
	
	Our extension of the theory consists of two steps. The first step deals with the applicability to the ``constant-like'' functions. This is achieved in Theorem \ref{AdelicRegThm}. In fact, after obtaining the \emph{fundamental identity of the regularized integral} Theorem \ref{AdelicRegThm} (2), there are two operations to consider: taking residue at $s=1/2$ and letting $T \to \infty$. The order of the two operations is important. Zagier's original treatment applies $T \to \infty$ before taking the residue. Reversing the order, we get a formula Theorem \ref{AdelicRegThm} (4) which applies also to the constant function. In the second step, we exploit the equivalence of two definitions of regularization as in Zagier's work \cite[\S Reinterpretation]{Za82} with the introduction of the regularizing Eisenstein series Definition \ref{RegEisDef}, as well as its derivatives. This allows us to remove the restriction on functions containing ``exponent $1$''.
	
	For completeness, we also include a full treatment in the ``regular case'' in the adelic setting, which was first developed in \cite[\S 4.3]{MV10}. We stick strictly to the original idea of Zagier, avoiding another definition of regularized integrals due to Michel \& Venkatesh. In particular, we establish the $\PGL_2(\ag{A})$-invariance of the regularized integral directly in Proposition \ref{RegGInv} (to be compared with \cite[\S 4.3.6]{MV10}).
	
	As a fundamental preliminary, bounding the smooth Eisenstein series is necessary. This follows the same (classical) idea in the $\gp{K}$-finite case, with a little more complicated technics. We stated the necessary results in the end of Section 1.2 and omit the proofs.
	
	In Section 3, we then apply the idea of \emph{deformation}, inspired by \cite[\S 5.2.6]{MV10}, more deeply into the formula itself instead of in the \emph{application} of the formula as in \cite[\S 3.1.11, 3.2.4, 3.2.8, 4.1.9 \& 4.4.3]{MV10}. This allows us to treat the regularized integral of product of any two Eisenstein series, yielding Theorem \ref{RIPEisUnitary}.
	
	In Section 4, we solve the puzzle.
	
	In Section 5, we include some fundamental estimations concerning smooth Eisenstein series. This section can be skipped for the first reading, and are relevant only for the method of Michel-Venkatesh on subconveixty problem.

\section{Zagier's Regularized Integral with Extension}

	\subsection{Brief Review of the Existing Theory}
	
	Trying to enlarge the applicability of the Rankin-Selberg method, Zagier \cite{Za82} invented the regularized integral in automorphic representation theory for $\PGL_2$, which deals with certain non convergent integrals. In the course of solving the subconvexity problem for $\GL_2$, Michel \& Venkatesh \cite[\S 4.3]{MV10} developed this theory adelically. Roughly speaking, three definitions of regularized integral are available:

\begin{itemize}
	\item[(1)] We subtract from $\varphi$ a non integral part $\Reis(\varphi)$, considered to have integral $0$ by abuse of orthogonality, and define
	$$ \int^{\reg} \varphi = \int \varphi - \Reis(\varphi). $$
	For example, on $\ag{R}$ with Lebesgue measure, one can define the ``integral'' of $x \mapsto e^{2\pi i x}$ to be $0$ since it is an eigenfunction of the Laplacian $d^2/dx^2$ with eigenvalue different from the one of a constant function. Then functions of the form $\phi(x) + a e^{2\pi i x}$ for $a \in \ag{C}, \phi \in \intL^1(\ag{R})$ becomes ``integrable''.
	\item[(2)] We introduce some suitable meromorphic function $E(s)$ with constant residue $1$ at $0$, such that
	$$ \int \varphi \cdot E(s) $$
is convergent for a certain range of $s$ and has a meromorphic continuation. We then define
	$$ \int^{\reg} \varphi = \Res_{s=0} \int \varphi \cdot E(s). $$
	In the automorphic setting, a natural candidate is the non-holomorphic Eisenstein series $E(z,s)$. This is the starting point of the original work of Zagier \cite{Za82}, which also provides the equivalence of (1) and (2).
	\item[(3)] We find a/any measure-operator $\sigma_0$ such that the convolution $\sigma_0 * \varphi$ becomes integrable, and that the dual measure $\sigma_0^{\vee}$ with respect to change of variables satisfies $\sigma_0^{\vee} * 1 = R_0 \neq 0$. We then define
	$$ \int^{\reg} \varphi = R_0^{-1} \int \sigma_0 * \varphi. $$
	For example, on $\Gamma \backslash \ag{H}$ with $\Gamma = \mathrm{PSL}_2(\ag{Z})$ we consider the above Eisenstein series
	$$ E(z,s) := \sideset{}{_{\gamma \in \Gamma_{\infty} \backslash \Gamma}} \sum \Im(\gamma.z)^s. $$
	If $T(p)$ denotes the usual normalized Hecke operator of level a prime $p$ and $s \neq 0,1$, then we have
	$$ T(p).E(z,s) = \lambda_p(s) E(z,s), \quad \lambda_p(s) = \frac{p^{s-1/2} + p^{1/2-s}}{p^{1/2} + p^{-1/2}}. $$
	We observe that $\sigma_0 := T(p) - \lambda_p(s)$ annihilates $E(z,s)$ while $ \sigma_0^{\vee}.1 = 1 - \overline{\lambda_p(s)} \neq 0 $. Hence we define the integral of such Eisenstein series to be $0$. This is the new viewpoint in Michel \& Venkatesh \cite[\S 4.3]{MV10} of Zagier's theory.
\end{itemize}
	
\begin{remark}
	\emph{It should be pointed out that (3) is an extension of the theory of regularized integral, not just another equivalent definition. In fact, $\sigma_0 * \varphi$ is a variant of making $\varphi$ integrable, different from the version of subtraction $\varphi - \Reis(\varphi)$. This variant seems to be more convenient as long as the inner product is concerned, such is the case for its application \cite{MV10} to the subconvexity problem for $\GL_2$. Only when the underlying period formula is not a pairing invariant by the group action, should one need to turn back to the viewpoint (1).}
\end{remark}

	However, neither \cite{Za82} nor \cite[\S 4.3]{MV10} was capable of treating an exceptional case (we shall call it the \emph{singular case}), in which $\varphi$ has the quasi-character $\norm_{\ag{A}}$ in its set of exponents \cite[\S 4.3.3]{MV10}. An example is given by the \emph{regularizing Eisenstein series} Definition \ref{RegEisDef}. It is easy to see that any measure $\sigma_0$ derived from Hecke operators which makes the regularizing Eisenstein series integrable also makes $\sigma_0^{\vee}$ annihilate $1$ (c.f. Remark \ref{KillRegEis}). Hence the definition (3) can not be extended to the singular case. Concretely, over $\ag{Q}$ the regularizing Eisenstein series
	$$ E^{\reg}(z,1) := \lim_{s \to 1} \left( E(z,s) - \frac{3}{\pi} \cdot \frac{1}{s-1} \right) $$
	is out of the applicability of the existing theory of regularized integrals. In fact, it satisfies
	$$ (T(p)-1).E^{\reg}(z,1) \neq 0, \quad (T(p)-1)^2.E^{\reg}(z,1) = 0. $$
	More strangely, the constant function $1$ is also beyond the applicability of the theory.

	\subsection{Regularized Integral on $\ag{R}_+$}
	
\begin{definition}
	Let $a: \ag{R}_+ \to \ag{C}$ be a continuous function. It is \emph{regularizable} if
\begin{itemize}
	\item[(1)] for any $N \gg 1$ as $t \to \infty$
	$$ a(t) = f(t)+O(t^{-N}), \quad f(t)=\sum_{i=1}^l \frac{c_i}{n_i!} t^{\frac{1}{2}+\alpha_i} \log^{n_i}t, $$
where $c_i, \alpha_i \in \ag{C}, n_i \in \ag{N}$;
	\item[(2)] for some $\alpha \in \ag{R}$ as $t \to 0^+$
	$$ a(t) = O(t^{\alpha}). $$
 \end{itemize}
 	In this case, we write for $T > 0$
 	$$ h_T(s) = \sum_{i=1}^l \frac{c_i}{n_i!} \frac{\partial^{n_i}}{\partial s^{n_i}} \left( \frac{T^{s+\alpha_i}}{s+\alpha_i} \right) = \sum_{i=1}^l c_i \sum_{m=0}^{n_i} \frac{(-1)^{n_i-m}}{m!} \frac{T^{s+\alpha_i} \log^m T}{(s+\alpha_i)^{n_i-m+1}}. $$
 \label{RegFuncRDef}
\end{definition}
\begin{lemma}
	For $f(t)$ regularizable, if $\lim_{T \to \infty} \int_1^T f(t) \frac{dt}{t^2}$ exists, then $\Re \alpha_i < 1/2$ for all $1 \leq i \leq l$.
\label{IntExpBd}
\end{lemma}
\begin{proof}
	If not, the condition implies that
\begin{align*}
	\int_1^T f(t) \frac{dt}{t^2} &= \sum_{\alpha_j \neq \frac{1}{2}} c_j \sum_{m=0}^{n_j} \frac{(-1)^{n_j-m}}{(\alpha_j - \frac{1}{2})^{n_j-m+1}} T^{\alpha_j-\frac{1}{2}} (\log T)^m - \sum_{\alpha_j \neq \frac{1}{2}} c_j \frac{(-1)^{n_j}}{(\alpha_j-\frac{1}{2})^{n_j+1}} \\
	&\quad + \sum_{\alpha_j=\frac{1}{2}} \frac{c_j}{n_j!} \frac{(\log T)^{n_j+1}}{n_j+1}
\end{align*}
is bounded as $T \to +\infty$. Let $\sigma = \max_j \Re \alpha_j$. We distinguish two cases.

\noindent (1) $\sigma=1/2$. Let $l=\max \{ n_j: \Re \alpha_j = 1/2, \alpha_j \neq 1/2 \} \cup \{ n_j+1: \alpha_j=1/2 \}$. We divide both sides of the equation by $(\log T)^l$ and let $T \to +\infty$ to get
\begin{equation}
	\lim_{T \to \infty} \sum c_j T^{i\tau_j} = 0
\label{AbsEq}
\end{equation}
where $\Im \alpha_j = \tau_j$ for $j$ such that either $\Re \alpha_j = 1/2, \alpha_j \neq 1/2, n_j=l$ or $\alpha_j=1/2, n_j+1=l$. In particular $\tau_j$ are mutually distinct.

\noindent (2) $\sigma > 1/2$. Let $l=\max \{ n_j: \Re \alpha_j = \sigma \}$. We divide both sides of the equation by $T^{\sigma-1/2} (\log T)^l$ and let $T \to +\infty$ to get an equation of the same form as (\ref{AbsEq}).

\noindent We conclude because (\ref{AbsEq}) contradicts the following Corollary \ref{ErgodicLemma}.
\end{proof}
\noindent Write $\ag{T} = \ag{R} / \ag{Z}$. Let $\vec{\theta} = (\theta_1, \cdots, \theta_n) \in \ag{R}^n$. For any $\vec{x} \in \ag{R}^n$, we write by $[\vec{x}]$ its image in $\ag{T}^n$. Define
	$$ \ag{T}_{\vec{\theta}} = \overline{\lbrace [t\vec{\theta}] : t\in \ag{R} \rbrace}. $$
It is a closed subgroup of $\ag{T}^n$. Furthermore, the one parameter subgroup $U_{\vec{\theta}} = \lbrace t.\vec{\theta} : t \in \ag{R} \rbrace$ of $\ag{R}^n$ acts uniquely ergodically on $\ag{T}_{\vec{\theta}}$ w.r.t. the Haar measure of $\ag{T}_{\vec{\theta}}$. More precisely,
\begin{lemma}
	For any $f \in \Cont(\ag{T}_{\vec{\theta}})$, we have
	$$ \lim_{T \to \infty} \frac{1}{T}\int_0^T f([\vec{x}] + [t.\vec{\theta}]) dt = \int_{\ag{T}_{\vec{\theta}}} f dm_{\theta} $$
	for all $[\vec{x}] \in \ag{T}_{\vec{\theta}}$. Here $dm_{\theta}$ is the normalized Haar measure on $\ag{T}_{\vec{\theta}}$.
\end{lemma}
\begin{proof}
	Consider the group of characters ${\rm Ch}(\ag{T}^n)$ of $\ag{T}^n$ given by
	$$ e_{\vec{n}}(\vec{x}) = e(\vec{n} \cdot \vec{x}) = e(\sum_{i=1}^n n_i x_i), \vec{n} \in \ag{N}^n, \vec{x} \in \ag{R}^n $$
where $e(x) = e^{2\pi i x}$. By the duality theorem for locally compact abelian groups, the group of characters ${\rm Ch}(\ag{T}_{\vec{\theta}})$ is the quotient of ${\rm Ch}(\ag{T}^n)$ by the subgroup of $e_{\vec{n}}$'s which vanish on $\ag{T}_{\vec{\theta}}$. Obviously, we have
	$$ e_{\vec{n}}(\ag{T}_{\vec{\theta}})=1 \Leftrightarrow e(t\vec{n} \cdot \vec{\theta}) = 1 \text{ for } \forall t \in \ag{R} \Leftrightarrow \vec{n} \cdot \vec{\theta}=0. $$
	So ${\rm Ch}(\ag{T}_{\vec{\theta}})$ are $e_{\vec{n}}$'s modulo the subgroup of $e_{\vec{n}}$'s with $\vec{n} \cdot \vec{\theta}=0$. Let $[e_{\vec{n}}] \neq 0$ denote a non trivial equivalence class of $e_{\vec{n}}$ in the quotient group, we calculate
	$$ \extnorm{ \frac{1}{T} \int_0^T [e_{\vec{n}}]( [\vec{x}] + [t\vec{\theta}] ) dt } = \extnorm{ \frac{e(\vec{n} \cdot \vec{x})}{T} \cdot \frac{e(T\vec{n} \cdot \vec{\theta} ) - 1}{\vec{n} \cdot \vec{\theta}} } \leq \frac{2}{\norm[T] \cdot \norm[\vec{n} \cdot \vec{\theta}]} \to 0, T \to \infty. $$
	The lemma is thus proved for $f = [e_{\vec{n}}]$, hence the $\ag{C}$-vector space generated by ${\rm Ch}(\ag{T}_{\vec{\theta}})$, which is also a $*$-subalgebra of $\Cont(\ag{T}_{\vec{\theta}})$. The lemma then follows by a standard application of the complex version of the Stone-Weierstrass theorem.
\end{proof}
\begin{corollary}
	Consider the function $f(x) = \sum_{k=1}^n a_k x^{i \theta_k}, x \in \ag{R}_+$, where $a_k \in \ag{C}, \theta_k \in \ag{R}, 1\leq k \leq n$. If $\lim_{x \to +\infty} f(x)$ exists, then $f(x)$ is a constant function.
\label{ErgodicLemma}
\end{corollary}
\begin{proof}
	Note that $f(e^{2\pi t}) = \sum_{k=1}^n a_k e(t\theta_k), t \in \ag{R}$. If $f$ is not constant, we can find $t_1, t_2 \in \ag{R}$ s.t. $f(e^{2\pi t_1}) \neq f(e^{2\pi t_2})$. By continuity, we then find some neighborhood $U_1$ of $[t_1 \vec{\theta}]$, and $U_2$ of $[t_2 \vec{\theta}]$ such that
	$$ \sum_{k=1}^n a_k e(x_k) \neq \sum_{k=1}^n a_k e(y_k), \forall \vec{x} \in U_1, \vec{y} \in U_2. $$
	But the flow $\gamma(t) = [t\vec{\theta}]$ meets both $U_1$ and $U_2$ infinitely often by the lemma, hence $\lim_{t \to \infty} f(e^{2\pi t})$ can not exist, contradicting the hypothesis.
\end{proof}

	\subsection{Regularized Integral for $\PGL_2$}
	
\begin{definition}
	Let $\varphi: \GL_2(\F)\gp{Z}(\ag{A}) \backslash \GL_2(\ag{A}) \to \ag{C}$ be a continuous function. It is \emph{slowly increasing} if for some $c \in \ag{R}$ and $g$ lying in some Siegel domain we have
	$$ \norm[\varphi(g)] \ll \Ht(g)^c, \Ht(g) \to \infty. $$
\label{SlowIncDef}
\end{definition}
\begin{definition}
	Let $\varphi: \GL_2(\F)\gp{Z}(\ag{A}) \backslash \GL_2(\ag{A}) \to \ag{C}$ be a slowly increasing function. Its \emph{regularizing kernel} $a(t,\varphi)$ is
	$$ a(t,\varphi) = \int_{\F \backslash \ag{A} \times \F^{\times} \backslash \ag{A}^{(1)} \times \gp{K}} \varphi(n(x)a(t^+y)k) dx d^{\times}y dk, $$
	where $t^+ = s_{\F}(t)$ is the image of $t$ under the section of the adelic norm map $\ag{A}^{\times} \to \ag{R}_+$ recalled in the beginning of this paper.
\end{definition}
\begin{definition}
	We call a slowly increasing function $\varphi: \GL_2(\F)\gp{Z}(\ag{A}) \backslash \GL_2(\ag{A}) \to \ag{C}$ \emph{regularizable} if its regularizing kernel $a(t,\varphi)$ satisfies the condition (1) of Definition \ref{RegFuncRDef}. In this case, we define for $s \in \ag{C}, \Re s \gg 1$
	$$ R(s,\varphi) = \int_0^{\infty} (a(t,\varphi)-f(t)) t^{s-\frac{1}{2}} \frac{dt}{t}; \quad R^*(s,\varphi) = \Lambda_{\F}(1+2s) R(s,\varphi). $$
	The space of regularizable functions is denoted by $\Aut^{\reg}(\GL_2)$.
\label{RegFuncDef}
\end{definition}
\begin{remark}
	\emph{This is equivalent to saying $a(t,\varphi)$ regularizable, i.e., the condition (2) of Definition \ref{RegFuncRDef} is automatically satisfied due to the following Corollary \ref{SIncBdAt0}.}
\end{remark}
\begin{lemma}
	If $\gamma \in \GL_2(\F) - \gp{B}(\F)$, then we have $ \Ht(\gamma g) \leq \Ht(g)^{-1}$.
\label{HtBd}
\end{lemma}
\begin{proof}
	This is \cite[Lemma 3.19]{Wu5}.
\end{proof}
\begin{corollary}
	If $\varphi$ is slowly increasing as in Definition \ref{SlowIncDef}, then we have
	$$ \norm[\varphi(g)] \ll \Ht(g)^{\min(0,-c)}, \Ht(g) \to 0. $$
\label{SIncBdAt0}
\end{corollary}
\begin{proof}
	If $c \leq 0$, then it is easy to see that $\varphi$ is bounded, since elements of bounded height in a Siegel domain form a compact subset and $\varphi$ is continuous. The same argument shows that if $c<0$ we can assume $\norm[\varphi(g)] \ll \Ht(g)^c$ to hold in a whole Siegel domain $S$ containing a fundamental domain. If $\Ht(g)$ is small, we take $\gamma \in \GL_2(\F)$ such that $\gamma g \in S$, thus by the lemma we get
	$$ \norm[\varphi(g)] = \norm[\varphi(\gamma g)] \ll \Ht(\gamma g)^c \leq \Ht(g)^{-c}. $$
\end{proof}
\noindent The following function together with its Taylor expansion plays an important role:
\begin{equation}
	\lambda_{\F}(s) := \frac{\Lambda_{\F}(-2s)}{\Lambda_{\F}(2+2s)} = \frac{\lambda_{\F}^{(-1)}(0)}{s} + \sum_{n=0}^{\infty} \frac{s^n}{n!} \lambda_{\F}^{(n)}(0).
\label{lambdaFDef}
\end{equation}
Recall the truncation operator $\Lambda^c$ defined in \cite[(5.5)]{GJ79}. Let $f_0 \in \Ind_{\gp{B}(\ag{A}) \cap \gp{K}}^{\gp{K}} (1,1)$ be constant equal to $1$ on $\gp{K}$.
\begin{theorem}
	(Adelic version of regularization due to Zagier \cite{Za82})
\begin{itemize}
	\item[(1)] Let $\varphi: \GL_2(\F)\gp{Z}(\ag{A}) \backslash \GL_2(\ag{A}) \to \ag{C}$ be a slowly increasing function. For $s \in \ag{C}, \Re s \gg 1$ and any $T \gg 1$ we have
	$$ \int_{[\PGL_2]} \varphi(g) \Lambda^T \eis(s,f_0)(g) dg = \int_0^T a(t,\varphi)t^{s-\frac{1}{2}} \frac{dt}{t} - \int_T^{\infty} a(t,\varphi) t^{-s-\frac{1}{2}} \frac{dt}{t} \cdot \lambda_{\F}(s-1/2). $$
	\item[(2)] If $\varphi$ is, in addition, regularizable, then we have for $T \gg 1$ the following \emph{fundamental identity of regularized integral}
\begin{align*}
	&\quad R^*(s,\varphi) + \Lambda_{\F}(1+2s) h_T(s) + \Lambda_{\F}(1-2s) h_T(-s) \\
	&= \int_{[\PGL_2]} \varphi(g) \Lambda^T \eis^*(s,f_0)(g) dg + \\
	&\quad \Lambda_{\F}(1+2s) \int_T^{\infty} (a(t,\varphi)-f(t)) t^{s-\frac{1}{2}} \frac{dt}{t} + \Lambda_{\F}(1-2s) \int_T^{\infty} (a(t,\varphi)-f(t)) t^{-s-\frac{1}{2}} \frac{dt}{t}.
\end{align*}
	In particular, $R(s,\varphi)$ has a meromorphic continuation to $s \in \ag{C}$ with possible poles at $s=\pm 1/2, \pm \alpha_i, (\rho-1)/2$ for $\rho$ running over the non-trivial zeros of $\zeta_{\F}$, and satisfies the functional equation
	$$ R^*(s,\varphi) = R^*(-s,\varphi). $$
	\item[(3)] Under the condition of (2), if $\Re \alpha_i < 0$ for all $1 \leq i \leq l$, then we have
	$$ R(s,\varphi) = \int_{[\PGL_2]} \varphi(g) \eis(s,f_0)(g) dg, \quad \max_{1 \leq i \leq l} \alpha_i < \Re s < -\max_{1 \leq i \leq l} \alpha_i. $$
	\item[(4)] Under the condition of (2), if $\varphi$ is integrable on $ [\PGL_2] := \GL_2(\F)\gp{Z}(\ag{A}) \backslash \GL_2(\ag{A}) $, then we have $\Re \alpha_i < 1/2$ for all $1 \leq i \leq l$, and
	$$ \int_{[\PGL_2]} \varphi(g) dg = \frac{1}{\lambda_{\F}^{(-1)}(0)} \left( \Res_{s=\frac{1}{2}} R(s,\varphi) + c_i \delta_{\substack{\alpha_i=-\frac{1}{2} \\ n_i=0}} \right), \quad \Vol([\PGL_2]) = \frac{\zeta_{\F}^*}{\lambda_{\F}^{(-1)}(0)}. $$
\end{itemize}
\label{AdelicRegThm}
\end{theorem}
\begin{proof}
	Since the proofs of (1) to (3) are quite similar to that in \cite{Za82}, we only mention some essential points of them. Only (4) needs more explanation.
	
\noindent (1) This is standard Rankin-Selberg unfolding together with
	$$ \Lambda^T \eis(s,f_0)(g) = \sum_{\gamma \in \gp{B}(\F) \backslash \GL_2(\F)} f_{0,s}(\gamma g) 1_{\Ht(\gamma g) \leq T} - \sum_{\gamma \in \gp{B}(\F) \backslash \GL_2(\F)} \Intw f_{0,s}(\gamma g) 1_{\Ht(\gamma g) > T}. $$
	
\noindent (2) It follows from rewriting the two terms at the right hand side of (1). For the first term we have
\begin{align*}
	\int_0^T a(t,\varphi) t^{s-\frac{1}{2}} \frac{dt}{t} &= R(s,\varphi) - \int_T^{\infty} (a(t,\varphi)-f(t)) t^{s-\frac{1}{2}} \frac{dt}{t} + \sum_{i=1}^l \frac{c_i}{n_i!} \int_0^T t^{s+\alpha_i} \log^{n_i} t \frac{dt}{t} \\
	&= R(s,\varphi) - \int_T^{\infty} (a(t,\varphi)-f(t)) t^{s-\frac{1}{2}} \frac{dt}{t} + h_T(s).
\end{align*}
For the second term we have a similar equality.

\noindent (3) In the case $\Re \alpha_i < 0$ for all $i$, we let $T \to \infty$, taking into account Proposition \ref{GlobRDEisWhi} and $h_T(s) \to 0$, to get the asserted equation.

\noindent (4) The integrability of $\varphi$ implies $\Re \alpha_i < 1/2$ for all $i$ by Lemma \ref{IntExpBd}. We take residue at $s=1/2$ on both sides of the equation obtained in (2) and divide by $\Lambda_{\F}(1+2s)$ to see for $T \gg 1$
\begin{align*}
	&\quad \Res_{s=\frac{1}{2}} R(s,\varphi) + \Res_{s=\frac{1}{2}} h_T(s) + \lambda_{\F}^{(-1)}(0) h_T(-\frac{1}{2}) \\
	&= \lambda_{\F}^{(-1)}(0) \int_{\FundD_T} \varphi(g)dg + \Res_{s=\frac{1}{2}} \int_{\FundD-\FundD_T} \varphi(g) \left( \eis(s,f_0)(g) - \eis(s,f_0)_N(g) \right) dg \\
	&\quad + \lambda_{\F}^{(-1)}(0) \int_T^{\infty} (a(t,\varphi)-f(t)) \frac{dt}{t^2},
\end{align*}
	where $\FundD$ is the standard fundamental domain for $[\PGL_2]$ and $\FundD_T$ is the set of $g \in \FundD$ such that $\Ht(g) \leq T$. It is easy to see that as $T \to \infty$
	$$ h_T(-\frac{1}{2}) \to 0; \quad \int_T^{\infty} (a(t,\varphi)-f(t)) \frac{dt}{t^2} \to 0. $$
	By Proposition \ref{GlobRDEisWhi}, $\eis(s,f_0)(g) - \eis(s,f_0)_N(g)$ is of uniformly rapid decay with respect to $\Ht(g), g \in \FundD$ as $s$ remains in a compact neighborhood of $1/2$. Hence
	$$ \int_{\FundD-\FundD_T} \varphi(g) \left( \eis(s,f_0)(g) - \eis(s,f_0)_N(g) \right) dg $$
	is holomorphic at $s=1/2$. We compute $\Res_{s=\frac{1}{2}} h_T(s)$ by noting that the second summand of
	$$ \frac{d^n}{d s^n} \left( \frac{T^s}{s} \right) = \frac{(-1)^n n!}{s^{n+1}} + \int_0^{\log T} t^n e^{st} dt $$
is holomorphic at $s=0$, and get
	$$ \Res_{s=\frac{1}{2}} h_T(s) = c_i \delta_{\substack{\alpha_i=-\frac{1}{2} \\ n_i=0}}. $$
	This completes the proof of the first equation. The volume computation follows by taking $\varphi \equiv 1$ and noting that $a(t,1) \equiv \Vol(\F^{\times} \backslash \ag{A}^{(1)}) = \zeta_{\F}^*$.
\end{proof}
\begin{definition}
	We define the \emph{regularized integral} of a regularizable function $\varphi: [\PGL_2] \to \ag{C}$ as
	$$ \int_{[\PGL_2]}^{\reg} \varphi(g) dg = \frac{1}{\lambda_{\F}^{(-1)}(0)} \left( \Res_{s=\frac{1}{2}} R(s,\varphi) + c_i \delta_{\substack{\alpha_i=-\frac{1}{2} \\ n_i=0}} \right), $$
	where $c_i, \alpha_i, n_i$ are associated with $a(t,\varphi)$ as in Definition \ref{RegFuncRDef}. We call the first term the \emph{principal part} of the regularized integral, the second the \emph{degenerate part} of the regularized integral. The regularized integral is linear and extends the integral on $\Aut^{\reg}(\GL_2) \cap \intL^1(\GL_2, 1)$.
\label{RegIntDef}
\end{definition}

	\subsection{Basic Properties}
	
\begin{definition}
	Let $\omega$ be a unitary character of $\F^{\times} \backslash \ag{A}^{\times}$. Let $\varphi$ be a smooth function on $\GL_2(\F) \backslash \GL_2(\ag{A})$ with central character $\omega$. We call $\varphi$ \emph{finitely regularizable} if there exist characters $\chi_i: \F^{\times} \backslash \ag{A}^{\times} \to \ag{C}^{(1)}$, $\alpha_i \in \ag{C}, n_i \in \ag{N}$ and smooth functions $f_i \in \Ind_{\gp{B}(\ag{A}) \cap \gp{K}}^{\gp{K}} (\chi_i, \omega \chi_i^{-1})$ for $1 \leq i \leq l$, such that for any $M \gg 1$
	$$ \varphi(n(x)a(y)k) = \varphi_{\gp{N}}^*(n(x)a(y)k) + O(\norm[y]_{\ag{A}}^{-M}), \text{ as } \norm[y]_{\ag{A}} \to \infty, $$
	where we have written the \emph{essential constant term}
	$$ \varphi_{\gp{N}}^*(n(x)a(y)k)=\varphi_{\gp{N}}^*(a(y)k)=\sum_{i=1}^l \chi_i(y) \norm[y]_{\ag{A}}^{\frac{1}{2}+\alpha_i} \log^{n_i} \norm[y]_{\ag{A}} f_i(k). $$
	In this case, we call $\Ex(\varphi)=\{ \chi_i \norm^{\frac{1}{2}+\alpha_i}: 1 \leq i \leq l \}$ the \emph{exponent set} of $\varphi$, and define
	$$ \Ex^+(\varphi) = \{ \chi_i \norm^{\frac{1}{2}+\alpha_i} \in \Ex(\varphi): \Re \alpha_i \geq 0 \}; \quad \Ex^-(\varphi) = \{ \chi_i \norm^{\frac{1}{2}+\alpha_i} \in \Ex(\varphi): \Re \alpha_i < 0 \}. $$
	The space of finitely regularizable functions with central character $\omega$ is denoted by $\Aut^{\freg}(\GL_2,\omega)$.
\label{FinRegFuncDef}
\end{definition}
\begin{remark}
	\emph{In the case $\omega=1$, a finitely regularizable is smooth and regularizable in the sense of Definition \ref{RegFuncDef}. But a smooth regularizable function doesn't need to be finitely regularizable.}
\end{remark}
\begin{definition}
	In the case $\omega^{-1}\xi^2(t)=\norm[t]_{\ag{A}}^{i\mu}$ for some $\mu \in \ag{R}$, we introduce the \emph{regularizing Eisenstein series} for $f \in V_{\xi,\omega\xi^{-1}}^{\infty}$ and $s$ in a neighborhood of $(1-i\mu)/2$
	$$ \eis^{\reg}(s,f)(g) = \eis(s,f)(g) - \frac{\Lambda_{\F}(1-2s-i\mu)}{\Lambda_{\F}(1+2s+i\mu)} \int_{\gp{K}} f(\kappa) d\kappa \cdot \chi^{-1}(\det g) \norm[\det g]_{\ag{A}}^{\frac{i\mu}{2}}. $$
	It is holomorphic at $s=(1-i\mu)/2$.
\label{RegEisDef}
\end{definition}
\begin{remark}
	\emph{Let $e_0 \in \Res_{\gp{K}}^{\GL_2(\ag{A})} \pi(1,1)$ be the spherical function taking value $1$ on $\gp{K}$. Let $T(\vp)$ denote the order $1$ Hecke operator at a finite place $\vp$ with cardinality of the residue field $q$. For $s \neq 0$, we have}
	$$ T(\vp). \eis(\frac{1}{2}+s, e_0) = \lambda_{\vp}(s) \eis(\frac{1}{2}+s, e_0), \quad \lambda_{\vp}(s) = \frac{q^{1/2+s} + q^{-(1/2+s)}}{q^{1/2}+q^{-1/2}}; $$
	$$ \text{or} \quad T(\vp). \eis^{\reg}(\frac{1}{2}+s, e_0) = \lambda_{\vp}(s) \eis^{\reg}(\frac{1}{2}+s, e_0) + (\lambda_{\vp}(s)-1) \cdot \lambda_{\F}(s). $$
	\emph{The pole of $\lambda_{\F}(s)$, defined in (\ref{lambdaFDef}), at $s=0$ is compensated by the zero of $\lambda_{\vp}(s)-1$, hence we get}
	$$ T(\vp). \eis^{\reg}(\frac{1}{2}, e_0) = \eis^{\reg}(\frac{1}{2}, e_0) + \lambda_{\vp}^{(1)}(0) \cdot \lambda_{\F}^{(-1)}(0). $$
	\emph{In particular, $(T(\vp)-1)^2$ annihilates $\eis^{\reg}(1/2,e_0)$.}
\label{KillRegEis}
\end{remark}
\begin{remark}{\emph{(Violation of Covariance)}}
	\emph{Unlike the usual Eisenstein series, the map}
	$$ \pi^{\infty}(\xi \norm^{(1-i\mu)/2}, \omega\xi^{-1} \norm^{-(1-i\mu)/2}) \to \Cont^{\infty}(\GL_2, \omega), \quad f_{(1-i\mu)/2} \mapsto \eis^{\reg}((1-i\mu)/2, f) $$
\emph{is NOT $\GL_2(\ag{A})$-covariant. But it's still $\gp{K}$-covariant.}
\end{remark}
\begin{remark}
	\emph{By Proposition \ref{GlobRDEisWhi}, \ref{GlobRDEisCst} and the definition of cusp forms, $\Aut^{\freg}(\GL_2,\omega)$ contains:}
\begin{itemize}
	\item $\chi(\det g)$ for quasi-characters $\chi$ such that $\chi^2=\omega$;
	\item smooth cusp forms, i.e., $\Aut_{\cusp}^{\infty}(\GL_2,\omega)$;
	\item $\frac{\partial^n}{\partial s^n}\eis(s,f)$ for some $n \in \ag{N}$ and smooth $f \in \Ind_{\gp{B}(\ag{A})}^{\GL_2(\ag{A})}(\xi, \omega\xi^{-1})$ with $\Re s \geq 0$ and $s \neq -i\mu/2, 1/2-i\mu/2$ for $\mu=\mu(\omega^{-1}\xi^2)$ if $\omega\xi^{-2}$ is trivial on $\ag{A}^{(1)}$;
	\item $\frac{\partial^n}{\partial s^n} \mid_{s=-i\frac{\mu}{2}} (s+i\frac{\mu}{2})\eis^*(s,f)$ for some $n \in \ag{N}$ and smooth $f$, $\mu$ the same as above;
	\item $\frac{\partial^n}{\partial s^n}\eis^{\reg}(\frac{1-i\mu}{2},f)$ for some $n \in \ag{N}$ and $f,\mu$ the same as above;
	\item $\varphi=\Pi_{j=1}^l \varphi_j$ where $\varphi_j \in \Aut^{\freg}(\GL_2,\omega_j)$ with $\omega = \Pi_j \omega_j$.
\end{itemize}
	\emph{In the last case, we have $ \varphi_{\gp{N}}^* = \Pi_j \varphi_{j,\gp{N}}^* $. Note that we have excluded $\eis(s,f)$ for $\Re s < 0$. But they are actually present since they are related to the case $\Re s > 0$ by functional equation.}
\end{remark}
\begin{definition}
	Let $\omega$ be a unitary character of $\F^{\times} \backslash \ag{A}^{\times}$. The $\intL^2$-\emph{residual space} of central character $\omega$, denoted by $\Reis(\GL_2,\omega)$, is the direct sum of the vector spaces $\Reis^+(\GL_2,\omega)$ resp. $\Reis^{\reg}(\GL_2,\omega)$ spanned by functions
	$$ \frac{\partial^n}{\partial s^n} \eis(s,f), \text{if } s \neq \frac{1-i\mu}{2} \quad \text{resp.} \quad \frac{\partial^n}{\partial s^n} \eis^{\reg}(\frac{1-i\mu}{2},f) $$
	where $s \in \ag{C}, \Re s > 0$ and for some unitary character $\xi$ of $\F^{\times} \backslash \ag{A}^{\times}$, $f \in V_{\xi,\omega\xi^{-1}}^{\infty}$ and $\mu$ as above.
\end{definition}
\begin{proposition}
	$\Aut^{\freg}(\GL_2,\omega)$ is stable under the right translation by $\GL_2(\ag{A})$. Moreover, for any $\varphi \in \Aut^{\freg}(\GL_2,\omega)$ and any $g \in \GL_2(\ag{A})$, their sets of exponents are the same:
	$$ \Ex(g.\varphi) = \Ex(\varphi). $$
\label{FregStab}
\end{proposition}
\begin{proof}
	Take $f \in \Res_{\gp{K}}^{\GL_2(\ag{A})} \pi(\xi, \omega\xi^{-1})$ and the flat section $f_s \in \pi(\xi \norm^s, \omega\xi^{-1} \norm^{-s})$ associated to it. It suffices to show that for any $s_0 \in \ag{C}, n \in \ag{N}$ and fixed $g \in \GL_2(\ag{A})$, the right translate by $g$ of the partial derivative of this flat section, as a function on $\GL_2(\ag{A})$
	$$ x \mapsto g.f_{s_0}^{(n)}(x) := \frac{\partial^n}{\partial s^n} \mid_{s=s_0} f_s(xg) $$
is a linear combination of such functions. The above function is clearly left invariant by $\gp{N}(\ag{A})$, of central character $\omega$. Taking $x=a(y)\kappa$ for $y \in \ag{A}^{\times}, \kappa \in \gp{K}$ and writing
	$$ \kappa g = z' n' a(y') \kappa', \quad \text{for} \quad z' \in \gp{Z}(\ag{A}), n' \in \gp{N}(\ag{A}), y' \in \ag{A}^{\times}, \kappa' \in \gp{K}, $$
where $z',n',y',\kappa'$ are viewed as functions in $\kappa$, we obtain
	$$ \frac{\partial^n}{\partial s^n} \mid_{s=s_0} f_s(a(y)\kappa g) = \sum_{k=0}^n \binom{n}{k} \xi(y) \norm[y]_{\ag{A}}^{\frac{1}{2}+s_0} (\log \norm[y]_{\ag{A}})^k \cdot \norm[y']_{\ag{A}}^{\frac{1}{2}+s_0} (\log \norm[y']_{\ag{A}})^{n-k} \omega(z') \xi(y') f(\kappa'). $$
	Although $z',n',y',\kappa'$ are not uniquely determined by $\kappa$, both $\norm[y']_{\ag{A}}$ and $\omega(z') \xi(y') f(\kappa')$ are, and define smooth functions on $\gp{K}$. Moreover, the function
	$$ f_k(\kappa) := \norm[y']_{\ag{A}}^{\frac{1}{2}+s_0} (\log \norm[y']_{\ag{A}})^{n-k} \omega(z') \xi(y') f(\kappa') \in \Res_{\gp{K}}^{\GL_2(\ag{A})} \pi(\xi, \omega\xi^{-1}). $$
	Hence we get the relation
	$$ \frac{\partial^n}{\partial s^n} \mid_{s=s_0} f_s = \sum_{k=0}^n \binom{n}{k} \frac{\partial^k}{\partial s^k} \mid_{s=s_0} f_{k,s} $$
and conclude.
\end{proof}
\begin{proposition}
	The vector space $\Reis^+(\GL_2, \omega)$ is stable under the right translation by $\GL_2(\ag{A})$.
\label{Reis+Stab}
\end{proposition}
\begin{proof}
	Take a flat section $f_s$ as in the proof of Proposition \ref{FregStab}. Fix $s_0 \in \ag{C}, n \in \ag{N}, g \in \GL_2(\ag{A})$ with $\Re s_0 > 0$. Let $e_{\vec{k}}$ be an orthonormal $\gp{K}$-isotypic basis of $\Res_{\gp{K}}^{\GL_2(\ag{A})} \pi(\xi, \omega\xi^{-1})$. There is an orthonormal $\gp{K}$-isotypic basis $\tilde{e}_{\vec{k}}$ of $\Res_{\gp{K}}^{\GL_2(\ag{A})} \pi(\omega\xi^{-1}, \xi)$ and for any $n \in \ag{N}$ there is $A > 0$ such that (c.f. \cite[\S 3]{Wu5})
	$$ \Intw_s e_{\vec{k},s} = \mu_{\vec{k}}(s) \tilde{e}_{\vec{k},-s}, \quad \extnorm{\mu_{\vec{k}}^{(n)}(s)} \ll \lambda_{\vec{k}}^A $$
for some meromorphic function $\mu_{\vec{k}}(s)$, regular in $\Re s > 0, s \neq 1/2$. Here $\lambda_{\vec{k}}$ is the eigenvalue of the Laplacian on $\gp{K}_{\infty}$ for $e_{\vec{k}}$, and the bound is uniform for $s$ lying in any  compact subset of the regular region. We also have an expansion
	$$ g.f_s = \sideset{}{_{\vec{k}}} \sum a_{\vec{k}}(s,g) e_{\vec{k},s}, $$
for some functions $a_{\vec{k}}(s,g)$ holomorphic in $s$ and smooth in $g$, since
	$$ a_{\vec{k}}(s,g) = \int_{\gp{K}} f_s(\kappa g) \overline{e_{\vec{k}}(\kappa)} d\kappa. $$
	Moreover, we have $ a_{\vec{k}}(s,g) \ll_{g,f,N} \lambda_{\vec{k}}^{-N} $ for any $N \in \ag{N}$, uniformly for $s$ lying in any compact neighborhood of $s_0$. It follows that for any $l \in \ag{N}$
	$$ a_{\vec{k}}^{(l)}(s,g) := \frac{\partial^l}{\partial s^l} \mid_{s=s_0} a_{\vec{k}}(s,g) \ll_{g,f,N} \lambda_{\vec{k}}^{-N}. $$
	We thus get
	$$ g.f_{s_0}^{(n)} = \sideset{}{_{k=0}^n} \sum \sideset{}{_{\vec{k}}} \sum \binom{n}{k} a_{\vec{k}}^{(n-k)}(s_0,g) e_{\vec{k},s_0}^{(k)}. $$
	The inner sum coincides with the value at $s=s_0$ of a flat section $f_{k,s}^{(k)}$ defined by
	$$ f_k = \binom{n}{k} \sideset{}{_{\vec{k}}} \sum a_{\vec{k}}^{(n-k)}(s_0,g) e_{\vec{k}}, $$
which is smooth. We can also verify that
\begin{align*}
	\sideset{}{_{k=0}^n} \sum \frac{\partial^k}{\partial s^k} \mid_{s=s_0} \Intw_s f_{k,s} &= \sideset{}{_{k=0}^n} \sum \sideset{}{_{\vec{k}}} \sum \binom{n}{k} a_{\vec{k}}^{(n-k)}(s_0,g) \frac{\partial^k}{\partial s^k} \mid_{s=s_0} (\mu_{\vec{k}}(s) \tilde{e}_{\vec{k},-s}) \\
	&= \sideset{}{_{\vec{k}}} \sum \frac{\partial^n}{\partial s^n} \mid_{s=s_0} (a_{\vec{k}}(s,g) \mu_{\vec{k}}(s) \tilde{e}_{\vec{k},-s}) = \frac{\partial^n}{\partial s^n} \mid_{s=s_0} \Intw_s f_s,
\end{align*}
	hence we deduce that
	$$ g.\eisCst^{(n)}(s_0,f) = \sideset{}{_{k=0}^n} \sum \eisCst^{(k)}(s_0,f_k). $$
	We conclude by
	$$ g.\eis^{(n)}(s_0,f) = \sideset{}{_{k=0}^n} \sum \eis^{(k)}(s_0,f_k) $$
since both sides are orthogonal to the cusp forms and have the same constant term.
\end{proof}
\begin{remark}
	\emph{$f_k$ constructed in both proofs of Proposition \ref{FregStab} and \ref{Reis+Stab} coincide with each other.}
\end{remark}
\begin{remark}
	\emph{The subspace $\Reis^{\reg}(\GL_2,\omega)$ is stable under right translation by $\gp{K}$ but not by $\GL_2(\ag{A})$. Take $\omega = 1$ for example. Choose a finite place $\vp_0$ with uniformizer $\varpi_0$. Let $e_0 \in \Res_{\gp{K}}^{\GL_2(\ag{A})} \pi(1,1)$ be the spherical function taking value $1$ on $\gp{K}$. Let $e_1$ be defined by $e_{1,v} = e_{0,v}$ for $v \neq \vp_0$; while $e_{1,\vp_0}$ is unitary, $\gp{K}_0[\vp_0]$-invariant and orthogonal to $e_{0,\vp_0}$. For example, if $q = \Nr(\vp_0)$, we can take}
	$$ e_{1,\vp_0} = \sqrt{1+q^{-1}} \cdot \left\{ \sqrt{q+1} \cdot 1_{\gp{K}_0[\vp]} - \frac{1}{\sqrt{q+1}} \cdot 1_{\gp{K}_{\vp}} \right\}. $$
	\emph{It can be computed, writing $\tilde{\lambda}_{\F}(s) = \lambda_{\F}(s-1/2)$, that}
	$$ \Intw_s e_{0,s} = \tilde{\lambda}_{\F}(s) e_{0,-s}, \quad \Intw_s e_{1,s} = \tilde{\lambda}_{\F}(s) \mu_1(s) e_{0,-s}, \quad \eisCst^{\reg}(1/2,e_0) = e_{0,1/2} - \lambda_{\F}^{(-1)}(0) e_{0,-1/2}^{(1)}, $$
	$$ a(\varpi_0^{-1}).e_{0,s} = c_1(s) e_{1,s} + c_0(s) e_{0,s}; \quad \text{where} $$
	$$ \mu_1(s) = q^{-2s} \frac{1-q^{-(1-2s)}}{1+q^{-(1+2s)}}, \quad c_1(s) = \frac{q^{s+1/2} - q^{-(s+1/2)}}{q^{1/2}+q^{-1/2}}, \quad c_0(s) = \frac{q^s+q^{-s}}{q^{1/2}+q^{-1/2}}. $$
	\emph{We deduce that}
	$$ \varphi := a(\varpi_0^{-1}).\eis^{\reg}(1/2,e_0) - c_0(1/2) \eis^{\reg}(1/2,e_0) - c_1(1/2) \eis^{\reg}(1/2,e_1) $$
\emph{has constant term}
	$$ \varphi_{\gp{N}} = \lambda_{\F}^{(-1)}(0) (c_1^{(1)}(-1/2) - c_1(1/2) \mu_1^{(1)}(1/2)) \cdot e_{1,-1/2} + \lambda_{\F}^{(-1)}(0) c_0^{(1)}(-1/2) e_{0,-1/2} \neq 0, $$
\emph{hence $0 \neq \tilde{\varphi} \notin \Reis(\GL_2,1)$. Consequently, $a(\varpi_{\vp}^{-1}).\eis^{\reg}(1/2,e_0) \notin \Reis(\GL_2,1)$. We also deduce that}
	$$ \int_{[\PGL_2]}^{\reg} a(\varpi_0^{-1}).\eis^{\reg}(1/2,e_0) = \int_{[\PGL_2]}^{\reg} \varphi = \lambda_{\F}^{(-1)}(0) c_0^{(1)}(-1/2) \neq 0, $$
\emph{hence the regularized integral is in general not $\PGL_2(\ag{A})$-invariant (c.f. Proposition \ref{RegGInv} (1) below).}
\label{RegIntGinvLoss}
\end{remark}
\begin{proposition}
	Let $\varphi \in \Aut^{\freg}(\GL_2,\omega)$.
\begin{itemize}
	\item[(1)] We can always find (not unique) $\Reis(\varphi) \in \Reis(\GL_2,\omega)$ such that $\varphi - \Reis(\varphi) \in \intL^1(\GL_2, \omega)$.
	\item[(2)] If for any $\chi \in \Ex(\varphi)$, we have $\Re \chi \neq 1/2$, then there is a unique function $\Reis(\varphi) \in \Reis(\GL_2,\omega)$ such that $\varphi - \Reis(\varphi) \in \intL^2(\GL_2, \omega)$. Moreover, for any $X$ in the universal enveloping algebra of $\GL_2(\ag{A}_{\infty})$, we have $X.\varphi - X.\Reis(\varphi) \in \intL^2(\GL_2, \omega)$, i.e., $\Reis(X.\varphi) = X.\Reis(\varphi)$.
\end{itemize}
\label{ResSp}
\end{proposition}
\begin{proof}
	By Lemma \ref{IntExpBd} and Proposition \ref{GlobRDEisWhi}, it is not difficult to see that $\Reis(\GL_2,\omega) \cap \intL^2(\GL_2(\F) \backslash \GL_2(\ag{A}), \omega) = \{0\}$, which implies the uniqueness of $\Reis(\varphi)$ in (2). For the existence, we find $\chi_i, \alpha_i, n_i, f_i$ as in Definition \ref{FinRegFuncDef}. Then we take, writing $\mu_i = \mu(\omega^{-1}\chi_i^2)$,
\begin{equation}
	\Reis(\varphi) = \sum_{\substack{\Re \alpha_j > 0 \\ \alpha_j \neq \frac{1}{2}+i\mu_j}} \frac{\partial^{n_j}}{\partial s^{n_j}}\eis(\alpha_j,\omega,\omega^{-1}\chi_j;f_j) + \sum_{\substack{\Re \alpha_j > 0 \\ \alpha_j = \frac{1}{2}+i\mu_j}} \frac{\partial^{n_j}}{\partial s^{n_j}}\eis^{\reg}(\alpha_j,\omega,\omega^{-1}\chi_j;f_j).
\label{ReisDef}
\end{equation}
	Thus for any $\chi \in \Ex(\varphi-\Reis(\varphi))$, we have $\Re \chi \leq 1/2$ resp. $\Re \chi < 1/2$ under the condition in (2). Hence $\varphi-\Reis(\varphi) \in \intL^1(\GL_2, \omega)$ resp. $\intL^2(\GL_2, \omega)$. For the ``moreover'' part, it suffices to see that the differential operator $X$ does not increase the real part of elements in $\Ex(\varphi)$, which is essentially due to the following calculation:
	$$ y \frac{d}{dy} (y^{\sigma} \log^k y) = y^{\sigma} \log^k y + k y^{\sigma} \log^{k-1}y, \forall y>0, \sigma \in \ag{C}, k \in \ag{N}. $$
\end{proof}
\begin{definition}
	In the case (2), we call $\Reis(\varphi)$ the $\intL^2$-residue of $\varphi$. For definiteness, we shall write $\Reis(\varphi)$ to be the one given by (\ref{ReisDef}).
\end{definition}
\begin{proposition}
\begin{itemize}
	\item[(1)] For any $\Reis \in \Reis(\GL_2,1)$, we have
	$$ \int_{[\PGL_2]}^{\reg} \Reis(g) dg = 0. $$
	If moreover $\Reis \in \Reis^+(\GL_2,1)$, then for any $g_0 \in \GL_2(\ag{A})$, we have
	$$ \int_{[\PGL_2]}^{\reg} g_0.\Reis(g) dg = 0. $$
	\item[(2)] Let $\varphi \in \Aut^{\freg}(\GL_2,1)$. For any $\Reis \in \Reis(\GL_2,1)$ such that $\varphi - \Reis \in \intL^1([\PGL_2])$ we have
	$$ \int_{[\PGL_2]}^{\reg} \varphi(g) dg = \int_{[\PGL_2]} (\varphi-\Reis)(g) dg. $$
	In particular, $\int_{[\PGL_2]}^{\reg}$ is always $\gp{K}$-invariant. It is $\GL_2(\ag{A})$-invariant on the subspace of $\varphi \in \Aut^{\freg}(\GL_2,1)$ such that $\Ex(\varphi)$ does not contain $\norm_{\ag{A}}$.
\end{itemize}
\label{RegGInv}
\end{proposition}
\begin{proof}
	For (1), the second assertion follows from the first by Proposition \ref{Reis+Stab}. We calculate $a(t,\Reis)$ for $\Reis(g) = \frac{\partial^n}{\partial s^n} \eis(s,f)(gg_0)$ with $s \neq 1/2-i\mu(\chi)$ resp. $\frac{\partial^n}{\partial s^n} \eis^{\reg}(\frac{1}{2}-i\mu(\chi),f)(g)$ for some $f \in V_{\chi,\chi^{-1}}^{\infty}$ with $\Re s > 0$, a unitary character $\chi$ of $\F^{\times} \backslash \ag{A}^{\times}$ if $\chi \mid_{\F^{\times} \backslash \ag{A}^{(1)}} = 1$. Due to the integral $\int_{\F^{\times} \backslash \ag{A}^{(1)}} d^{\times}y$, it is easy to see that $a(t,\varphi)$ is non-vanishing only if $\chi$ is trivial on $\F^{\times} \backslash \ag{A}^{(1)}$, in which case $\mu(\omega\chi^2)=2\mu(\chi)$. We also notice that we can interchange the order of $\Intw$ and $\int_{\gp{K}} dk$ since they commute with each other.
\begin{itemize}
	\item $\Reis(g) = \frac{\partial^n}{\partial s^n} \eis(s,f)(g)$ with $s \neq 1/2-i\mu(\chi)$: We get
	$$ a(t,\Reis) = \zeta_{\F}^* \int_{\gp{K}} f(k)dk \cdot \left\{ t^{\frac{1}{2}+s+i\mu(\chi)} \log^n t + \frac{\partial^n}{\partial s^n} \left( t^{\frac{1}{2}-s-i\mu(\chi)} \frac{\Lambda_{\F}(1-2s-2i\mu(\chi))}{\Lambda_{\F}(1+2s+2i\mu(\chi))} \right) \right\}, $$
	and conclude by the fact that $a(t,\Reis)$ has no constant term as a function of $t$.
	\item $\Reis(g)=\frac{\partial^n}{\partial s^n} \eis^{\reg}(\frac{1}{2}-i\mu(\chi),f)(g)$: We get
\begin{align*}
	a(t,\Reis) &= \zeta_{\F}^* \int_{\gp{K}} f(k)dk \cdot \\
	&\quad \left\{ t \log^n t + \frac{\partial^n}{\partial s^n} \mid_{s=\frac{1}{2}-i\mu(\chi)} \left( (t^{\frac{1}{2}-s-i\mu(\chi)}-1) \frac{\Lambda_{\F}(1-2s-2i\mu(\chi))}{\Lambda_{\F}(1+2s+2i\mu(\chi))} \right) \right\} \\
	&= \zeta_{\F}^* \int_{\gp{K}} f(k)dk \cdot \\
	&\quad \left\{ t \log^n t + \sum_{l=0}^n \binom{n}{l} \frac{(-1)^l}{l+1} \cdot \log^{l+1}t \cdot \frac{d^{n-l}}{d s^{n-l}} \mid_{s=0} \left( \frac{s\Lambda_{\F}(2s)}{\Lambda_{\F}(2-2s)} \right) \right\},
\end{align*}
	and conclude the same way as in the previous case.
\end{itemize}
	For (2), the first part is trivial. For the second part, we note that $\varphi-\Reis(\varphi) \in \intL^1$ implies $g_0.\varphi-g_0.\Reis(\varphi) \in \intL^1$ for any $g_0 \in \GL_2(\ag{A})$, and $g_0.\Reis(\varphi)$ has regularized integral $0$ by (1) if either $\Reis \in \Reis^+(\GL_2,1)$ or $g_0 \in \gp{K}$.
\end{proof}
\begin{remark}
	\emph{The above proof of (2) is to be compared with \cite[\S 4.3.6]{MV10}, where another simpler but indirect proof was given for the ``regular case''.}
\end{remark}

\section{Product of Two Eisensetein Series: Singular Cases}

	\subsection{Deformation Technics}
	
	Above all, we have the following result in the regular case (c.f. \cite[\S 3]{Za82}).
\begin{lemma}
	Let $\xi_j,\xi_j'$ be (unitary) Hecke characters with $\xi_1\xi_1' \xi_2 \xi_2' = 1$ and write $\pi_j = \pi(\xi_j,\xi_j')$, $j=1,2$. For $f_j \in \pi_j$, we shall write $\eis^{\sharp}$ for $\eis$ or $\eis^{\reg}$, whichever is regular at $s=1/2$. If $\pi_1 \not \simeq \widetilde{\pi_2}$, then for any $n_1, n_2 \in \ag{N}$, we have
	$$ \int_{[\PGL_2]}^{\reg} \eis^{(n_1)}(0,f_1) \eis^{(n_2)}(0,f_2) = 0, \quad \int_{[\PGL_2]}^{\reg} \eis^{\sharp,(n_1)}(1/2,f_1) \eis^{\sharp,(n_2)}(1/2,f_2) = 0. $$
\label{SimpleProd}
\end{lemma}
\begin{proof}
	$R(s, \eis^{(n_1)}(0,f_1) \eis^{(n_2)}(0,f_2))$ resp. $R(s, \eis^{\sharp,(n_1)}(1/2,f_1) \eis^{\sharp,(n_2)}(1/2,f_2))$ represents the value at $s_1=s_2=0$ of $\partial_1^{n_1} \partial_2^{n_2}$ of
	$$ \Lambda(1/2+s+s_1+s_2, \xi_1\xi_2) \Lambda(1/2+s+s_1-s_2, \xi_1\xi_2') \Lambda(1/2+s-s_1+s_2, \xi_1'\xi_2) \Lambda(1/2+s-s_1-s_2, \xi_1'\xi_2') \quad \text{resp.} $$
	$$ \Lambda(3/2+s+s_1+s_2, \xi_1\xi_2) \Lambda(1/2+s+s_1-s_2, \xi_1\xi_2') \Lambda(1/2+s-s_1+s_2, \xi_1'\xi_2) \Lambda(-1/2+s-s_1-s_2, \xi_1'\xi_2'), $$
which is regular at $s=1/2$ by assumption. The degenerate part is also easily seen to be $0$ by assumption. We conclude by Definition \ref{RegIntDef}.
\end{proof}

	In other cases beyond the above one, it seems to be difficult to obtain simple formulas by definition. However, the idea of deformation does provide simple and useful formulas. In general, if $\varphi \in \Aut^{\freg}(\PGL_2), \Reis \in \Reis(\PGL_2)$ are given, so that $\varphi - \Reis \in \intL^1([\PGL_2])$, and if we can find continuous families $\varphi_s \in \Aut^{\freg}(\PGL_2), \Reis_s \in \Reis(\PGL_2)$ which coincide with $\varphi, \Reis$ at $s=0$, then we have
\begin{equation}
	\int_{[\PGL_2]}^{\reg} \varphi = \int_{[\PGL_2]} \varphi - \Reis = \lim_{s \to 0} \int_{[\PGL_2]} \varphi_s - \Reis_s = \lim_{s \to 0} \left( \int_{[\PGL_2]}^{\reg} \varphi_s - \int_{[\PGL_2]}^{\reg} \Reis_s \right).
\label{DeformTec}
\end{equation}
	All the formulas we are going to obtain will follow this principle together with (suitable simple variants of Lemma \ref{SimpleProd}). Since the computation is long, we shall only give detail in the most complicated cases. The notations in Lemma \ref{SimpleProd} will be used unless otherwise explicitly reset.

	\subsection{Unitary Series}

\begin{definition}
	If $f \in \Res_{\gp{K}}^{\GL_2(\ag{A})} \pi(\xi_1,\xi_2)$, we define for any $s \in \ag{C}$ an operator $\Intw_s: \Res_{\gp{K}}^{\GL_2(\ag{A})} \pi(\xi_1,\xi_2) \to \Res_{\gp{K}}^{\GL_2(\ag{A})} \pi(\xi_2,\xi_1)$ (abuse of notations) by requiring
	$$ \Intw_s f_s(a(y)\kappa) = \xi_2(y)\norm[y]_{\ag{A}}^{\frac{1}{2}-s} \Intw_sf(\kappa), \quad \text{i.e.,} \quad \Intw_s f_s = (\Intw_s f)_{-s}; $$
	$$ \text{resp.} \quad \widetilde{\Intw}_s = \Intw_s \circ (I-\ProjP_{\gp{K}}e_{\xi}), \quad \text{if } \xi_1=\xi_2=\xi $$
	with the Taylor expansion at $s=0$ resp. $s=1/2$ when $\xi_1=\xi_2 = \xi$ (since $\Intw_s$ is ``diagonalizable'')
	$$ \Intw_sf = \sum_{n=0}^{\infty} \frac{s^n}{n!} \Intw_0^{(n)}f, \quad \text{resp.} \quad \widetilde{\Intw}_{1/2+s} f = \sum_{n=0}^{\infty} \frac{s^n}{n!} \widetilde{\Intw}_{1/2}^{(n)}f. $$
Here $\ProjP_{\gp{K}}$, with $d\kappa$ the probability Haar measure on $\gp{K}$, is defined to be the map
	$$ \pi(\xi,\xi) \to \ag{C}, f \mapsto \int_{\gp{K}} f(\kappa) \xi^{-1}(\kappa) d\kappa, $$
and $e_{\xi} = \xi \circ \det \in \Res_{\gp{K}}^{\GL_2(\ag{A})} \pi(\xi,\xi)$.
\end{definition}
\begin{lemma}
	Let $f_1,f_2 \in \Res_{\gp{K}}^{\GL_2(\ag{A})} \pi(1,1)$. For $0 \neq s \in \ag{C}$ small, we have
	$$ \int_{[\PGL_2]}^{\reg} \eis(s,f_1) \eis^{(1)}(0,f_2) = 0. $$
\label{SimpleProdU}
\end{lemma}
\begin{proof}
	This is a variant of Lemma \ref{SimpleProd}.
\end{proof}
\begin{lemma}
	Let $f, f_1, f_2 \in \Res_{\gp{K}}^{\GL_2(\ag{A})} \pi(1,1)$. For $0 \neq s \in \ag{C}$ small, we have for any $n, n_1, n_2 \in \ag{N}$
	$$ \int_{[\PGL_2]}^{\reg} \eis^{\reg, (n)}(\frac{1}{2}+s, f) = -\frac{\lambda_{\F}^{(n)}(s)}{\lambda_{\F}^{(-1)}(0)} \ProjP_{\gp{K}}(f); $$
	$$ \int_{[\PGL_2]}^{\reg} \eis^{\reg, (n_1)}(\frac{1}{2}+s, f_1) \eis^{\reg, (n_2)}(\frac{1}{2}, f_2) = 0. $$
\label{SimpleProdSing}
\end{lemma}
\begin{proof}
	The first formula follows immediately from Proposition \ref{RegGInv} and definition. The second one is a variant of Lemma \ref{SimpleProd}.
\end{proof}

	We continue to use the notations in the previous lemma. We can write
	$$ \eisCst(s,f_1) \eisCst^{(1)}(0,f_2) = 2(f_1f_2)_{1/2+s}^{(1)} + 2(\Intw_s f_1f_2)_{1/2-s}^{(1)} + (f_1 \Intw_0^{(1)}f_2)_{1/2+s} + (\Intw_s f_1 \Intw_0^{(1)}f_2)_{1/2-s}. $$
	We tentatively define
\begin{align*}
	\Reis^{\reg}(s) &:= s^{-1} \left\{ 2 \eis^{\reg,(1)}(1/2+s, f_1f_2) + 2 \eis^{\reg,(1)}(1/2-s, \Intw_s f_1f_2) \right. \\
	 &\left. + \eis^{\reg}(1/2+s, f_1 \Intw_0^{(1)}f_2) + \eis^{\reg}(1/2-s, \Intw_s f_1 \Intw_0^{(1)}f_2) \right\}
\end{align*}
	Applying Lemma \ref{SimpleProdU}, \ref{SimpleProdSing} with $n=0, 1$ together with (\ref{DeformTec}), we get
\begin{align*}
	\int_{[\PGL_2]}^{\reg} \eis^{(1)}(0,f_1) \eis^{(1)}(0,f_2) &= \lim_{s \to 0} \int_{[\PGL_2]} s^{-1} \eis(s,f_1) \eis^{(1)}(0,f_2) - \Reis^{\reg}(s) \\
	&= \frac{1}{s} \cdot \left\{ 2 \frac{\lambda_{\F}^{(1)}(s)}{\lambda_{\F}^{(-1)}(0)} \ProjP_{\gp{K}}(f_1f_2) + 2  \frac{\lambda_{\F}^{(1)}(-s)}{\lambda_{\F}^{(-1)}(0)} \ProjP_{\gp{K}}(\Intw_s f_1f_2) \right. \\
	&\left. + \frac{\lambda_{\F}(s)}{\lambda_{\F}^{(-1)}(0)} \ProjP_{\gp{K}}(f_1 \Intw_0^{(1)}f_2) + \frac{\lambda_{\F}(-s)}{\lambda_{\F}^{(-1)}(0)} \ProjP_{\gp{K}}(\Intw_s f_1 \Intw_0^{(1)}f_2) \right\}.
\end{align*}
	Taking Laurent expansions, we verify that the function in $s$ in the range of the above limit is regular at $s=0$, unlike its appearance. The properties
	$$ \ProjP_{\gp{K}}(f_1 \Intw_0^{(k)}f_2) = \ProjP_{\gp{K}}(f_2 \Intw_0^{(k)}f_1), \forall k \in \ag{N}; \quad \Intw_0^{(2)} = \Intw_0^{(1)} \circ \Intw_0^{(1)} $$
	$$ \text{coming from} \quad \Intw_s \circ \Intw_{-s} = 1 $$
must be used. Taking limit as $s \to 0$, we obtain (2) of the following:
\begin{theorem}
	The regularized integral of the product of two unitary Eisenstein series is computed as:
\begin{itemize}
	\item[(1)] If $\pi_1 = \pi(\xi_1,\xi_2), \pi_2 = \pi(\xi_1^{-1}, \xi_2^{-1})$ resp. $\pi_2 = \pi(\xi_2^{-1}, \xi_1^{-1})$ and $\xi_1 \neq \xi_2$, then
	$$ \int_{[\PGL_2]}^{\reg} \eis(0,f_1) \eis(0,f_2) = \frac{2\lambda_{\F}^{(0)}(0)}{\lambda_{\F}^{(-1)}(0)} \ProjP_{\gp{K}}(f_1f_2) - \ProjP_{\gp{K}}(\Intw_0^{(1)}f_1 \cdot \Intw_0 f_2), \quad \text{resp.} $$
	$$ \frac{\lambda_{\F}^{(0)}(0)}{\lambda_{\F}^{(-1)}(0)} (\ProjP_{\gp{K}}(f_1 \Intw_0 f_2) + \ProjP_{\gp{K}}(f_2 \Intw_0 f_1)) - \ProjP_{\gp{K}}(\Intw_0^{(1)}f_1 \cdot f_2). $$
	\item[(2)] If $\pi_1 = \pi(\xi,\xi), \pi_2 = \pi(\xi^{-1},\xi^{-1})$, then
\begin{align*}
	&\quad \int_{[\PGL_2]}^{\reg} \eis^{(1)}(0,f_1) \eis^{(1)}(0,f_2) = \frac{4\lambda_{\F}^{(2)}(0)}{\lambda_{\F}^{-1}(0)} \ProjP_{\gp{K}}(f_1f_2) + \frac{4\lambda_{\F}^{(2)}(0)}{\lambda_{\F}^{-1}(0)} \ProjP_{\gp{K}}(f_1 \cdot \Intw_0^{(1)} f_2) \\
	&\quad + \frac{\lambda_{\F}^{(0)}(0)}{\lambda_{\F}^{-1}(0)} \ProjP_{\gp{K}}(\Intw_0^{(1)}f_1 \cdot \Intw_0^{(1)}f_2) - \frac{1}{3} \ProjP_{\gp{K}}(\Intw_0^{(3)}f_1 \cdot f_2) - \ProjP_{\gp{K}}(\Intw_0^{(2)}f_1 \cdot \Intw_0^{(1)}f_2).
\end{align*}
\end{itemize}
\label{RIPEisUnitary}
\end{theorem}
\begin{remark}
	\emph{It is possible to get formulas for all derivatives, exploiting more the relation $\Intw_s \circ \Intw_{-s} = 1$. Since we don't have applications of these formulas, we do not include them here.}
\end{remark}

\section{Solution of the Puzzle}

	$I(s)$ has a similar regroupment to that of $\tilde{I}(s)$ since this regroupment is for the kernel function unrelated to the Eisenstein series. Recall that we write the terms for $I(s)$ by simply dropping the tilde from the counterpart terms for $\tilde{I}(s)$. In particular, we have $I_{\chi}(s,i\tau)$ coming from the continuous spectrum
\begin{equation} 
	I_{\chi}(s,i\tau) = - \sum_{f \in \Bas_{i\tau}(\chi)} \int_{\gp{Z}(\ag{A}) \gp{N}(\ag{A}) \backslash \GL_2(\ag{A})} (\pi_{\chi}(i\tau)(\Psi) W_f)(x) \overline{W_f(x)} e_s(x) dx,
\label{JZPInt}
\end{equation}
where $\Bas_{i\tau}(\chi)$ is an orthonormal basis in the induced model of
	$$ V_{\chi}(i\tau) = \Ind_{\gp{B}(\ag{A})}^{\GL_2(\ag{A})} (\chi \norm_{\ag{A}}^{i\tau}, \omega \chi^{-1} \norm_{\ag{A}}^{-i\tau}), $$
and $W_f$ is the Whittaker function associated with $f$, defined via analytic continuation by the formula
	$$ W_f(g) := \int_{\ag{A}} f(wn(x)g) \psi(-x) dx. $$
	We need to convert the right hand side of (\ref{JZPInt}) as functionals on the induced model. More precisely, if we write
	$$ f_1 = \pi_{\chi}(i\tau)(\Psi) f, \quad f_2 = f $$
we need to write the integral on the right hand side of (\ref{JZPInt}) as functionals in $f_1,f_2$. But if we produce Eisenstein series from flat sections based on $f_1,f_2$
	$$ f_{1,s} \in V_{\chi}(i\tau+s), \quad f_{2,s} \in V_{\chi}(i\tau+s), $$
we recognize the inner integral in (\ref{JZPInt}) as
	$$ \int_{\gp{B}(\F)\gp{Z}(\ag{A}) \backslash \GL_2(\ag{A})} \left( \left(\eis(0,f_1) \overline{ \eis(0,f_2) } \right)_{\gp{N}} - \eisCst(0,f_1) \cdot \overline{\eisCst(0,f_2)} \right)(x) \cdot e_s(x) dx, $$
which is simply $R(s,\varphi)$ defined in Definition \ref{RegFuncDef} for the \emph{finitely regularizable function} (Definition \ref{FinRegFuncDef})
	$$ \varphi := \eis(0,f_1) \cdot \overline{ \eis(0,f_2) }. $$
In particular, Theorem \ref{RIPEisUnitary} (1) implies
\begin{align*}
	&\quad \Res_{s=1/2} R(s, \varphi) \\
	&= \Res_{s=1/2} \int_{\gp{B}(\F)\gp{Z}(\ag{A}) \backslash \GL_2(\ag{A})} \left( \left(\eis(0,f_1) \overline{ \eis(0,f_2) } \right)_{\gp{N}} - \eisCst(0,f_1) \cdot \overline{\eisCst(0,f_2)} \right)(x) \cdot e_s(x) dx \\
	&= 2 \lambda_{\F}^{(0)}(0) \int_{\gp{K}} f_1(\kappa) \overline{f_2(\kappa)} d\kappa - \lambda_{\F}^{(-1)}(0) \int_{\gp{K}} \left( \Intw_{\chi}'(i\tau)f_1 \right)(\kappa) \overline{ \left( \Intw_{\chi}(i\tau)f_2 \right)(\kappa) } d\kappa,
\end{align*}
where $\Intw_{\chi}(s)$ is the usual intertwining operator
	$$ \Intw_{\chi}(s): V_{\chi}(s) \to V_{\omega\chi^{-1}}(-s), $$
and $\Intw_{\chi}'(s)$ is its derivative with respect to $s$.

	We can also obtain the order $2$ part of $R(s,\varphi)$ as follows. We rewrite the \emph{fundamental identity of regularized integral} in Theorem \ref{AdelicRegThm} (2) as
\begin{align*}
	&\quad R(s,\varphi) + h_T(s) + \frac{\Lambda_{\F}(1-2s)}{\Lambda_{\F}(1+2s)} h_T(-s) \\
	&= \int_{[\PGL_2]} \varphi(g) \Lambda^T \eis(s,f_0)(g) dg + \\
	&\quad \int_T^{\infty} (a(t,\varphi)-f(t)) t^{s-\frac{1}{2}} \frac{dt}{t} + \frac{\Lambda_{\F}(1-2s)}{\Lambda_{\F}(1+2s)} \int_T^{\infty} (a(t,\varphi)-f(t)) t^{-s-\frac{1}{2}} \frac{dt}{t}.
\end{align*}
	With finitely many exceptions of $\tau$, we have
\begin{align*}
	h_T(-(s+1/2)) &\sim_0 -\frac{1}{s} \left\{ \int_{\gp{K}} f_1(\kappa) \overline{f_2(\kappa)} d\kappa + \int_{\gp{K}} \left( \Intw_{\chi}(i\tau)f_1 \right)(\kappa) \overline{ \left( \Intw_{\chi}(i\tau)f_2 \right)(\kappa) } d\kappa \right\} \\
	&= -\frac{2}{s} \Pairing{f_1}{f_2};
\end{align*}
	$h_T(s+1/2)$ is holomorphic at $s=0$; the two integrals on the right hand side are absolutely convergent for any $s$ hence entire in $s$. If we choose a fundamental domain $\mathcal{D}$ and its truncation $\mathcal{D}_T$ at height $T$ as in the proof of Theorem \ref{AdelicRegThm} (4), we see that in
	$$ \int_{[\PGL_2]} \varphi(g) \Lambda^T \eis(s,f_0)(g) dg = \int_{\mathcal{D}_T} \varphi(g) \eis(s,f_0)(g) dg + \int_{\mathcal{D} - \mathcal{D}_T} \varphi(g) (\eis(s,f_0)(g) - \eisCst(s,f_0)(g)) dg, $$
the second integral on the right hand side defines a holomorphic function in $s$ due to the rapid decay of the integrand for any $s$, while the first integral defines a meromorphic function in $s$ admitting a potential pole at $s=1/2$ of order $1$. Hence the only contributor to the order $2$ part of $R(s,\varphi)$ is
	$$ \frac{\Lambda_{\F}(1-2s)}{\Lambda_{\F}(1+2s)} h_T(-s) = \frac{1}{(s-1/2)^2} \cdot \frac{\Lambda_{\F}^{(-1)}(0)}{\Lambda_{\F}(2)} \Pairing{f_1}{f_2} + \cdots. $$
	
	Consequently we identify the principal part of $I_{\chi}(s,i\tau)$ as
\begin{align*}
	I_{\chi}(s,i\tau) &= \left( \frac{1}{(s-1/2)^2} \cdot \frac{\Lambda_{\F}^{(-1)}(0)}{\Lambda_{\F}(2)} - \frac{2\lambda_{\F}^{(0)}(0)}{s-1/2} \right) \cdot \Tr(\pi_{\chi}(i\tau)(\Psi)) \\
	&\quad + \frac{\lambda_{\F}^{(-1)}(0)}{s-1/2} \Tr(\Intw_{\chi}(-i\tau) \Intw_{\chi}'(i\tau) \pi_{\chi}(i\tau)(\Psi)) + O(1).
\end{align*}
	The relation between $\eis(s,x)$ and $\eis(s,\Phi)(x)$ given in (\ref{SecRel}) yields
	$$ \tilde{I}_{\chi}(s, i\tau) = \Dis_{\F}^{-1-s} \cdot (2\pi)^{r_2} \cdot \Lambda_{\F}(1+2s) \cdot I_{\chi}(s, i\tau). $$
	Its principal part is thus given by
\begin{align*}
	\frac{ \tilde{I}_{\chi}(s,i\tau) }{ \Dis_{\F}^{-3/2} \cdot (2\pi)^{r_2} \cdot \Lambda_{\F}(2) } &= \Tr(\pi_{\chi}(i\tau)(\Psi)) \cdot \left\{ \frac{1}{(s-1/2)^2} \cdot \frac{\Lambda_{\F}^{(-1)}(0)}{\Lambda_{\F}(2)} \right. \\
	&\quad \left.+ \frac{1}{s-1/2} \left( -2\lambda_{\F}^{(0)}(0) + \frac{\Lambda_{\F}^{(-1)}(0)}{\Lambda_{\F}(2)} \left( - \log \Dis_{\F} + \frac{2\Lambda_{\F}^{(1)}(2)}{\Lambda_{\F}(2)} \right) \right) \right\} \cdot  \\
	&\quad + \frac{\lambda_{\F}^{(-1)}(0)}{s-1/2} \Tr(\Intw_{\chi}(-i\tau) \Intw_{\chi}'(i\tau) \pi_{\chi}(i\tau)(\Psi)) + O(1).
\end{align*}
	From the definition (\ref{lambdaFDef}) of $\lambda_{\F}(s)$, it is easy to calculate
	$$ \lambda_{\F}^{(-1)}(0) = -\frac{\Lambda_{\F}^{(-1)}(0)}{2\Lambda_{\F}(2)} = \frac{\Lambda_{\F}^{(-1)}(1)}{2\Lambda_{\F}(2)}, \qquad \lambda_{\F}^{(0)}(0) = \frac{\Lambda_{\F}^{(1)}(2) \Lambda_{\F}^{(-1)}(0)}{\Lambda_{\F}(2)^2} + \frac{\Lambda_{\F}^{(0)}(0)}{\Lambda_{\F}(2)}. $$
	Together with the obvious equality
	$$ \Lambda_{\F}^{(-1)}(1) = \Dis_{\F}^{1/2} \cdot (2\pi)^{-r_2} \cdot \zeta_{\F}^*, $$
	we simplify and get
\begin{align*}
	\tilde{I}_{\chi}(s,i\tau) &= \Tr(\pi_{\chi}(i\tau)(\Psi)) \cdot \left\{ - \frac{\Dis_{\F}^{-1} \zeta_{\F}^*}{(s-1/2)^2} - \frac{\Lambda_{\F}^{(-1)}(0) \log \Dis_{\F} + 2 \Lambda_{\F}^{(0)}(0)}{s-1/2} \cdot \Dis_{\F}^{-3/2} \cdot (2\pi)^{r_2} \right\}  \\
	&\quad + \frac{\Dis_{\F}^{-1} \zeta_{\F}^*}{2(s-1/2)} \Tr(\Intw_{\chi}(-i\tau) \Intw_{\chi}'(i\tau) \pi_{\chi}(i\tau)(\Psi)) + O(1).
\end{align*}
	Recall and compute the functionals $B(\Phi)$ and $C(\Phi)$ \cite[p.42]{JZ87} for our chosen $\Phi$ given by (\ref{SphPhi})
	$$ C(\Phi) = \int_{\ag{A}^2} \Phi(x,y) dx dy = \Dis_{\F}^{-1}, $$
\begin{align*}
	B(\Phi) &= \mathrm{f.p.} \left( t \mapsto \int_{\ag{A}} \Phi(t,u) du \right) = \left. \frac{\partial}{\partial s} \right|_{s=1} (s-1) \int_{\ag{A}^{\times}} \left( \int_{\ag{A}} \Phi(t,u) du \right) \norm[t]_{\ag{A}}^s d^{\times}t \\
	&= \left. \frac{\partial}{\partial s} \right|_{s=1} (s-1) \Dis_{\F}^{-1-s/2} \Lambda_{\F}(s) = \Dis_{\F}^{-3/2} \cdot (2\pi)^{r_2} \cdot \left( \frac{\Lambda_{\F}^{(-1)}(0) \log \Dis_{\F}}{2} + \Lambda_{\F}^{(0)}(0) \right) .
\end{align*}
	Also recall the following functionals on $\Cont_c^{\infty}(\GL_2(\ag{A}), \omega^{-1})$
	$$ \Psi \mapsto A(\Psi) := \int_{\gp{K}} \int_{\ag{A}} \sideset{}{_{\alpha \in \F^{\times}}} \sum \Psi \left( \kappa^{-1} \begin{pmatrix} \alpha & x \\ 0 & 1 \end{pmatrix} \kappa \right) dx d\kappa, $$
	$$ \Psi \mapsto T_1(\Psi) := -\frac{1}{2} \int_{\gp{K}} \int_{\gp{N}(\ag{A})} \sideset{}{_{\substack{ \alpha \in \F^{\times} \\ \alpha \neq 1 }}} \sum \Psi \left( \kappa^{-1} n^{-1} \begin{pmatrix} \alpha & 0 \\ 0 & 1 \end{pmatrix} n \kappa \right) \log \Ht(wn\kappa) dn d\kappa, $$
	$$ \Psi \mapsto T_2(\Psi) := \mathrm{f.p.} \left( a \mapsto \int_{\gp{K}} \Psi \left( \kappa^{-1} \begin{pmatrix} 1 & a \\ 0 & 1 \end{pmatrix} \kappa \right) d\kappa \right). $$
	We finally get
\begin{align*}
	\tilde{I}_{\chi}(s,i\tau) &= \frac{- 2 \Tr(\pi_{\chi}(i\tau)(\Psi))}{\zeta_{\F}^*} \cdot \left\{ \frac{(\zeta_{\F}^*)^2 C(\Phi)}{2(s-1/2)^2} + \frac{\zeta_{\F}^* B(\Phi)}{s-1/2} \right\}  \\
	&\quad + \frac{(\zeta_{\F}^*)^2 C(\Phi)}{2(s-1/2)} \cdot \frac{\Tr(\Intw_{\chi}(-i\tau) \Intw_{\chi}'(i\tau) \pi_{\chi}(i\tau)(\Psi))}{\zeta_{\F}^*} + O(1),
\end{align*}
	and conclude the justification of (\ref{JZPuzzle}) and (\ref{Order2Match}) by defining
	$$ A_{\chi}(\Psi) =  \frac{1}{4\pi} \int_{-\infty}^{\infty} \frac{- 2 \Tr(\pi_{\chi}(i\tau)(\Psi))}{\zeta_{\F}^*} d\tau. $$

\section{Appendix: Bounds of Smooth Eisenstein Series}

	\subsection{General Remarks}
	
	We take the notations and assumptions in \cite{Wu5}. Namely we fix a section $s_{\F}: \ag{R}_+ \to \F^{\times} \backslash \ag{A}^{\times}$ and assume the Hecke characters $\omega, \xi$ to be trivial on the image of $s_{\F}$. We then have the definition of the Eisenstein series $\eis(s,\xi,\omega\xi^{-1};f)$ for $f \in V_{\xi,\omega\xi^{-1}}^{\infty}$.
\begin{remark}
	We will sometimes omit $\xi,\omega\xi^{-1}$ and write $\eis(s,f)$ when it is clear from the context.
\end{remark}
\noindent In \cite{Wu5}, we studied the size of $\eis(s,\xi,\omega\xi^{-1};f)$. For the purpose of the present paper, we need something finer. Precisely, we shall decompose it as
	$$ \eis(s,\xi,\omega\xi^{-1};f) = \eisCst(s,\xi,\omega\xi^{-1};f) + \left( \eis(s,\xi,\omega\xi^{-1};f) - \eisCst(s,\xi,\omega\xi^{-1};f) \right) \text{ with} $$
	$$ \eisCst(s,\xi,\omega\xi^{-1};f)(g) := \int_{\F \backslash \ag{A}} \eis(s,\xi,\omega\xi^{-1};f)(n(x)g) dx, n(x) = \begin{pmatrix} 1 & x \\ & 1 \end{pmatrix} $$
and study the growth in $g$ of $\eisCst(s,\xi,\omega\xi^{-1};f)$ and $\eis(s,\xi,\omega\xi^{-1};f) - \eisCst(s,\xi,\omega\xi^{-1};f)$ separately, as well as all their derivatives with respect to $s$.

	The study of the constant term is reduced to the study of the intertwining operator, which is already done in \cite{Wu5}. We focus on
	$$ \eis(s,\xi,\omega\xi^{-1};f)(g) - \eisCst(s,\xi,\omega\xi^{-1};f)(g) = \sum_{\alpha \in \F^{\times}} W(s,\xi,\omega\xi^{-1};f)(a(\alpha)g) \text{ with } a(\alpha)=\begin{pmatrix} \alpha & \\ & 1 \end{pmatrix} $$
	$$ \text{and} \quad W(s,\xi,\omega\xi^{-1};f)(g) = \int_{\F \backslash \ag{A}} \psi(-x) \eis(s,\xi,\omega\xi^{-1};f)(n(x)g) dx, $$
where $\psi$ is the standard additive character of $\F \backslash \ag{A}$. We are thus reduced to the study of the Whittaker functions $W(s,\xi,\omega\xi^{-1};f)$. If we were only interested in $W(s,\xi,\omega\xi^{-1};f)$ itself, then its behavior is already completely clear by \cite{J04} or more generally with the ``singular'' cases by \cite[Proposition 2.2]{JS90}. However, we need a bit more for our purpose in this paper. Namely, we also need to estimate $\frac{\partial^n}{\partial s^n} W(s,\xi,\omega\xi^{-1};f)(g)$. Then not the results of \textit{loc.cit.} but the method serves, i.e., the method of integral representation of Whittaker functions.

	If $\Phi \in \Sch(\ag{A}^2)$ is a Schwartz function, we can define the following (non flat) section in $V_{s,\xi,\omega\xi^{-1}}^{\infty}$
	$$ f_{\Phi}(s,\xi,\omega\xi^{-1};g) = \xi(\det g)\norm[\det g]_{\ag{A}}^{\frac{1}{2}+s} \int_{\ag{A}^{\times}} \Phi((0,t)g) \omega^{-1}\xi^2(t) \norm[t]_{\ag{A}}^{1+2s} d^{\times}t $$
first defined for $\Re s > 0$ then meromorphically continued to $s \in \ag{C}$. Given $f \in V_{0,\xi,\omega\xi^{-1}}^{\infty}$, we want to give an explicit $\Phi$ associated with $f$. For simplicity of notations, we may assume $f$ to be a pure tensor. We then construct $\Phi = \otimes_v' \Phi_v$ place by place:

\noindent (1) At $\F_v = \ag{C}$ resp. $\F_v = \ag{R}$ and for $f_v$ spherical resp. not spherical, we choose $\Phi_v$ using the construction in \cite[Lemma 3.5 (1)]{Wu5} resp. \cite[Lemma 3.8 (1)]{Wu5} for spherical resp. smooth functions.

\noindent (2) At $v < \infty$ and for $f_v$ not spherical, we choose $\Phi_v$ by
	$$ \Phi_v((0,1) \kappa) = \Cond(\psi_v)\xi_v(\det \kappa)^{-1}f_v(\kappa), \kappa \in \gp{K}_v \text{ i.e.} $$
	$$ \Phi_v(ux,u) = \Cond(\psi_v)\omega_v \xi_v^2(u) f_v \begin{pmatrix} 1 & \\ x & 1 \end{pmatrix}, \forall u \in \vo_v^{\times}, x \in \vo_v; $$
	$$ \Phi_v(u,uy) = \Cond(\psi_v)\omega_v \xi_v^2(u) f_v \begin{pmatrix} & -1 \\ 1 & y \end{pmatrix}, \forall u \in \vo_v^{\times}, y \in \varpi_v \vo_v; $$
	and $\Phi_v(x,y)=0$ for $\max(\norm[x]_v,\norm[y]_v) \neq 1$.
	
\noindent (3) At $v < \infty$ and for $f_v$ spherical, we choose $\Phi_v$ by
	$$ \Phi_v = \Cond(\psi_v) \cdot f_v(1) \cdot 1_{\vo_v^2}. $$

Let $S=S(f)$ be the set of places $v$ such that $f_v$ is not spherical. Then we get
\begin{equation}
	f_{\Phi}(s,\xi,\omega\xi^{-1};g) = \Dis(\F)^{\frac{1}{2}}\Lambda^S(1+2s,\omega^{-1}\xi^2) \prod_{\substack{v \in S \\ v \mid \infty}} K_{v,a_v}(s,\omega_v^{-1}\xi_v^2) \cdot f_s(g).
\label{RelSec}
\end{equation}
\noindent We can thus deduce the bounds of $W(s,\xi,\omega\xi^{-1};f)$ from those of
\begin{equation}
	W_{\Phi}(s,\xi,\omega\xi^{-1}; g) = \xi(\det g)\norm[\det g]_{\ag{A}}^{\frac{1}{2}+s} \int_{\ag{A}^{\times}} \Four[2]{\rpR(g).\Phi}(t,\frac{1}{t}) \omega^{-1}\xi^2(t)\norm[t]_{\ag{A}}^{2s} d^{\times}t,
\label{SchWhi}
\end{equation}
where the partial Fourier transforms are defined as in \cite[(3.3)]{Wu5}.

	In Section 5.2, we will bound (\ref{SchWhi}) locally place by place. We then use the obtained bound to get a bound for the sum $\Sigma_{\alpha \in \F^{\times}} \norm[W_{\Phi}(s,\xi,\omega\xi^{-1}; a(\alpha)g)]$, using a convergence lemma treated in Section 4.4. We will treat all bounds with uniformity for $s$ with real part lying in any compact interval, so that the bounds for the derivatives in $s$ follow automatically by Cauchy's integral formulae.
	
	In Section 5.3, we will determine the behavior of the constant term based on \cite{Wu5}.

	\subsection{Bounds of Non Constant Terms}
	
		\subsubsection{Archimedean Places}
		
	We omit the subscript $v$ since we work locally. The local integral representation has the form
	$$ W_{\Phi}(s,\xi,\omega\xi^{-1}; a(y)\kappa) = \omega\xi^{-1}(y)\norm[y]_{\F}^{\frac{1}{2}-s} \int_{\F^{\times}} \Four[2]{\rpR(\kappa).\Phi}(t,\frac{y}{t}) \omega^{-1}\xi^2(t)\norm[t]_{\F}^{2s} d^{\times}t, \text{ for } y \in \F^{\times}, \kappa \in \gp{K}. $$
We are thus reduced to studying the integral at the right hand side. By \cite[Proposition 4.1]{J04} as well as its counterpart in the singular cases, it is easy to see the rapid decay at $\infty$ of
	$$ \norm[W_{\Phi}(s,\xi,\omega\xi^{-1}; a(y)\kappa)] \ll \norm[y]_{\F}^{-N}, \forall N \in \ag{N}, $$
and the polynomial increase at $0$ of
	$$ \norm[W_{\Phi}(s,\xi,\omega\xi^{-1}; a(y)\kappa)] \ll_{\epsilon} \norm[y]_{\F}^{\frac{1}{2} - \norm[\Re s] - \epsilon}, \forall \epsilon > 0. $$
As for the implied constants in the above estimations, one naturally guess it is related to the Schwartz norms of $\Four[2]{\rpR(\kappa).\Phi}$. Then we need to related these norms to the Schwartz norms of $\Phi$ itself. According to this strategy, we state the following two lemmas and the desired proposition.
\begin{lemma}
	For any Schwartz norm $\Sch^*$ there is a Schwartz norm $\Sch^{**}$ such that
	$$ \sup_{\kappa \in \gp{K}} \Sch^*(\Four[2]{\rpR(\kappa).\Phi}) \ll \Sch^{**}(\Phi). $$
\label{SchNormEquivA}
\end{lemma}
\begin{lemma}
	For the real part of $s$ lying in a fixed compact interval, any Schwartz function $\Phi \in \Sch(\F^2)$ and any integer $N \in \ag{N}$, there is a Schwartz norm $\Sch^*$ such that as $\norm[y]_{\F} \to \infty$
	$$ \extnorm{  \int_{\F^{\times}} \Phi(t,\frac{y}{t}) \omega^{-1}\xi^2(t)\norm[t]_{\F}^{2s} d^{\times}t } \ll \Sch^*(\Phi) \norm[y]_{\F}^{-N}; $$
while for any $\epsilon > 0$ there is a Schwartz norm $\Sch^{**}$ such that as $\norm[y]_{\F} \to 0$
	$$ \extnorm{  \int_{\F^{\times}} \Phi(t,\frac{y}{t}) \omega^{-1}\xi^2(t)\norm[t]_{\F}^{2s} d^{\times}t } \ll_{\epsilon} \Sch^{**}(\Phi) \max( \norm[y]_{\F}^{-\epsilon}, \norm[y]_{\F}^{-2 \Re s - \epsilon}). $$
\label{IntBdA}
\end{lemma}
\begin{proposition}
	Let the real part of $s$ vary in a fixed compact interval. For any integer $N \in \ag{N}$, as $\norm[y] \to \infty$ and uniformly in $\kappa$, there is a Schwartz norm $\Sch^*$ such that
	$$ \norm[W_{\Phi}(s,\xi,\omega\xi^{-1}; a(y)\kappa)] \ll \Sch^*(\Phi) \norm[y]_{\F}^{-N}; $$
while for any $\epsilon > 0$, as $\norm[y] \to 0$ and uniformly in $\kappa$, there is a Schwartz norm $\Sch^{**}$ such that
	$$ \norm[W_{\Phi}(s,\xi,\omega\xi^{-1}; a(y)\kappa)] \ll_{\epsilon} \Sch^{**}(\Phi) \norm[y]_{\F}^{\frac{1}{2} - \norm[\Re s] - \epsilon}. $$
\label{LocWhiBdA}
\end{proposition}
\noindent We recall the definition of Schwartz norms on $\ag{R}^d$ for positive integers $d$.
\begin{definition}
	For $l \in [1,\infty]$, $\vec{p}, \vec{m} \in \ag{N}^d$, we put the semi-norm $\Sch_l^{\vec{p},\vec{m}}$ on $\Sch(\ag{R}^d)$ by
	$$ \Sch_l^{\vec{p},\vec{m}}(\Phi) = \extNorm{ \vec{x}^{\vec{p}} \cdot \partial^{\vec{m}} \Phi }_l. $$
Here we have written:
\begin{itemize}
	\item $\Norm_l$ is the $\intL^l$-norm on $\ag{R}^d$.
	\item For $\vec{x}=(x_i)_{1 \leq i \leq d} \in \ag{R}^d, \vec{p}=(p_i)_{1 \leq i \leq d} \in \ag{N}^d$, $\vec{x}^{\vec{p}}=\Pi_{i=1}^d x_i^{p_i}$.
	\item For $\vec{n}=(n_i)_{1 \leq i \leq d} \in \ag{N}^d, \vec{x}=(x_i)_{1 \leq i \leq d} \in \ag{R}^d$,
	$$ \partial^{\vec{n}} = \prod_{i=1}^d \frac{\partial^{n_i}}{\partial x_i^{n_i}}. $$
\end{itemize}
\end{definition}
\begin{remark}
	Since $\ag{C} \simeq \ag{R}^2$, we put the semi-norms for $\Sch(\ag{R}^2)$ on $\Sch(\ag{C})$.
\end{remark}
\begin{remark}
	If we do not specify the parameters of a Schwartz norm $\Sch^*$, we mean the max of a finite collection of Schwartz norms. This applies to Lemma \ref{SchNormEquivA}, \ref{IntBdA} and Proposition \ref{LocWhiBdA}.
\end{remark}

	We first treat Lemma \ref{SchNormEquivA}.
\begin{proposition}
	The topology on $\Sch(\ag{R}^d)$ defined by the system of semi-norms $\Sch_l^*$ does not depend on $l \in [1,\infty]$.
\end{proposition}
\begin{proof}
	In the case $d=1$, we have for $l \in [1,\infty)$ and any $\Phi \in \Sch(\ag{R})$
	$$ \int_{\ag{R}} \norm[\Phi(x)]^l dx \leq \Sch_{\infty}^{0,0}(\Phi)^l \int_{-1}^1 dx + \Sch_{\infty}^{2,0}(\Phi)^l \int_{\norm[x]>1} \norm[x]^{-2l} dx, $$
from which we deduce by replacing $\Phi(x)$ with $x^p \Phi^{(m)}(x)$ that
	$$ \Sch_l^{p,m}(\Phi) \ll_l \Sch_{\infty}^{p,m}(\Phi) + \Sch_{\infty}^{p+2,m}(\Phi). $$
	In the opposite direction, from H\"older inequality
	$$ \norm[\Phi(y) - \Phi(x)] = \extnorm{ \int_x^y \Phi'(t) dt} \leq \Norm[\Phi']_l \cdot \left( \int_x^y dt \right)^{\frac{1-l}{l}} $$
and $(a+b)^l \leq 2^{l-1}(a^l + b^l)$ we deduce
	$$ \norm[\Phi(y)]^l \leq 2^{l-1} \left( \norm[\Phi(x)]^l + \Norm[\Phi']_l^l \cdot \norm[y-x]^{l-1} \right). $$
	Integrating both sides against $\min(1,\norm[x-y]^{-l-1})dx \leq dx$ gives
	$$ (2+\frac{2}{l}) \norm[\Phi(y)]^l \leq 2^{l-1} \left( \Norm[\Phi]_l^l + \Norm[\Phi']_l^l \cdot \int_{\ag{R}} \min(\norm[x]^{l-1}, \norm[x]^{-2}) dx \right). $$
	Hence we get (a Sobolev inequality) and conclude the case $d=1$ by
	$$ \Norm[\Phi]_{\infty} \ll_l \Norm[\Phi]_l + \Norm[\Phi']_l. $$
	
	For general $d$, one deduces easily by induction
	$$ \Sch_l^{\vec{p},\vec{m}}(\Phi)^l \ll_{l,d} \sum_{\vec{\epsilon} \in \{ 0,2 \}^d} \Sch_{\infty}^{\vec{p}+\vec{\epsilon},\vec{m}}(\Phi)^l, $$
	$$ \Norm[\Phi]_{\infty}^l \ll_{l,d} \sum_{\vec{\epsilon} \in \{ 0,1 \}^d} \Sch_l^{0,\vec{\epsilon}}(\Phi)^l. $$
\end{proof}
\begin{proof}{(of Lemma \ref{SchNormEquivA})}
	By the above proposition, the problem is reduced to the uniform continuity of
	$$ \Four[2]{\cdot} \circ \rpR(\kappa): \Sch(\F^2) \to \Sch(\F^2) $$
with respect to $\kappa \in \gp{K}$. The continuity of $\Four[2]{\cdot}$ follows by considering the $\Sch_2^*$ semi-norms. The uniform continuity of $\rpR(\kappa)$ follows by considering the $\Sch_{\infty}^*$ semi-norms.
\end{proof}

	We then turn to Lemma \ref{IntBdA}. Actually, we are going to reduce to the situation of Mellin transform on $\ag{R}_+$, which we shall study at the first place. For any $c \in \ag{R}$, define
	$$ \fsB_c(\ag{R}_+) = \left\{ f: \in \Cont^{\infty}(\ag{R}_+): \sup_{y > 0} \norm[f^{(k)}(y) y^{\sigma+k}] < \infty, \forall \sigma > c \right\}. $$
	$$ \fsH_c(\ag{C}) = \left\{ M \text{ holomorphic in } \Re s > c: \sup_{\Re s = \sigma} \norm[s(s+1)\cdots (s+k-1)M(s)] < \infty, \forall \sigma > c \right\}. $$
\begin{definition}
	For any fixed $l \in [0,\infty]$, we put a system of semi-norms $B_l^{k,\sigma}$ with $k \in \ag{N}, \sigma \in (c,\infty)$ on $\fsB_c(\ag{R}_+)$ by
	$$ B_l^{k,\sigma}(f) = \left( \int_0^{\infty} \norm[f^{(k)}(y) y^{\sigma+k}]^l \frac{dy}{y} \right)^{\frac{1}{l}}, l \neq \infty; B_{\infty}^{k,\sigma}(f) = \sup_{y > 0} \norm[f^{(k)}(y) y^{\sigma+k}]. $$
\end{definition}
\begin{definition}
	For any fixed $l \in [0,\infty]$, we put a system of semi-norms $H_l^{k,\sigma}$ with $k \in \ag{N}, \sigma \in (c,\infty)$ on $\fsH_c(\ag{C})$ by
	$$ H_l^{k,\sigma}(M) = \left( \int_{\Re s = \sigma} \norm[s(s+1)\cdots (s+k-1)M(s)]^l \frac{ds}{2\pi i} \right)^{\frac{1}{l}}, l \neq \infty; $$
	$$ H_{\infty}^{k,\sigma}(M) = \sup_{\Re s = \sigma} \norm[s(s+1)\cdots (s+k-1)M(s)]. $$
\end{definition}
\begin{proposition}
	The topology on $\fsB_c(\ag{R}_+)$ defined by $B_l^{k,\sigma}$ does not depend on $l$. More precisely, for any $f \in \fsB_c(\ag{R}_+)$ we have for $1 \leq l < \infty$ and $\epsilon > 0$ small with $\sigma - \epsilon > c$
	$$ B_l^{k,\sigma}(f) \ll_{\epsilon, l} B_{\infty}^{k,\sigma + \epsilon}(f) + B_{\infty}^{k,\sigma - \epsilon}(f); $$
	$$ B_{\infty}^{k,\sigma}(f) \ll_{\epsilon, l, \sigma + k} B_l^{k,\sigma }(f) + B_l^{k+1,\sigma}(f). $$
\end{proposition}
\begin{proof}
	The first inequality follows from
	$$ \int_0^{\infty} \norm[f^{(k)}(y) y^{\sigma+k}]^l \frac{dy}{y} \leq B_{\infty}^{k,\sigma + \epsilon}(f)^l \int_1^{\infty} y^{-\epsilon l} \frac{dy}{y} + B_{\infty}^{k,\sigma - \epsilon}(f)^l \int_0^1 y^{\epsilon l} \frac{dy}{y}. $$
	For the second inequality, we first note that for any $x,y > 0$
	$$ f^{(k)}(y)y^{\sigma+k} - f^{(k)}(x)x^{\sigma+k} = \int_y^x f^{(k+1)}(u) u^{\sigma+k+1} \frac{du}{u} + (\sigma + k) \int_y^x f^{(k)}(u) u^{\sigma+k} \frac{du}{u}. $$
	We can bound the integrals using H\"older inequality as
	$$ \extnorm{\int_y^x f^{(k)}(u) u^{\sigma+k} \frac{du}{u}} \leq B_l^{k,\sigma}(f) \cdot \norm[\log (y/x)]^{\frac{l-1}{l}}, $$
from which we deduce
	$$ \norm[f^{(k)}(y)y^{\sigma+k}] \leq \norm[f^{(k)}(x)x^{\sigma+k}] + \left[ B_l^{k+1,\sigma}(f) + \norm[\sigma+k] \cdot B_l^{k,\sigma}(f) \right] \cdot \norm[\log (y/x)]^{\frac{l-1}{l}}. $$
	Raising to the power $l \geq 1$ and use $(a+b)^l \leq 2^{l-1}(a^l + b^l)$ gives
	$$ \norm[f^{(k)}(y)y^{\sigma+k}]^l \leq 2^{l-1} \left\{ \norm[f^{(k)}(x)x^{\sigma+k}]^l + \left[ B_l^{k+1,\sigma}(f) + \norm[\sigma+k] \cdot B_l^{k,\sigma}(f) \right]^l \cdot \norm[\log (y/x)]^{l-1} \right\}. $$
	Integrating both sides against $\min((x/y)^{\epsilon l}, (x/y)^{-\epsilon l}) dx/x \leq dx/x$ gives
	$$ \frac{2}{\epsilon l} \cdot \norm[f^{(k)}(y)y^{\sigma+k}]^l \leq 2^{l-1} \left\{ B_l^{k,\sigma}(f)^l + \left[ B_l^{k+1,\sigma}(f) + \norm[\sigma+k] \cdot B_l^{k,\sigma}(f) \right]^l \cdot \int_0^{\infty} \min(x^{\epsilon},x^{-\epsilon}) \norm[\log x]^{l-1} \frac{dx}{x} \right\}. $$
	We conclude since $\int_0^{\infty} \min(x^{\epsilon},x^{-\epsilon}) \norm[\log x]^{l-1} \frac{dx}{x} < \infty$.
\end{proof}
\begin{proposition}
	The topology on $\fsH_c(\ag{C})$ defined by $H_l^{k,\sigma}$ does not depend on $l$. More precisely, for any $M \in \fsH_c(\ag{C})$ we have for $1 \leq l < \infty$ and $\epsilon > 0$ small with $\sigma - \epsilon > c$
	$$ H_l^{k,\sigma}(M) \ll_{k, l} H_{\infty}^{k,\sigma}(M) + H_{\infty}^{k+2,\sigma}(M); $$
	$$ H_{\infty}^{k,\sigma}(M) \ll_{\epsilon, l, \sigma + k} H_l^{k,\sigma+\epsilon}(M) + H_l^{k+1,\sigma+\epsilon}(M) + H_l^{k,\sigma-\epsilon}(M) + H_l^{k+1,\sigma-\epsilon}(M). $$
\end{proposition}
\begin{proof}
	The first inequality follows from
\begin{align*}
	\int_{\Re s = \sigma} \extnorm{s(s+1)\cdots (s+k-1)M(s)}^l \frac{ds}{2\pi i} &\leq H_{\infty}^{k,\sigma}(M)^l \int_{\substack{\Re s = \sigma \\ \norm[\Im s] \leq 1}} \frac{ds}{2\pi i} \\
	&\quad + H_{\infty}^{k+2,\sigma}(M)^l \int_{\substack{\Re s = \sigma \\ \norm[\Im s] > 1}} \frac{1}{\norm[(s+k)(s+k+1)]^l} \frac{ds}{2\pi i}.
\end{align*}
	For the second inequality, we first note that for any $s_0$ with $\Re s_0 = \sigma$
\begin{align*}
	s_0(s_0+1)\cdots (s_0+k-1)M(s_0) &= \int_{\Re s = \sigma + \epsilon} \frac{s(s+1)\cdots (s+k-1)M(s)}{s-s_0} \frac{ds}{2\pi i} \\
	&\quad - \int_{\Re s = \sigma - \epsilon} \frac{s(s+1)\cdots (s+k-1)M(s)}{s-s_0} \frac{ds}{2\pi i}.
\end{align*}
	To bound the integrals, we apply H\"older inequality to get
\begin{align*}
	\int_{\Re s = \sigma + \epsilon} \extnorm{\frac{s(s+1)\cdots (s+k-1)M(s)}{s-s_0}} \frac{ds}{2\pi i} &\leq \frac{H_l^{k+1,\sigma+\epsilon}(M)}{\epsilon} \cdot \left( \int_{{\substack{\Re s = \sigma+\epsilon \\ \norm[\Im s] \geq 1}}} \frac{1}{\norm[s+k]^{\frac{l}{l-1}}} \frac{ds}{2\pi i} \right)^{\frac{l-1}{l}} \\
	&\quad + \frac{H_l^{k,\sigma+\epsilon}(M)}{\epsilon} \cdot \left( \int_{{\substack{\Re s = \sigma+\epsilon \\ \norm[\Im s] \leq 1}}} \frac{ds}{2\pi i} \right)^{\frac{l-1}{l}},
\end{align*}
	and conclude by the similar bound on $\Re s = \sigma - \epsilon$.
\end{proof}
\begin{proposition}
	The two maps
	$$ \fsB_c(\ag{R}_+) \to \fsH_c(\ag{C}), f \mapsto \Mellin{f}(s):= \int_0^{\infty} f(y) y^s \frac{dy}{y}, \text{ for } \Re s > c; $$
	$$ \fsH_c(\ag{C}) \to \fsB_c(\ag{R}_+), M \mapsto f_M(y) := \int_{\Re s = \sigma} M(s)y^{-s} \frac{ds}{2\pi i}, \forall \sigma > c $$
	are continuous with respect to the above topologies defined by semi-norms.
\end{proposition}
\begin{proof}
	By integration by parts we get
	$$ \Mellin{f}(s) = \frac{(-1)^k}{s(s+1) \cdots (s+k-1)} \int_0^{\infty} f^{(k)}(y) y^{s+k} \frac{dy}{y}, $$
from which it follows readily that $H_{\infty}^{k,\sigma}(\Mellin{f}) \leq B_1^{k,\sigma}(f)$. Similarly we pass the derivatives under the integral to get
	$$ f_M^{(k)}(y) = (-1)^k \int_{\Re s = \sigma} s(s+1) \cdots (s+k-1) M(s)y^{-s-k} \frac{ds}{2\pi i}, $$
from which it follows readily that $B_{\infty}^{k,\sigma}(f_M) \leq H_1^{k,\sigma}(M)$.
\end{proof}
\begin{definition}
	We write the multiplicative group $\F^1 = \{ x \in \F: \norm[x]_{\F} = 1 \}$. For any function $f$ on $\F^{\times}$ and any character $\xi \in \widehat{\F^1}$ we define a function on $\ag{R}_+$
	$$ f_{\xi}(t) = f(t;\xi) = \int_{\F^1} f(tu) \xi(u) du, t > 0 $$
where $du$ is the probability Haar measure on $\F^1$. Concretely:
\begin{itemize}
	\item[(1)] If $\F=\ag{R}$ then $\F^1=\{ \pm 1\}, \widehat{\F^1}=\{ \xi_+, \xi_- \}$ with $\xi_+ \equiv 1$ and $\xi_-(-1)=-1$. We then define
	$$ f(t;+) = f_+(t) = \frac{1}{2} \left( f(t) + f(-t) \right), f(t;-) = f_-(t) = \frac{1}{2} \left( f(t) - f(-t) \right). $$
	\item[(2)] If $\F=\ag{C}$ then $\F^1=\{ e^{i \theta}: \theta \in \ag{R}/2\pi \ag{Z} \}, \widehat{\F^1}=\{ \xi_n: n \in \ag{Z} \}$ with $\xi_n(e^{i\theta}) = e^{i n \theta}$. We then define
	$$ f(t;n) = f_n(t) = \int_{\ag{R} / 2\pi \ag{Z}} f(te^{i\theta}) e^{in\theta} \frac{d\theta}{2\pi}. $$
\end{itemize}
\label{FourF^1}
\end{definition}
\begin{proposition}
	$f_{\xi} \in \fsB_0(\ag{R}_+)$ for any $f \in \Sch(\F)$ and $\xi \in \widehat{\F^1}$. The map
	$$ \Sch(\F) \to \fsB_0(\ag{R}_+), f \mapsto f_{\xi} $$
is continuous. Moreover, in the case $\F=\ag{C}$, for any $k,l \in \ag{N}, \sigma > 0$ there is a finite collection of norms $\Sch_{\infty}^*$ independent of $n$ ($n \neq 0$ if $l \neq 0$) such that
	$$ B_{\infty}^{k,\sigma}(f_n) \ll_{k,\sigma} \norm[n]^{-l} \Sch_{\infty}^*(f). $$
\end{proposition}
\begin{proof}
	In the case $\F=\ag{R}$, we have
	$$ t^k\frac{d^k}{dt^k} f(t;+) = \frac{1}{2} \left( t^kf^{(k)}(t) + (-1)^k t^k f^{(k)}(-t) \right), $$
	$$ t^k\frac{d^k}{dt^k} f(t;-) = \frac{1}{2} \left( t^kf^{(k)}(t) + (-1)^{k+1} t^k f^{(k)}(-t) \right), $$
from which it is easy to see
	$$ B_{\infty}^{k,\sigma}(f_{\pm}) \leq \Sch_{\infty}^{\lfloor k+\sigma \rfloor, k}(f) + \Sch_{\infty}^{\lceil k+\sigma \rceil, k}(f), \forall k \in \ag{N}, \sigma > 0. $$
	In the case $\F=\ag{C}$, with the Cartesian \& Polar coordinates
	$$ (x,y) = (t \cos \theta, t \sin \theta), (z,\bar{z})=(x+iy,x-iy) $$
	we have
	$$ \frac{\partial}{\partial \theta} = i \left( z \frac{\partial}{\partial z} - \bar{z} \frac{\partial}{\partial \bar{z}} \right); t\frac{\partial}{\partial t} = z \frac{\partial}{\partial z} + \bar{z} \frac{\partial}{\partial \bar{z}}. $$
	By induction on $k \in \ag{N}$, it is easy to see
	$$ t^k \frac{\partial^k}{\partial t^k} = P_k(z \frac{\partial}{\partial z} + \bar{z} \frac{\partial}{\partial \bar{z}}) $$
for some polynomial $P_k \in \ag{Z}[X]$ and any $k \in \ag{N}$. It follows that
\begin{align*}
	t^k f_n^{(k)}(t) &= \int_{\ag{R} / 2\pi \ag{Z}} (P_k(z \frac{\partial}{\partial z} + \bar{z} \frac{\partial}{\partial \bar{z}})f)(te^{i\theta}) e^{in\theta} \frac{d\theta}{2\pi} \\
	&= \frac{(-1)^l}{n^l} \int_{\ag{R} / 2\pi \ag{Z}} (\left( z \frac{\partial}{\partial z} - \bar{z} \frac{\partial}{\partial \bar{z}} \right)^l P_k(z \frac{\partial}{\partial z} + \bar{z} \frac{\partial}{\partial \bar{z}})f)(te^{i\theta}) e^{in\theta} \frac{d\theta}{2\pi}.
\end{align*}
	Hence, we deduce that
\begin{align*}
	B_{\infty}^{k,\sigma}(f_n) &\leq \norm[n]^{-l} \left\{ \extNorm{(z\bar{z})^{\lfloor \frac{\sigma}{2} \rfloor}\left( z \frac{\partial}{\partial z} - \bar{z} \frac{\partial}{\partial \bar{z}} \right)^l P_k(z \frac{\partial}{\partial z} + \bar{z} \frac{\partial}{\partial \bar{z}})f}_{\infty} \right. \\
	&\quad \left. + \extNorm{(z\bar{z})^{\lceil \frac{\sigma}{2} \rceil}\left( z \frac{\partial}{\partial z} - \bar{z} \frac{\partial}{\partial \bar{z}} \right)^l P_k(z \frac{\partial}{\partial z} + \bar{z} \frac{\partial}{\partial \bar{z}})f}_{\infty} \right\}.
\end{align*}
	The right hand side is obviously bounded by some Schwartz norm of $f$.
\end{proof}
\begin{proof}{(of Lemma \ref{IntBdA})}
	We only treat the case $\F=\ag{C}$, the real case being similar and simpler. Writing
	$$ f(y) = \int_{\ag{C}^{\times}} \Phi(t,\frac{y}{t}) \omega^{-1}\xi^2(t)\norm[t]_{\ag{C}}^{2s} d^{\times}t, $$
we can take its Fourier expansion on $\ag{C}^1$
	$$ f(te^{i\theta}) = \sum_{n \in \ag{Z}} f_n(t) e^{-in\theta}, t \in \ag{R}_+. $$
	Extending each $\xi_n \in \widehat{\ag{C}^1}$ to $\ag{C}^{\times}$ by triviality on $\ag{R}_+$ we have the Mellin transform
\begin{align*}
	\Mellin{f_n}(s_1) &= \int_{\ag{C}^{\times}} f(y) \xi_n(y) \norm[y]_{\ag{C}}^{\frac{s_1}{2}} d^{\times}y \\
	&= \int_{\ag{C}^{\times} \times \ag{C}^{\times}} \Phi(t,y) \omega^{-1}\xi^2\xi_n(t) \norm[t]_{\ag{C}}^{\frac{s_1}{2}+2s} \xi_n(y) \norm[y]_{\ag{C}}^{\frac{s_1}{2}} d^{\times}t d^{\times}y \\
	&= \int_{\ag{R}_+ \times \ag{R}_+} \Phi_{\omega^{-1}\xi^2\xi_n, \xi_n}(t_1,t_2) \norm[t_1]_{\ag{C}}^{\frac{s_1}{2}+2s+i\mu(\omega^{-1}\xi^2)} \norm[t_2]_{\ag{C}}^{\frac{s_1}{2}} d^{\times}t_1 d^{\times}t_2.
\end{align*}
	Considering the $H_{\infty}^*$ semi-norms it is easy to see $\Mellin{f_n} \in \fsH_{\max(0,-4\Re s)}(\ag{C})$. We can also bound
	$$ H_{\infty}^{k,\sigma}(\Mellin{f_n}) \ll_{k,\sigma} \min(1, \norm[n]^{-2}) \Sch_1^*(\Phi), \forall \sigma > \max(0,-4\Re s). $$
	As $\Re s$ lies in a compact interval, the orders of $\Sch_1^*$ can be made uniform (but depends on $\sigma$). Hence $f_n \in \fsB_{\max(0,-4\Re s)}(\ag{R}_+)$ and for any $\sigma > \max(0,-4\Re s)$ we get
	$$ \norm[t^{\sigma}f(te^{i\theta})] \leq \sum_n B_{\infty}^{0,\sigma}(f_n) \ll \Sch_1^*(\Phi) \sum_n \min(1, \norm[n]^{-2}). $$
	We conclude by noting $t^{\sigma} = \norm[te^{i\theta}]_{\ag{C}}^{\sigma/2}$.
\end{proof}

	Obviously, Proposition \ref{LocWhiBdA} is a direct consequence of Lemma \ref{SchNormEquivA} and \ref{IntBdA}.

		\subsubsection{Non Archimedean Places}
		
	We continue to omit the subscript $v$ for simplicity of notations.
\begin{definition}
	Let $d \geq 1$ be an integer. For any $\Phi \in \Sch(\F^d)$ we define its \emph{support index} $D(\Phi) \in \ag{Z}$, \emph{additive invariance index} $\delta(\Phi) \in \ag{Z}$ and \emph{multiplicative invariance index} $m(\Phi) \in \ag{N}$ as follows.
\begin{itemize}
	\item[(1)] $D(\Phi)$ is the largest integer $D$ such that
	$$ \Phi(\vec{x}) \neq 0 \Rightarrow \vec{x} \in \vp^D \times \cdots \times \vp^D. $$
	\item[(2)] $\delta(\Phi)$ is the smallest integer $\delta$ such that
	$$ \Phi(\vec{x}+\vec{t}) = \Phi(\vec{x}), \forall \vec{x} \in \F^d, \vec{t} \in \vp^{\delta} \times \cdots \times \vp^{\delta}. $$
	\item[(3)] $m(\Phi)$ is the smallest integer $m \geq 0$ such that for any $\kappa \in \GL_d(\vo)$ with $\kappa-1 \in \Mat_d(\vp^m)$
	$$ \rpR(\kappa).\Phi(\vec{x}) = \Phi(\vec{x}.\kappa) = \Phi(\vec{x}), \forall \vec{x} \in \F^d. $$
\end{itemize}
\end{definition}
\begin{proposition}
	The three indices satisfy the following relations.
\begin{itemize}
	\item[(0)] $m(\Phi) \leq \delta(\Phi)-D(\Phi)$.
	\item[(1)] For any $\kappa \in \GL_d(\vo)$, we have $D(\rpR(\kappa).\Phi)=D(\Phi)$, $\delta(\rpR(\kappa).\Phi)=\delta(\Phi)$ and $m(\rpR(\kappa).\Phi)=m(\Phi)$.
	\item[(2)] Let $\Four{\cdot}$ denote the Fourier transform
	$$ \Four{\Phi}(\vec{x}) = \int_{\F^d} \Phi(\vec{y}) \psi(-\vec{y}.\vec{x}) d\vec{y}. $$
	Then we have
	$$ D(\Phi)+\delta(\Four{\Phi}) = \delta(\Phi) + D(\Four{\Phi}) = -\cond(\psi). $$
	\item[(3)] More generally, let $I = \{ i_1, \cdots, i_j \} \subset \{ 1, \dots, d \}$. We define the partial Fourier transform $\Four[I]{\cdot} = \Four[i_1]{\Four[i_2]{ \cdots \Four[i_j]{\cdot}}}$. Then we have
	$$ \delta(\Four[I]{\Phi})) \leq \max(\delta(\Phi), -\cond(\psi)-D(\Phi)); $$
	$$ D(\Four[I]{\Phi}) \geq \min(D(\Phi), -\cond(\psi)-\delta(\Phi)). $$
\end{itemize}
\label{DdmRel}
\end{proposition}
\begin{proof}
	(0) and (1) are obvious from definition. (3) follows easily from (2). We prove (2) as follows. From
	$$ \Four{\Phi}(\vec{x}+\vec{t}) = \int_{\vp^{D(\Phi)} \times \cdots \times \vp^{D(\Phi)}} \Phi(\vec{y}) \psi(-\vec{y}.\vec{x}) \psi(-\vec{y}.\vec{t}) d\vec{y}, $$
we see that for $\vec{t} \in \vp^{-\cond(\psi)-D(\Phi)}$, $\vec{y}.\vec{t} \in \vp^{-\cond(\psi)}$ hence $\psi(-\vec{y}.\vec{t})=1$ and $\Four{\Phi}(\vec{x}+\vec{t}) = \Four{\Phi}(\vec{x})$. Thus
	$$ \delta(\Four{\Phi}) \leq -\cond(\psi)-D(\Phi). $$
	On the other hand, if $\vec{x} \notin \vp^{-\cond(\psi)-\delta(\Phi)} \times \cdots \times \vp^{-\cond(\psi)-\delta(\Phi)}$ then at least for one component, say $x_1 \notin \vp^{-\cond(\psi)-\delta(\Phi)}$, i.e., $v(x_1) < -\cond(\psi)-\delta(\Phi)$. As $t_1$ runs under the condition $v(t_1)=-\cond(\psi)-1-v(x_1) \geq \delta(\Phi)$, $x_1t_1$ runs under the condition $v(x_1t_1)=-\cond(\psi)-1$. Hence at least for one $t_1$, $\psi(x_1t_1) \neq 1$. Writing $\vec{t}=(t_1,0,\dots,0)$, we get
	$$ \Four{\Phi}(\vec{x}) = \int_{\F^d} \Phi(\vec{y}+\vec{t}) \psi(-\vec{y}.\vec{x}) d\vec{y} = \int_{\F^d} \Phi(\vec{y}) \psi(-\vec{y}.\vec{x}) d\vec{y} \cdot \psi(\vec{t}.\vec{x}) = \psi(t_1x_1) \Four{\Phi}(\vec{x}). $$
	Hence $\Four{\Phi}(\vec{x}) = 0$ and
	$$ D(\Four{\Phi}) \geq -\cond(\psi)-\delta(\Phi). $$
	Replacing $\Phi$ with $\Four{\Phi}$ in the above argument gives the inequalities in the opposite direction.
\end{proof}
\begin{definition}
	For $l \in [1,\infty]$, $\vec{\sigma} \in \ag{R}_{\geq 0}^d$, we put the semi-norm $\Sch_l^{\vec{\sigma}}$ on $\Sch(\F^d)$ by
	$$ \Sch_l^{\vec{\sigma}}(\Phi) = \extNorm{ \norm[\vec{x}^{\vec{\sigma}}]_{\F} \cdot \Phi }_l. $$
Here we have written:
\begin{itemize}
	\item $\Norm_l$ is the $\intL^l$-norm on $\F^d$.
	\item For $\vec{x}=(x_i)_{1 \leq i \leq d} \in \F^d, \vec{\sigma}=(\sigma_i)_{1 \leq i \leq d} \in \ag{R}_{\geq 0}^d$, $\norm[\vec{x}^{\vec{\sigma}}]_{\F}=\Pi_{i=1}^d \norm[x_i]_{\F}^{\sigma_i}$.
\end{itemize}
	We shall also write $\norm[\vec{\sigma}] = \sum_i \sigma_i$.
\end{definition}
\begin{proposition}
	We have the following relations of norms for any $\Phi \in \Sch(\F^d)$.
\begin{itemize}
	\item[(1)] $\Norm[\Phi]_{\infty} \leq q^{\frac{d\delta(\Phi)}{l}} \Norm[\Phi]_l$ and $\Norm[\Phi]_l \leq q^{-\frac{dD(\Phi)}{l}} \Norm[\Phi]_{\infty}$.
	\item[(2)] $\Sch_l^{\vec{\sigma}}(\Phi) \leq q^{-\norm[\vec{\sigma}] D(\Phi)} \Norm[\Phi]_l$.
\end{itemize}
\label{SchEquiv}
\end{proposition}
\begin{proof}
	It suffices to prove the case $d=1$. Let $x_0 \in \F$ be such that $\norm[\Phi(x_0)] = \Norm[\Phi]_{\infty}$, then
	$$ \Norm[\Phi]_{\infty}^l = \Vol(\vp^{\delta(\Phi)})^{-1} \int_{x_0 + \vp^{\delta(\Phi)}} \norm[\Phi(x)]^l dx \leq q^{\delta(\Phi)} \Norm[\Phi]_l^l $$
and we get the first inequality. The second follows from
	$$ \Norm[\Phi]_l^l \leq \Norm[\Phi]_{\infty}^l \int_{{\rm supp}(\Phi)} dx \leq \Norm[\Phi]_{\infty}^l \cdot \Vol(\vp^{D(\Phi)}) = q^{-D(\Phi)} \Norm[\Phi]_{\infty}^l. $$
	For the last, we deduce it from
	$$ \Sch_l^{\sigma}(\Phi)^l = \int_{\vp^{D(\Phi)}} \norm[x]_{\F}^{\sigma l} \norm[\Phi(x)]^l dx \leq \sup_{x \in \vp^{D(\Phi)}} \norm[x]_{\F}^{\sigma l} \cdot \Norm[\Phi]_l^l = q^{-\sigma l D(\Phi)} \Norm[\Phi]_l^l. $$
\end{proof}
\begin{definition}
	For any $c \in \ag{R}$, we define $\fsB_c(\ag{Z};\varpi)$ to be the space of functions $f: \varpi^{\ag{Z}} \to \ag{C}$ satisfying
\begin{itemize}
	\item[(1)] $\lim_{n \to -\infty} f(\varpi^n) q^{-n\sigma} = 0$ for any $\sigma > 0$.
	\item[(2)] $\lim_{n \to +\infty} f(\varpi^n) q^{-n \sigma} = 0$ for any $\sigma > c$.
\end{itemize}
	The subspace $\fsB_c^0(\ag{Z};\varpi) \subset \fsB_c(\ag{Z};\varpi)$ is defined by replacing (1) with
\begin{itemize}
	\item[(1')] $f(\varpi^n)=0$ for $n \ll -1$.
\end{itemize}
\end{definition}
\begin{definition}
	For any $c \in \ag{R}$ we define $\fsH_c(\ag{C};q)$ to be the space of meromorphic functions $M: \ag{C} \to \ag{C}$ satisfying
\begin{itemize}
	\item[(1)] $M(s+i\frac{2\pi}{\log q}) = M(s)$ for all $s \in \ag{C}$.
	\item[(2)] $M(s)$ is holomorphic for $\Re s > c$. 
\end{itemize}
\end{definition}
\begin{definition}
	For any $l \in [1,\infty]$ we put a system of semi-norms $B_l^{\sigma}$ for $\sigma > c$ on $\fsB_c(\ag{Z};\varpi)$ by
	$$ B_l^{\sigma}(f) = \left( \sum_{n \in \ag{Z}} q^{-n \sigma l} \norm[f(\varpi^n)]^l \right)^{\frac{1}{l}}, l \neq \infty; \quad B_{\infty}^{\sigma}(f) = \sup_{n \in \ag{Z}} q^{-n\sigma} \norm[f(\varpi^n)]. $$
\end{definition}
\begin{definition}
	For any $l \in [1,\infty]$ we put a system of semi-norms $H_l^{\sigma}$ for $\sigma > c$ on $\fsH_c(\ag{C};q)$ by
	$$ H_l^{\sigma}(M) = \left( \int_0^{\frac{2\pi}{\log q}} \norm[M(\sigma+i\tau)]^l \frac{\log q d\tau}{2\pi} \right)^{\frac{1}{l}}, l \neq \infty; \quad H_{\infty}^{\sigma}(M) = \max_{\Re s = \sigma} \norm[M(s)]. $$
\end{definition}
\begin{proposition}
	The topology on $\fsB_c(\ag{Z};\varpi)$ defined by $B_l^*$ does not depend on $l$. More precisely, for any $f \in \fsB_c(\ag{Z};\varpi)$ we have for $1 \leq l < \infty$ and $\epsilon > 0$ small with $\sigma - \epsilon > c$
	$$ B_l^{\sigma}(f) \ll_{\epsilon,l} B_{\infty}^{\sigma-\epsilon}(f) + B_{\infty}^{\sigma+\epsilon}(f); $$
	$$ B_{\infty}^{\sigma}(f) \leq B_l^{\sigma}(f). $$
\label{BEquiv}
\end{proposition}
\begin{proof}
	The first follows from
	$$ B_l^{\sigma}(f)^l \leq B_{\infty}^{\sigma-\epsilon}(f) \sum_{n \geq 0} q^{-n\epsilon} + B_{\infty}^{\sigma+\epsilon}(f) \sum_{n < 0} q^{n\epsilon}. $$
	The second is obvious by positivity.
\end{proof}
\begin{proposition}
	The topology on $\fsH_c(\ag{C};q)$ defined by $H_l^*$ does not depend on $l$. More precisely, for any $M \in \fsH_c(\ag{C};q)$ we have for $1 \leq l < \infty$ and $\epsilon > 0$ small with $\sigma - \epsilon > c$
	$$ H_l^{\sigma}(M) \leq H_{\infty}^{\sigma}(M); $$
	$$ H_{\infty}^{\sigma}(M) \ll_{\epsilon} H_l^{\sigma-\epsilon}(M) + H_l^{\sigma+\epsilon}(M). $$
\label{HEquiv}
\end{proposition}
\begin{proof}
	The first is obvious. For the second, let $s_0$ be any complex number with $\Re s_0 = \sigma$ and $0 < \Im s_0 < 2\pi / \log q$. Selecting the contour joining $\sigma+\epsilon$, $\sigma+\epsilon + i 2\pi / \log q$, $\sigma - \epsilon + i 2\pi / \log q$ and $\sigma-\epsilon$, we see
	$$ \log q \cdot M(s_0) = \int_0^{\frac{2\pi}{\log q}} \frac{M(\sigma+\epsilon+i\tau)}{\sigma+\epsilon+i\tau - s_0} \frac{\log q d\tau}{2\pi} - \int_0^{\frac{2\pi}{\log q}} \frac{M(\sigma-\epsilon+i\tau)}{\sigma-\epsilon+i\tau - s_0} \frac{\log q d\tau}{2\pi}. $$
	Using H\"older inequality we deduce
\begin{align*}
	\log q \cdot \norm[M(s_0)] &\leq \frac{1}{\epsilon} H_l^{\sigma+\epsilon}(M) \cdot \left( \int_0^{\frac{2\pi}{\log q}} \frac{\log q d\tau}{2\pi} \right)^{\frac{l-1}{l}} + \frac{1}{\epsilon} H_l^{\sigma-\epsilon}(M) \cdot \left( \int_0^{\frac{2\pi}{\log q}} \frac{\log q d\tau}{2\pi} \right)^{\frac{l-1}{l}} \\
	&= \frac{1}{\epsilon} H_l^{\sigma+\epsilon}(M) + \frac{1}{\epsilon} H_l^{\sigma-\epsilon}(M).
\end{align*}
	We conclude by taking $\sup$ with respect to $\Re s_0 = \sigma$.
\end{proof}
\begin{proposition}
	The two maps
	$$ \fsB_c(\ag{Z};\varpi) \to \fsH_c(\ag{C};q), f \mapsto \Mellin{f}(s) = \sum_{n \in \ag{Z}} f(\varpi^n) q^{-ns}, \text{ for } \Re s > c; $$
	$$ \fsH_c(\ag{C};q) \to \fsB_c(\ag{Z};\varpi), M \mapsto f_M(\varpi^n) = \int_0^{\frac{2\pi}{\log q}} M(\sigma+i\tau) q^{n(\sigma + i \tau)} \frac{\log q d\tau}{2\pi}, \forall \sigma > c $$
are continuous with respect to the above topologies defined by semi-norms.
\label{BHEquiv}
\end{proposition}
\begin{proof}
	The continuity follows from
	$$ H_{\infty}^{\sigma}(\Mellin{f}) \leq \sum_{n \in \ag{Z}} \norm[f(\varpi^n)] q^{-n\sigma} = B_1^{\sigma}(f); $$
	$$ B_{\infty}^{\sigma}(f_M) \leq \sup_{n \in \ag{Z}} \int_0^{\frac{2\pi}{\log q}} \norm[M(\sigma+i\tau) q^{n i \tau}] \frac{\log q d\tau}{2\pi} = H_1^{\sigma}(M). $$
\end{proof}
	Note that the abstract part of Definition \ref{FourF^1} still makes sense in the current case, i.e., for any function $f: \F \to \ag{C}$ and $\xi \in \widehat{\F^1}$ we can define
	$$ f_{\xi}(\varpi^n) = \int_{\F^1} f(\varpi^n u) \xi(u) du, n \in \ag{Z}. $$
\begin{proposition}
	For any $f \in \Sch(\F)$, $f_{\xi} \neq 0$ only for $\xi$ satisfying $\cond(\xi) \leq m(f)$, hence for only finitely many $\xi$. We have $f_{\xi} \in B_0^0(\ag{Z};\varpi)$ and
	$$ B_{\infty}^{\sigma}(f_{\xi}) \leq \Sch_{\infty}^{\sigma}(f), \forall \sigma > 0. $$
\label{SchToB}
\end{proposition}
\begin{proof}
	Obvious.
\end{proof}
\begin{lemma}
	For any $\Phi \in \Sch(\F^2)$, $s \in \ag{C}$, $\sigma > \max(0,-2\Re s)$ and $\epsilon > 0$ with $\sigma-\epsilon, \sigma+2\Re s - \epsilon >0$, there is $N=N(\epsilon, \sigma, \sigma+2\Re s) > 0$ such that with implied constant depending only on $\epsilon$
	$$ \extnorm{  \int_{\F^{\times}} \Phi(t,\frac{y}{t}) \omega^{-1}\xi^2(t)\norm[t]_{\F}^{2s} d^{\times}t } \ll_{\epsilon} q^{-ND(\Phi)+m(\Phi)} \Norm[\Phi]_{\infty} \cdot \norm[y]_{\F}^{-\sigma} 1_{y \in \vp^{2D(\Phi)}}. $$
\end{lemma}
\begin{proof}
	Writing
	$$ f(y) = \int_{\F^{\times}} \Phi(t,\frac{y}{t}) \omega^{-1}\xi^2(t)\norm[t]_{\F}^{2s} d^{\times}t, $$
we have for any $\xi_1 \in \widehat{\F^1}$ and $s_1 \in \ag{C}$ with $\Re s_1 > \max(0,-2\Re s)$
\begin{align*}
	\Mellin{f_{\xi_1}}(s_1) &= \int_{\F^{\times} \times \F^{\times}} \Phi(t,y) \omega^{-1}\xi^2\xi_1(t) \norm[t]_{\F}^{s_1+2s} \xi_1(y) \norm[y]_{\F}^{s_1} d^{\times}t d^{\times}y \\
	&= \Mellin{\Phi_{\omega^{-1}\xi^2\xi_1, \xi_1}}(s_1+2s+i\mu(\omega^{-1}\xi^2), s_1),
\end{align*}
where the second Mellin transform is the natural two dimensional one. By Proposition \ref{SchToB}, $\Phi_{\omega^{-1}\xi^2\xi_1, \xi_1} \neq 0$ only if $\cond(\omega^{-1}\xi^2\xi_1), \cond(\xi_1) \leq m(\Phi)$. In particular, the number of such $\xi_1$ is bounded by $q^{m(\Phi)}$. From the Mellin inversion for $\sigma > \max(0,-2\Re s)$
	$$ f(\varpi^n u) = \sum_{\xi_1} f_{\xi_1}(\varpi^n) \xi_1(u)^{-1} = \sum_{\xi_1} \xi_1(u)^{-1} \int_0^{\frac{2\pi}{\log q}} \Mellin{f_{\xi_1}}(\sigma+i\tau) q^{n(\sigma+i\tau)} \frac{\log q d\tau}{2\pi} $$
we can successively apply Proposition \ref{HEquiv}, \ref{BHEquiv}, \ref{BEquiv}, \ref{SchToB} and \ref{SchEquiv} to get
\begin{align*}
	\norm[y]_{\F}^{\sigma} \norm[f(y)] &\leq \sum_{\xi_1} H_1^{\sigma}(\Mellin{f_{\xi_1}}) \leq \sum_{\xi_1} H_{\infty}^{\sigma}(\Mellin{f_{\xi_1}}) \\
	&\leq \sum_{\xi_1} H_{\infty}^{\sigma+2\Re s, \sigma}(\Mellin{\Phi_{\omega^{-1}\xi^2\xi_1, \xi_1}}) \leq \sum_{\xi_1}  B_1^{\sigma+2\Re s, \sigma}(\Phi_{\omega^{-1}\xi^2\xi_1, \xi_1})  \\
	&\ll_{\epsilon} \sum_{\xi_1}  \left( B_{\infty}^{\sigma+2\Re s+\epsilon, \sigma+\epsilon} +B_{\infty}^{\sigma+2\Re s+\epsilon, \sigma-\epsilon} + B_{\infty}^{\sigma+2\Re s-\epsilon, \sigma+\epsilon} +B_{\infty}^{\sigma+2\Re s-\epsilon, \sigma-\epsilon} \right)(\Phi_{\omega^{-1}\xi^2\xi_1, \xi_1}) \\
	&\leq \left( \Sch_{\infty}^{\sigma+2\Re s+\epsilon, \sigma+\epsilon} + \Sch_{\infty}^{\sigma+2\Re s+\epsilon, \sigma-\epsilon} + \Sch_{\infty}^{\sigma+2\Re s-\epsilon, \sigma+\epsilon} + \Sch_{\infty}^{\sigma+2\Re s-\epsilon, \sigma-\epsilon} \right)(\Phi) \sum_{\xi_1}1 \\
	&\leq q^{-ND(\Phi)+m(\Phi)} \Norm[\Phi]_{\infty}.
\end{align*}
	Finally, it is obvious that $f(y) \neq 0$ implies the existence of some $t \in \F^{\times}$ such that $(t,y/t)$ lies in the support of $\Phi$ hence in $\vp^{D(\Phi)} \times \vp^{D(\Phi)}$, thus $y \in \vp^{2D(\Phi)}$.
\end{proof}
\begin{proposition}
	For any $\Phi \in \Sch(\F^2)$, let
	$$ D = \min(D(\Phi), -\cond(\psi)-\delta(\Phi)); \delta = \max(\delta(\Phi), \cond(\psi)-D(\Phi)). $$
	Then for any $\sigma > \norm[\Re s]$ and $\epsilon > 0$ with $\sigma - \epsilon > \norm[\Re s]$, there is $N=N(\epsilon, \sigma+\Re s, \sigma - \Re s)>0$ continuous in $\epsilon, \sigma\pm \Re s$ such that
	$$ \norm[W_{\Phi}(s,\xi,\omega\xi^{-1}; a(y)\kappa)] \ll_{\epsilon} q^{-(N+1)D+2\delta} \Norm[\Phi]_2 \cdot \norm[y]_{\F}^{\frac{1}{2}-\sigma} 1_{y \in \vp^{2D}}. $$
	Moreover, at an unramified place with $\cond(\psi)=\cond(\xi)=\cond(\omega\xi^{-1})=0$ and $\Phi=1_{\vo \times \vo}$ we have for any $\epsilon > 0$
	$$ \norm[W_{\Phi}(s,\xi,\omega\xi^{-1}; a(y)\kappa)] \leq \frac{2}{\epsilon \log q} \cdot \left( \sup_{x>0} \frac{x}{e^x} \right) \cdot \norm[y]_{\F}^{\frac{1}{2} - \norm[\Re s] - \epsilon} 1_{y \in \vp} + 1_{y \in \vo^{\times}}. $$
\label{LocWhiBdNA}
\end{proposition}
\begin{proof}
	The first part is a direct consequence of the previous lemma and the following inequalities deduced from Proposition \ref{DdmRel} and \ref{SchEquiv}:
	$$ D(\Four[2]{\rpR(\kappa).\Phi}) \geq \min(D(\Phi), -\cond(\psi)-\delta(\Phi)) = D; $$
	$$ m(\Four[2]{\rpR(\kappa).\Phi}) \leq \delta(\Four[2]{\rpR(\kappa).\Phi}) - D(\Four[2]{\rpR(\kappa).\Phi}) \leq \delta - D; $$
	$$ \extNorm{\Four[2]{\rpR(\kappa).\Phi}}_{\infty} \leq q^{\delta(\Four[2]{\rpR(\kappa).\Phi})} \extNorm{\Four[2]{\rpR(\kappa).\Phi}}_2 \leq q^{\delta} \Norm[\Phi]_2. $$
	The ``moreover'' part follows from a direct computation (or \cite[Theorem 4.6.5]{Bu98})
	$$ W_{\Phi}(s,\xi,\omega\xi^{-1};a(\varpi^n)) = q^{-\frac{n}{2}} \frac{\alpha^{n+1}-\beta^{n+1}}{\alpha - \beta} 1_{n \geq 0}, \text{ with } \alpha = \xi(\varpi)q^{-s}, \beta=\omega\xi^{-1}(\varpi) q^s, $$
which implies for $n \geq 1$
	$$ \norm[W_{\Phi}(s,\xi,\omega\xi^{-1};a(\varpi^n))] \leq (n+1) \norm[\varpi^n]_{\F}^{\frac{1}{2}-\norm[\Re s]} = \frac{n+1}{\epsilon n \log q} \cdot \frac{\epsilon n \log q}{q^{\epsilon n}} \cdot \norm[\varpi^n]_{\F}^{\frac{1}{2}-\norm[\Re s]-\epsilon} $$
is bounded as in the statement.
\end{proof}

		\subsubsection{Global Bound}
		
	Writing $D_v=\min(D(\Phi_v), -\cond(\psi_v)-\delta(\Phi_v))$ for $v<\infty$, it follows from Proposition \ref{LocWhiBdA} and \ref{LocWhiBdNA} that for any $\epsilon > 0$ and $N \gg 1$
	$$ \extnorm{W_{\Phi}(s,\xi,\omega\xi^{-1};na(y)\kappa)} \ll_{\epsilon,N,\Phi} \prod_{v \mid \infty} \min(\norm[y_v]_v^{\frac{1}{2}-\norm[\Re s]-\epsilon}, \norm[y_v]_v^{-N}) \prod_{v < \infty} \norm[y_v]_v^{\frac{1}{2}-\norm[\Re s]-\epsilon} 1_{y_v \in \vp_v^{2D_v}}. $$
	The bound is uniform when $\Re s$ lies in a fixed compact interval. Together with Lemma \ref{GTechEst}, we obtain
\begin{proposition}
	For any $\Phi \in \Sch(\ag{A}^2)$, $m \in \ag{N}$ and any $\sigma > \norm[\Re s]$, we have
	$$ \extnorm{\frac{\partial^m}{\partial s^m} \left( \eis(s,\xi,\omega\xi^{-1};\Phi)(g) - \eisCst(s,\xi,\omega\xi^{-1};\Phi)(g) \right)} \leq \sum_{\alpha \in \F^{\times}} \extnorm{\frac{\partial^m}{\partial s^m} W_{\Phi}(s,\xi,\omega\xi^{-1};a(\alpha)g)} \ll_{\sigma,\Phi} \Ht(g)^{-\frac{1}{2}-\sigma}, $$
which is of rapid decay with respect to $\Ht(g)$.
\label{GlobRDEisWhi}
\end{proposition}

	\subsection{Behavior of Constant Term}
	
	Consider $f \in V_{\xi,\omega\xi^{-1}}^{\infty}$ and take its Fourier expansion into $\gp{K}$-isotypic types
	$$ f = \sum_{\vec{n}} \hat{f}_{\vec{n}} \cdot e_{\vec{n}}(\xi,\omega\xi^{-1}) $$
for some $\hat{f}_{\vec{n}} \in \ag{C}$, where $e_{\vec{n}}(\xi,\omega\xi^{-1})$ is unitary with $\gp{K}$-type parametrized by $\vec{n}$ as in \cite[Section 3.5]{Wu5}. For $\Re s \gg 1$, we have
	$$ \eisCst(s,\xi,\omega\xi^{-1};f) = f(s,\xi,\omega\xi^{-1}) + \sum_{\vec{n}} \hat{f}_{\vec{n}} \cdot \mu(s,\xi,\omega\xi^{-1};\vec{n}) e_{\vec{n}}(-s,\omega\xi^{-1},\xi), $$
where $\mu(s,\xi,\omega\xi^{-1};\vec{n})$ are the explicit coefficients of the intertwining operator on the $\gp{K}$-type $\vec{n}$ part as in \textit{loc.cit.}
\begin{proposition}
	The possible poles of the collection of meromorphic functions $\{ \mu(s,\xi,\omega\xi^{-1};\vec{n}) : \vec{n} \}$ are
\begin{itemize}
	\item The pole of $\Lambda(1-2s,\omega\xi^{-2})$ at $s=(1+i\mu(\omega\xi^{-2}))/2$ with order at most $1$, when $\omega\xi^{-2}$ is trivial on $\ag{A}^{(1)}$ and $\vec{n}=\vec{0}$. We call this pole the \emph{spherical pole}.
	\item The (both trivial and non-trivial) zeros of $L(1+2s,\omega^{-1}\xi^2)$ with order at most that of the zero.
\end{itemize}
\end{proposition}
\begin{proof}
	This can be seen either from our explicit computation of $\mu(s,\xi,\omega\xi^{-1};\vec{n})$ in \cite{Wu5}, or from (\ref{RelSec}) and
	$$ \Intw f_{\Phi}(s,\xi,\omega\xi^{-1}) = f_{\widehat{\Phi}}(-s,\omega\xi^{-1},\xi), $$
which is meromorphic with a simple pole at $s=(1+i\mu(\omega\xi^{-2}))/2$ (the other pole is cancelled by that of $f_{\Phi}(s,\xi,\omega\xi^{-1})$) by Tate's theory.
\end{proof}
\begin{lemma}
	Assume $\mu(s,\xi,\omega\xi^{-1};\vec{n}_0)$ has a pole at $s_0$ with order $n$ ($n=0$ if it is holomorphic at $s_0$) for some $\vec{n}_0$. Define
	$$ \Norm[\vec{n}] = \prod_v (\norm[n_v]+1) \text{ for } \vec{n} = (n_v)_v. $$
	Then as $s$ lying in any small compact neighborhood $K$ of $s_0$ where no other pole occur, we have
	$$ \extnorm{(s-s_0)^n \mu(s,\xi,\omega\xi^{-1};\vec{n})} \ll \Norm[\vec{n}]^N $$
for some $N$ and the implied constant depending only on $K$, $\xi$ and $\omega\xi^{-1}$ (i.e., independent of $\vec{n}$).
\end{lemma}
\begin{proof}
	This follows from the explicit computation of $\mu(s,\xi,\omega\xi^{-1};\vec{n})$ in \cite{Wu5} together with the following obvious bound
	$$ \prod_{k=1}^n \extnorm{\frac{k-s}{k+s}} \leq \prod_k \left( 1+\frac{2\norm[s]}{\norm[k+s]} \right) \leq \exp \left\{ \sum_k \frac{2\norm[s]}{\norm[k+s]} \right\} \ll (n+1)^N. $$
\end{proof}
\begin{proposition}
	Under the condition of the lemma we have
	$$ (s-s_0)^n \eisCst(s,\xi,\omega\xi^{-1};f) = (s-s_0)^n f(s,\xi,\omega\xi^{-1}) + \sum_{\vec{n}} \hat{f}_{\vec{n}} \cdot (s-s_0)^n \mu(s,\xi,\omega\xi^{-1};\vec{n}) e_{\vec{n}}(-s,\omega\xi^{-1},\xi). $$
	Consequently, for any $m \in \ag{N}$, $y \in \ag{A}^{\times}$ and $\kappa \in \gp{K}$ we can write
\begin{align*}
	&\quad \frac{d^m}{d s^m} \mid_{s=s_0} (s-s_0)^n \eisCst(s,\xi,\omega\xi^{-1};f)(a(y)\kappa) \\
	&= \frac{m!}{(m-n)!} \norm[y]_{\ag{A}}^{\frac{1}{2}+s_0} \xi(y) (\log \norm[y]_{\ag{A}})^{m-n} f(\kappa) + \sum_{k=0}^m \norm[y]_{\ag{A}}^{\frac{1}{2}-s_0} \omega\xi^{-1}(y) (\log \norm[y]_{\ag{A}})^k f_k(\kappa),
\end{align*}
where $(\log \norm[y]_{\ag{A}})^{m-n}$ is understood as $0$ if $m<n$ and where the maps
	$$ V_{\xi,\omega\xi^{-1}}^{\infty} \to V_{\omega\xi^{-1},\xi}^{\infty}, f \to f_k \text{ for } 0 \leq k \leq m $$
are $\gp{K}$-maps and continuous with respect to the Sobolev norms on $\gp{K}_{\infty}$.
\label{GlobRDEisCst}
\end{proposition}
\begin{proof}
	The lemma shows that $(s-s_0)^n \mu(s,\xi,\omega\xi^{-1};\vec{n})$ is polynomially increasing in $\vec{n}$. But $\hat{f}_{\vec{n}}$ is rapidly decreasing in $\vec{n}$ by smoothness of $f$. The sum over $\vec{n}$ is thus absolutely and uniformly convergent for $s$ lying in $K$, hence defines an analytic function in $s$. By uniqueness of analytic continuation, we get the first equation. The rest follows from the polynomial increase of $(s-s_0)^n \mu(s,\xi,\omega\xi^{-1};\vec{n})$ and Cauchy's integral formula for derivatives.
\end{proof}
\begin{remark}
	If $s_0$ is the spherical pole, then $n=1$ and for $m=0$ we have $f_k=0$ unless $k=0$. Then
	$$ f_0(\kappa) = \lim_{s \to s_0} \frac{(s-s_0)\Lambda_{\F}(2(s_0-s)}{\Lambda_{\F}(2-2(s_0-s)))} \int_{\gp{K}} f(\kappa) d\kappa = - \frac{\Lambda_{\F}^*(0)}{\Lambda_{\F}(2)} \int_{\gp{K}} f(\kappa) d\kappa. $$
\end{remark}

	\subsection{A Convergence Lemma}
	
	Let $r,N \in \ag{N}, r,N \geq 2$. It is well-known that there is an exact sequence
	$$ 1 \to \Gamma(N) \to \SL_r(\ag{Z}) \to \SL_r(\ag{Z}/N\ag{Z}) \to 1, $$
	where $\Gamma(N)$ is called the principal congruence subgroup of $\SL_r(\ag{Z})$ modulo $N$. The pre-image of the upper resp. lower triangular subgroup of $\SL_r(\ag{Z}/N\ag{Z})$ in $\SL_r(\ag{Z})$ is denoted by $\Gamma_0(N)$ resp. $\Gamma_0^-(N)$, which includes $\Gamma(N)$ as a normal subgroup. For $\vec{\alpha} \in \ag{Z}^{r-1}$ realized as a column vector, we define matrices, with $I_k$ denoting the identity matrix of rank $k$
	$$ n_r^+(\vec{\alpha}) = \begin{pmatrix}
	1 & \vec{\alpha}^T \\ \vec{0} & I_{r-1}
	\end{pmatrix}, n_r^-(\vec{\alpha}) = \begin{pmatrix}
	1 & \vec{0}^T \\ \vec{\alpha} & I_{r-1}
	\end{pmatrix}. $$
\begin{lemma}
	Any coset of $\Gamma_0(N) \backslash \SL_r(\ag{Z})$ resp. $\Gamma_0^-(N) \backslash \SL_r(\ag{Z})$ has a representative of the form $N_- N_+$ resp. $N_+N_-$, where $N_-$ resp. $N_+$ is lower resp. upper unipotent with off-diagonal entries lying in $[-\frac{N}{2}, \frac{N}{2}] \cap \ag{Z}$.
\label{CongGpLemma}
\end{lemma}
\begin{proof}
	We treat the case for $\Gamma_0(N)$. Take any $A \in \SL_r(\ag{Z})$. Let the first column of $A$ be $(a_1,\cdots,a_r)^T \in \ag{Z}^r$. Then we have $\lcd(a_1,\cdots,a_r)=1$, which implies the existence of $u_i \in \ag{Z}$ such that $\Sigma_{i=1}^r u_i a_i = 1$. In particular, we have $\lcd(u_1,a_2,\cdots,a_r)=1$. For any $k_j \in \ag{Z}, 2 \leq j \leq r$, the substitution
	$$ u_1' = u_1+\sum_{j=2}^r k_j a_j; u_j' = u_j-k_j, 2 \leq j \leq r $$
	still gives $\Sigma_{i=1}^r u_i' a_i = 1$. By an iterative application of Dirichlet's theorem on primes in arithmetic progression, we can choose $k_j$'s such that $u_1'$ is a prime number as large as we want. In particular, we can make $\lcd(u_1',N)=1$. Since $\lcd(u_1',\cdots,u_r')=1$, at least one of $u_j', 2 \leq j \leq r$, say $u_2'$, is coprime with $u_1'$. Hence $\lcd(u_1',Nu_2')=1$ and we can find $v_1,v_2 \in \ag{Z}$ such that $u_1'v_2-Nu_2'v_1=1$. The matrix
	$$ B = \begin{pmatrix} u_1' & u_2' & u_3' \cdots \\ Nv_1 & v_2 & 0 \cdots \\ \vec{0} & \vec{0} & I_{r-2} \end{pmatrix} \in \Gamma_0(N) $$
	and gives, for some $A'=\tilde{A}-\vec{\alpha} \vec{\beta_1}^T \in \SL_{r-1}(\ag{Z})$, that
	$$ BA = \begin{pmatrix} 1 & \vec{\beta_1}^T \\ \vec{\alpha} & \tilde{A} \end{pmatrix} \Rightarrow A \in \Gamma_0(N) \begin{pmatrix} 1 &  \\ \vec{\alpha} & A' \end{pmatrix} n_r^+(\vec{\beta_1}). $$
	Repeating the process on $A'$ or making an induction on $r$ we then find successively $\vec{\beta_k} \in \ag{Z}^{r-k}$ such that for some lower unipotent $N_-^1$
	$$ A \in \Gamma_0(N) N_-^1 N_+^1, N_+^1=\prod_{k=0}^{r-2} \begin{pmatrix} I_k & 0 \\ 0 & n_{r-k}^+(\vec{\beta_{k+1}}) \end{pmatrix}. $$
	Taking $N_-$ resp. $N_+$ with off-diagonal entries lying in $[-\frac{N}{2}, \frac{N}{2}] \cap \ag{Z}$ such that $N_-^1 \equiv N_- \pmod{N}$ resp. $N_+^1 \equiv N_+ \pmod{N}$, we then find $N_-^1 N_-^{-1}, N_+^1 N_+^{-1} \in \Gamma(N)$ and conclude by its normality.
\end{proof}
	Let $[\F:\ag{Q}]=r=r_1+2r_2$ where $r_1$ resp. $2r_2$ is the number of embeddings of $\F$ into real resp. complex numbers. Recall that we have a canonical map by choosing one complex embedding $\sigma_{r_1+j}, 1 \leq j \leq r_2$ in a pair conjugate to each other by complex conjugation
	$$ \sigma: \F \to \ag{R}^{r_1} \times \ag{C}^{r_2} \simeq \ag{R}^r, $$
	$$ \sigma(x) = (\sigma_1(x), \cdots, \sigma_{r_1+r_2}(x)) = (\sigma_1(x), \cdots, \sigma_{r_1}(x), \Re \sigma_{r_1+1}(x), \Im \sigma_{r_1+1}(x), \cdots, \Re \sigma_{r_1+r_2}(x), \Im \sigma_{r_1+r_2}(x)). $$
	For every fractional ideal $\idlJ$, $\sigma(\idlJ)$ is then a $\ag{Z}$-lattice of $\ag{R}^r$. For $c \gg 1$, we define a functions $f_c$ on $\ag{R}^{r_1} \times \ag{C}^{r_2}$ by
	$$ f_c(\vec{x}) = \prod_{i=1}^{r_1} \min(1, \norm[x_i]^{-c}) \prod_{j=1}^{r_2} \min(1, \norm[x_{r_1+j}]^{-2c}), \vec{x}=(x_i)_{1\leq i \leq r_1+r_2}. $$
\begin{lemma}
	Let $\idlJ \subset \vo$ be an integral ideal. We have the following two estimations for $t>0$.
\begin{itemize}
	\item[(1)] $ \sum_{\alpha \in \idlJ^{-1} - \{0\}} f_c(t\sigma(\alpha)) \ll_{\F,c} \norm[\vo/\idlJ]^{3c} t^{-c} $ if $c>r$.
	\item[(2)] $ \sum_{\alpha \in \idlJ^{-1}} f_c(t\sigma(\alpha)) \ll_{\F,c} t^{-r} \norm[\vo/\idlJ]^{-1} \left( 1+t\frac{\norm[\vo/\idlJ]^2}{\sqrt{r}} \right)^{rc} $ if $c>1$.
\end{itemize}
\label{GTechEstClassical}
\end{lemma}
\begin{proof}
	(1) On $\ag{R}^{r_1} \times \ag{C}^{r_2}$ we have a usual norm $\Norm_2$ given by
	$$ \Norm[\vec{x}]_2 = \sqrt{\sum_{i=1}^{r_1+r_2} \norm[x_i]^2}, \vec{x}=(x_i)_{1 \leq i \leq r_1+r_2}. $$
	We then easily see, essentially by comparing $f_c$ with the infinity norm, that
	$$ \sup_{\vec{x}} f_c(\vec{x}) \Norm[\vec{x}]_2^c \leq (r_1+r_2)^{\frac{c}{2}}. $$
	If $\alpha_i \in \idlJ^{-1}$ is an integral basis such that $\idlJ^{-1}=\Sigma_{i=1}^r \ag{Z} \alpha_i$, then $\sigma(\alpha_i)$ is a basis of the lattice $\sigma(\idlJ^{-1})$. We can define another norm $\Norm_{\idlJ^{-1}}$ by
	$$ \Norm[\vec{x}]_{\idlJ^{-1}}=\sqrt{\sum_{i=1}^r n_i^2}, \vec{x} = \sum_{i=1}^r n_i \sigma(\alpha_i) \in \ag{R}^r. $$
	Or equivalently, if we write $A_{\idlJ^{-1}}=(\sigma(\alpha_1), \cdots, \sigma(\alpha_r)) \in \GL_r(\ag{R})$, we have
	$$ \Norm[\vec{x}]_{\idlJ^{-1}} = \Norm[A_{\idlJ^{-1}}^{-1}\vec{x}]_2. $$
	Fix a basis $\vec{e_i}$ of $\sigma(\vo)$ and write $A_{\vo}=(\vec{e_1},\cdots,\vec{e_r})$. By elementary divisor theorem, there are $\gamma,\gamma_{\idlJ} \in \SL_r(\ag{Z})$ and some $d_i \in \ag{Z}-\{0\}, d_i \mid d_{i+1}, \norm[d_1 \cdots d_r] = \norm[\vo/\idlJ] = \norm[\idlJ^{-1} / \vo]$, such that
	$$ A_{\idlJ^{-1}}^{-1} = \gamma \diag(d_1,\cdots,d_r) \gamma_{\idlJ} A_{\vo}^{-1}. $$
	Changing $\gamma$ is equivalent to changing the choice of basis $\alpha_i$ for $\idlJ^{-1}$. Applying Lemma \ref{CongGpLemma}, we can find $\gamma$ such that
	$$ \diag(d_1,\cdots,d_r)^{-1} \gamma \diag(d_1,\cdots,d_r) \gamma_{\idlJ} = N_+ N_- $$
	where $N_+$ resp. $N_-$ is upper resp. lower unipotent with entries lying in $[-\frac{d_r}{2d_1}, \frac{d_r}{2d_1}]$. We thus get a bound of the operator norm
	$$ \Norm[A_{\idlJ^{-1}}^{-1}]_2 = \Norm[\diag(d_1,\cdots,d_r) N_+N_- A_{\vo}^{-1}]_2 \leq \norm[d_r] \left( r+\frac{r(r-1)d_r^2}{8d_1^2} \right) \Norm[A_{\vo}^{-1}]_2. $$
	We finally estimate and conclude by
\begin{align*}
	\sum_{\alpha \in \idlJ^{-1} - \{0\}} f_c(t\sigma(\alpha)) &= \sum_{\alpha \in \idlJ^{-1} - \{0\}} \Norm[t\sigma(\alpha)]_{\idlJ^{-1}}^{-c} \frac{\Norm[t\sigma(\alpha)]_{\idlJ^{-1}}^c}{\Norm[t\sigma(\alpha)]_2^c} f_c(t\sigma(\alpha)) \Norm[t\sigma(\alpha)]_2^c \\
	&\leq t^{-c} \norm[d_r]^c \left( r+\frac{r(r-1)d_r^2}{8d_1^2} \right)^c \Norm[A_{\vo}^{-1}]_2^c (r_1+r_2)^{\frac{c}{2}} \sum_{\vec{n} \in \ag{Z}^r-\{0\}} \Norm[\vec{n}]_2^{-c}.
\end{align*}

\noindent (2) If $\latL$ is a lattice in $\ag{R}^r$ with a basis given by the column vectors in a matrix $A_{\latL} \in \GL_r(\ag{R})$, we define its diameter $d(\latL)$ associated to this basis as the diameter of the fundamental parallelogram spanned by this basis, i.e.,
	$$ d(\latL) = \max_{-1 \leq \lambda_i \leq 1, 1\leq i \leq r} \Norm[A_{\latL} \vec{\lambda}]_2. $$
	Thus we have obviously $d(\latL) \leq \sqrt{r} \Norm[A_{\latL}]_2$. The same argument in the last part of (1) gives, for some choice of basis of $\idlJ^{-1}$, that
	$$ d(\sigma(\idlJ^{-1})) \leq \sqrt{r} \Norm[A_{\vo}]_2 \norm[d_1]^{-1} \left( r+\frac{r(r-1)d_r^2}{8d_1^2} \right). $$
	It is also easy to see that for $\vec{x},\vec{y}$ in the same translate of a parallelogram $\vec{x_0}+\latPlg$ of $t\sigma(\idlJ^{-1})$, we have
	$$ \frac{f_c(\vec{x})}{f_c(\vec{y})} \leq 2^{2r_2c} \left( \frac{r_1+r_2+\sqrt{r}d(t\sigma(\idlJ^{-1}))}{r} \right)^{rc}. $$
	We thus conclude by
\begin{align*}
	\Vol(\ag{R}^r/t\sigma(\idlJ^{-1})) \sum_{\alpha \in \idlJ^{-1}} f_c(t\sigma(\alpha)) &\leq 2^{2r_2c} \left( \frac{r_1+r_2+t\sqrt{r}d(\sigma(\idlJ^{-1}))}{r} \right)^{rc} \int_{\ag{R}^r} f_c(\vec{x}) d\vec{x} \\
	&\ll_{\F,c} \left( 1+t\frac{\norm[\vo/\idlJ]^2}{\sqrt{r}} \right)^{rc}.
\end{align*}
\end{proof}
	We shall write, denoting by $\vp_v$ the prime ideal corresponding to $v<\infty$, $v(\idlJ) \in \ag{N}$ such that
	$$ \idlJ = \prod_{v<\infty} \vp_v^{v(\idlJ)}. $$
	A variant of the above estimations in the adelic language is the following lemma.
\begin{lemma}
	Let $\idlJ$ be an integral ideal. Let $c_1,c_2 \in \ag{R}, c_2-c_1 > r$, we have
\begin{align*}
	&\quad \sum_{\alpha \in \F^{\times}} \prod_{v\mid \infty} \min(|\alpha y_v|_v^{-c_1},|\alpha y_v|_v^{-c_2}) \prod_{v<\infty} |\alpha y_v|_v^{-c_1} 1_{v(\alpha y_v) \geq -\cond(\psi_v)-v(\idlJ)} \\
	&\ll_{\F,c_2-c_1} \min \left(|y|_{\ag{A}}^{-c_1-1} \norm[\vo/\idlJ]^{-1} \left( 1+|y|_{\ag{A}}^{\frac{1}{r}} \frac{\norm[\vo/\idlJ]^2}{\sqrt{r}} \right)^{r(c_2-c_1)}, |y|_{\ag{A}}^{-c_1-\frac{c_2-c_1}{r}} \norm[\vo/\idlJ]^{3(c_2-c_1)} \right)
\end{align*}
	where $y=(y_v)_v \in \ag{A}^{\times}$, $r=[\F:\ag{Q}]$.
\label{GTechEst}
\end{lemma}
\begin{proof}
	The following exact sequence is split with a splitting section $s: \ag{R}_+ \to \ag{A}^{\times}$
	$$ 1 \to \ag{A}^{(1)} \to \ag{A}^{\times} \xrightarrow{|\cdot|_{\ag{A}}} \ag{R}_+ \to 1 $$
	such that the image of $s$ is contained in $\ag{A}_{\infty}^{\times} = \prod_{v \mid \infty} \F_v^{\times}$. For any $t \in \ag{R}_+$, we write $t^+=s(t)$ with $t_v^+=t^{1/r} \in \ag{R}_+ \subset \F_v$, such that
	$$ |t^+|_{\ag{A}} = \prod_{v \mid \infty} |t_v^+|_v = t. $$
	We may apply the compactness of $\F^{\times} \backslash \ag{A}^{(1)}$, or proceed alternatively in the following more classical way. Let $\gCl(\F)$ be the class group of $\F$ and choose an integral ideal $\tau$ in each class $[\tau] \in \gCl(\F)$. Let $\delta_{\tau} \in \ag{A}_{\fin}^{\times}$ be a representative of $\tau$ in the group of ideles. Since $\gCl(\F) \simeq \F^{\times} \backslash \ag{A}_{\fin}^{\times} / \widehat{\vo}^{\times}$, there is a unique $\tau$, and some $\beta=\beta_0 u \in \F^{\times}$ with freely chosen $u \in \vo^{\times}$ such that $ y_{\fin} \in \beta \delta_{\tau} \widehat{\vo}^{\times}$. Let $t=\norm[y]_{\ag{A}}$. We can write $y=y_{\infty} y_{\fin} = t^+ y_{\infty}' y_{\fin}$, and find
	$$ \prod_{v \mid \infty} \norm[\beta^{-1}y_v(t_v^+)^{-1}]_v = \prod_{v \mid \infty} \norm[u^{-1}\beta_0^{-1}y_v' ]_v = t^{-1}\norm[\beta^{-1}y]_{\ag{A}} \norm[\delta_{\tau}]_{\ag{A}}^{-1} = \norm[\vo/\tau]. $$
	By Dirichlet's unit theorem, we can adjust $u \in \vo^{\times}$ so that each $\norm[u^{-1}\beta_0^{-1}y_v' ]_v$ remains in a compact set in $\ag{R}_+$ depending only on $\F$. We thus find
\begin{align*}
	&\quad \sum_{\alpha \in \F^{\times}} \prod_{v\mid \infty} \min(|\alpha y_v|_v^{-c_1},|\alpha y_v|_v^{-c_2}) \prod_{v<\infty} |\alpha y_v|_v^{-c_1} 1_{v(\alpha y_v) \geq -\cond(\psi_v)-v(\idlJ)} \\
	&= \sum_{\alpha \in \F^{\times}} \prod_{v\mid \infty} \min(|\alpha \beta^{-1} y_v|_v^{-c_1},|\alpha \beta^{-1} y_v|_v^{-c_2}) \prod_{v<\infty} |\alpha \delta_{\tau}|_v^{-c_1} 1_{v(\alpha) \geq -\cond(\psi_v)-v(\idlJ \tau)} \\
	&= t^{-c_1} \sum_{\alpha \in (\Dif_{\F} \idlJ \tau)^{-1} - \{0\}} \prod_{v\mid \infty} \min(1,\norm[\alpha t_v^+]_v^{-(c_2-c_1)} \cdot \norm[u^{-1}\beta_0^{-1}y_v']_v^{c_1-c_2}) \\
	&\ll_{\F,c_2-c_1} t^{-c_1} \sum_{\alpha \in (\Dif_{\F} \idlJ \tau)^{-1} - \{0\}} f_{c_2-c_1}(t^{\frac{1}{r}} \sigma(\alpha)).
\end{align*}
	We then apply Lemma \ref{GTechEstClassical} to conclude.
\end{proof}

\section*{Acknowledgement}

	The preparation of the paper scattered during the stays of the author's in FIM at ETHZ, YMSC at Tsinghua University, Alfr\'ed Renyi Institute in Hungary supported by the MTA R\'enyi Int\'ezet Lend\"ulet Automorphic Research Group and TAN at EPFL. The author would like to thank all these institutes for their hospitality. The informal seminar of automorphic forms at the Alfr\'ed Renyi Institute is specially important for the inspiration. The author would like to thank all the participants of the seminar. Last but not least, the author would like to thank the referee for careful reading and useful suggestions which help improve a lot the clarity of the presentation.

\bibliographystyle{acm}

\bibliography{mathbib}

\address{\quad \\ Han WU \\ EPFL SB MATHGEOM TAN \\ MA C3 604 \\ Station 8 \\ CH-1015, Lausanne \\ Switzerland \\ wuhan1121@yahoo.com}

\end{document}